\documentclass{article}
\usepackage[utf8]{inputenc}
\usepackage{color}

\usepackage[T1]{fontenc}
\usepackage[normalem]{ulem}
\usepackage[english]{babel}
\usepackage{verbatim}
\usepackage{graphicx}
\usepackage{enumerate}
\usepackage{amsmath,amssymb,amsfonts,amsthm,amscd,mathrsfs}
\usepackage{array}
\usepackage{amsmath,amssymb,graphicx,stmaryrd,enumerate,bbm,alltt}
\usepackage{dsfont}
\usepackage{comment}
\usepackage{mathtools}
\usepackage{bbm}

\usepackage[T1]{fontenc}
\usepackage{babel}
\usepackage[bookmarksopen, bookmarksnumbered]{hyperref}

\usepackage{url}
\usepackage{caption}

\DeclareMathOperator{\dist}{dist}
\DeclareMathOperator{\trace}{trace}

\newtheorem{theorem}{Theorem}[section]

\newtheorem{lemma}[theorem]{Lemma}
\newtheorem{cla}[theorem]{Claim}
\newtheorem{prop}[theorem]{Proposition}
\newtheorem{cor}[theorem]{Corollary}

\theoremstyle{definition}
\newtheorem{defn}[theorem]{Definition}
\newtheorem{rem}[theorem]{Remark}

\newtheorem{obs}[theorem]{Observation}

\newcommand{\Z}{\mathbb Z}
\newcommand{\R}{\mathbb R}
\newcommand{\Prob}{\mathbb P}

\title{Anderson-Bernoulli localization at large disorder on the 2D lattice}
\author{Linjun Li
\thanks{Department of Mathematics, University of Pennsylvania, Philadelphia, PA. e-mail:
 linjun@sas.upenn.edu} 
}
\date{}
\numberwithin{equation}{section}
\begin{document}

\maketitle
\begin{abstract}
We consider the Anderson model at large disorder on $\Z^2$ where the potential has a symmetric Bernoulli distribution. 
We prove that Anderson localization happens outside a small neighborhood of finitely many energies. These finitely many energies are Dirichlet eigenvalues of the minus Laplacian restricted on some finite subsets of $\Z^{2}$.
\end{abstract}

\tableofcontents

\section{Introduction}
\subsection{Main result}
Let $p \in (0,1)$ and $\Bar{V}>0$. Let $V:\Z^d \rightarrow \{0,\Bar{V}\}$ be a random function such that $\{V(a):a \in \Z^{d}\}$ is a family of independent Bernoulli random variables with $\Prob (V(a)=0)=p$ and $\Prob (V(a)=\Bar{V})=1-p$ for each $a\in \Z^{d}$. Let $\Delta$ denote the Laplacian  
\begin{equation}
\Delta u(a) =- 2d u(a)+ \sum_{b \in \Z^d, |a-b|=1}u(b)  ,\; \forall u: \Z^d\rightarrow \mathbb{R}, a \in \Z^d.
\end{equation}
Here and throughout the paper, $|a|=\|a\|_{\infty}$ for $a \in \Z^d$.
We study the spectra property of the random Hamiltonian operator 
\begin{equation}\label{eq:andersonmodel}
    H=-\Delta +V
\end{equation}
when $\Bar{V}$ is large enough.

This model is sometimes called the ``Anderson-Bernoulli model''. It is known that (see e.g. \cite{pastur1980spectral}),  almost surely, the spectrum of $H=-\Delta + V$ is 
\begin{equation}\label{eq:spectrum-interval}
    \sigma(H) = \left[0,4d\right]\cup \left[\Bar{V},\Bar{V}+4d\right]
\end{equation}
which is a union of two disjoint intervals when $\Bar{V}>4d$. Here and throughout the paper, we denote by $\sigma(A)$ the spectrum of a self-adjoint operator $A$. 
Our main theorem is the following

\begin{theorem}[Main theorem]\label{thm:1/2_main}
  Let $d=2$, $p=\frac{1}{2}$. There exist positive integer $n$ and energies $\lambda^{(1)},\lambda^{(2)},\cdots,\lambda^{(n)} \in \left[0,8 \right]$ such that the following holds.
  
  For each $\Bar{V}$ large enough, suppose $\widetilde{\lambda^{(i)}}=\Bar{V}+8-\lambda^{(i)}$ for $i=1,\cdots,n$. Let $$Y_{\Bar{V}} = \bigcup_{i=1}^{n} \left[ \lambda^{(i)}-\Bar{V}^{-\frac{1}{4}},\lambda^{(i)}+\Bar{V}^{-\frac{1}{4}} \right]$$
  and 
  $$\widetilde{Y_{\Bar{V}}}= \bigcup_{i=1}^{n} \left[ \widetilde{\lambda^{(i)}}-\Bar{V}^{-\frac{1}{4}},\widetilde{\lambda^{(i)}}+\Bar{V}^{-\frac{1}{4}} \right].
  $$
  Let $H$ be defined as in \eqref{eq:andersonmodel}.
  Then almost surely, for any $\lambda_{0} \in \sigma(H)\setminus (Y_{\Bar{V}}\cup \widetilde{Y_{\Bar{V}}})$ and $u:\Z^2 \rightarrow \mathbb{R}$, if $H u= \lambda_{0} u$ and
\begin{equation}\label{eq:polynomialbound}
    \inf_{m>0} \sup_{a \in \Z^2} (|a|+1)^{-m}|u(a)|<\infty,
\end{equation} 
then
\begin{equation}\label{eq:exponentialdecaybound}
    \inf_{c>0} \sup_{a \in \Z^2} \exp(c|a|)|u(a)|<\infty.
\end{equation}
\end{theorem}
\begin{rem}
The energies $\lambda^{(i)}$'s are defined in Definition \ref{def:J_eig} below and they do not depend on $\Bar{V}$. In fact, $\lambda^{(i)}$'s are Dirichlet eigenvalues of $-\Delta$ restricted on finite subsets of $\Z^{2}$ and $\widetilde{\lambda^{(i)}}$'s are simply images of $\lambda^{(i)}$'s under the mapping $x \mapsto \Bar{V}+8-x$. 
\end{rem}
\begin{rem}\label{rem:perc-threshold}
Our proof and conclusions in Theorem \ref{thm:1/2_main} extend to $1-p_{c}<p<p_{c}$ where $p_{c}>\frac{1}{2}$ is the site percolation threshold on $\Z^{2}$ (see Section \ref{sec:site-percolation}). $p\in (1-p_{c},p_{c})$ is an essential assumption for our method to prove Theorem \ref{thm:1/2_main} (see Section \ref{sec:wegner-outline} below) and it is an interesting question whether a similar result can be proved for $p\in (0,1-p_{c}]\cup[p_{c},1)$.

For simplicity, throughout this paper, we restrict ourselves to the case where $p=\frac{1}{2}$. 
\end{rem}

Denote $\widehat{\sigma(H)} = \sigma(H)\setminus (Y_{\Bar{V}}\cup \widetilde{Y_{\Bar{V}}})$. The result in Theorem \ref{thm:1/2_main} means any polynomially bounded solution of $H u= \lambda_{0} u$ decreases exponentially provided $\lambda_{0} \in \widehat{\sigma(H)}$. This is sometimes called ``Anderson localization'' (in $\widehat{\sigma(H)}$) and it implies that $H$ has pure point spectrum in $ \widehat{\sigma(H)}$, see e.g. \cite[Section 7]{kirsch2007invitation} by Kirsch.  In the literature of physics, Anderson localization was introduced by Anderson in his seminal paper \cite{anderson1958absence} in which
Anderson said,
\begin{multline}\label{eq:anderson-prediction}
    \textit{\parbox{.85\textwidth}{The theorem is that at sufficiently low densities, transport does not take place; the exact wave functions are localized in a small region of space.}}
\end{multline}
Here, the density refers to the \emph{density of states measure} (\emph{DOS measure}) which is defined as follows.
Given a real number $l \geq 1$ and a vertex $a\in \Z^{2}$, we let 
\begin{equation}\label{eq:def-cube-Q_l}
    Q_{l}(a)=\left\{a' \in \Z^{2}:|a-a'| \leq \frac{l-1}{2}\right\}
\end{equation}
be the box centered at $a$ and define its side length $\ell(Q_{l}(a))=2\lfloor \frac{l-1}{2} \rfloor$. 
We restrict operator $H$ to the box centered at the origin with side length $L$ and denote the (random) empirical distribution of the restricted operator's eigenvalues by $\mu_{L}$. It is known that, almost surely, when $L$ goes to infinity, $\mu_{L}$ converges weakly to some probability measure which is called the DOS measure (see e.g. \cite[Chapter 3]{aizenman2015random} by Aizenman and Warzel). Intuitively, the DOS measure in the interval $[E_{1},E_{2}]$ gives the ``number of states per unit volume'' with energy in $[E_{1},E_{2}]$. The smallness of DOS measure was mathematically verified for several cases, in particular for the following two cases,

\begin{enumerate}
    \item For any nontrivial distribution of $V$, the DOS measure is extremely small near the bottom of the spectrum. This is also called the ``Lifshitz tail phenomenon''. See e.g. \cite[Chapter 4.4]{aizenman2015random} and also \cite[Section 6.2]{kirsch2007invitation}. 
    \item Suppose $V=\delta V_{0}$ where $V_{0}$ has a uniformly H{\"o}lder continuous distribution (see \cite[Definition 4.5]{aizenman2015random}). The DOS measure of any finite interval with a given length becomes uniformly small when the disorder strength $\delta$ increases to infinity. See e.g. \cite[Theorem 4.6]{aizenman2015random}.
\end{enumerate}
In both cases, according to \eqref{eq:anderson-prediction}, one expects Anderson localization to happen in the corresponding spectrum range, namely, near the bottom in the first case and throughout the whole spectrum in the second case. Both cases have been studied extensively and Anderson localization was proved for several distributions of $V$. 

For $V$ with a H{\"o}lder continuous distribution, Anderson localization was proved in both cases in any dimension, namely, near the bottom of the spectrum or throughout the spectrum when the disorder strength is large enough. This was first proved for any distribution with a bounded probability density in \cite{frohlich1983absence},\cite{frohlich1985constructive} by Fr{\"o}hlich, Martinelli, Scoppola, and Spencer. Later on, the multi-scale method in \cite{frohlich1983absence},\cite{frohlich1985constructive} was strengthened to prove the same result for general H{\"o}lder continuous distributions in \cite{carmona1987anderson} by Carmona, Klein, and Martinelli.

As for the Bernoulli potential case, Anderson-Bernoulli localization near the bottom of the spectrum was verified in the continuous model $\mathbb{R}^{d}(d\geq 2)$ by Bourgain and Kenig in \cite{bourgain2005localization}. Their method relies on the unique continuation principle in $\mathbb{R}^{d}$ (\cite[Lemma 3.10]{bourgain2005localization}) and thus can not be directly applied to the discrete model on $\Z^{d}$. Recently, Buhovsky, Logunov, Malinnikova, and Sodin \cite{buhovsky2017discrete} developed a certain discrete version of the unique continuation principle for harmonic functions on $\Z^2$. Inspired by their work, Anderson-Bernoulli localization near the bottom of the spectrum was proved for $d=2$ by Ding and Smart in \cite{ding2020localization}, and for $d=3$ by Zhang and the author in \cite{li2022anderson}. 

For the large Bernoulli potential case (i.e. operator \eqref{eq:andersonmodel} with large $\Bar{V}$), the total length of the spectrum is always $8d$ by equation \eqref{eq:spectrum-interval}. 
When $\Bar{V}$ increases, the DOS measure behaves completely different from the case where $V$ has a H{\"o}lder continuous distribution. 
When $d=2$ and $p=\frac{1}{2}$,
the DOS measure always has a constant lower bound in the sets $Y_{\Bar{V}}$ and $\widetilde{Y_{\Bar{V}}}$ defined in Theorem \ref{thm:1/2_main} for sufficiently large $\Bar{V}$. On the other hand, the DOS measure is constantly small outside $Y_{\Bar{V}}\cup \widetilde{Y_{\Bar{V}}}$. 
Hence, Theorem \ref{thm:1/2_main} is again under the umbrella of prediction \eqref{eq:anderson-prediction}. 

Let us also mention that, although the smallness of DOS implies localization in many cases, the converse is not true. A much stronger result for Anderson localization is expected in dimensions one and two. For one dimension, it is proved that Anderson localization happens throughout the whole spectrum for any nontrivial distribution of $V$ with finite moment (see e.g. \cite{carmona1987anderson}). It is a general belief among physicists that (see e.g. \cite{simon2000schrodinger} by Simon), in dimension two, Anderson localization also happens throughout the whole spectrum for any finite nontrivial distribution of $V$.
Thus it is reasonable to conjecture that, in our model, localization also happens inside $Y_{\Bar{V}}\cup \widetilde{Y_{\Bar{V}}}$ and it is more of a technical limitation that we have to exclude $Y_{\Bar{V}}\cup \widetilde{Y_{\Bar{V}}}$ in the current paper. 

To prove Theorem \ref{thm:1/2_main}, we only need to consider the spectrum of $H$ contained in $[0,8]$ and prove the exponential decaying property of the resolvent as in Theorem \ref{thm:resolvent-exponential-decay} below.

We first set up some notations. 
Given $a \in \Z^2$, we define $\mathbbm{1}_{a}(a)=1$ and $\mathbbm{1}_{a}(a')=0$ if $a' \not = a$. Given $S \subset \Z^2$, an operator $A$ on $\ell^{2}(S)$ and $a,b \in S$, we write $A(a,b)=\left\langle \mathbbm{1}_{a},A \mathbbm{1}_{b} \right\rangle_{\ell^{2}(S)}$ where $\left\langle \cdot , \cdot  \right\rangle_{\ell^{2}(S)}$ denotes the inner product in $\ell^{2}(S)$. We also denote by $\|A\|$ the operator norm of $A$ in $\ell^{2}(S)$.

\begin{theorem}\label{thm:resolvent-exponential-decay}
   Let $d=2$, $p=\frac{1}{2}$. There exist positive integer $n$, constants $\kappa,\alpha,\varepsilon>0$ and energies $\lambda^{(1)},\lambda^{(2)},\cdots,\lambda^{(n)} \in \left[0,8 \right]$ such that the following holds.
  
  For any $\Bar{V}>0$, denote $Y_{\Bar{V}} = \bigcup_{i=1}^{n} \left[ \lambda^{(i)}-\Bar{V}^{-\frac{1}{4}},\lambda^{(i)}+\Bar{V}^{-\frac{1}{4}} \right]$. Let $H$ be defined as in \eqref{eq:andersonmodel}. Then for each $\Bar{V},L>\alpha$, each $\lambda_{0} \in [0,8]\setminus Y_{\Bar{V}}$ and each box $Q \subset \Z^2$ of side length $L$,
      \begin{equation}\label{eq:resolvent-exponentially-decaying}
          \Prob\left[|(H_{Q}-\lambda_{0})^{-1}(a,b)| \leq \Bar{V}^{L^{1-\varepsilon}-\varepsilon |a-b|} \text{ for $a,b \in Q$}\right]\geq 1-L^{-\kappa}.
      \end{equation}

\end{theorem}
Here $H_{Q}:\ell^{2}(Q)\rightarrow \ell^{2}(Q)$ is the restriction of Hamiltonian $H$ to the box $Q$ with the Dirichlet boundary condition.
\begin{proof}[Proof of Theorem \ref{thm:1/2_main} assuming Theorem \ref{thm:resolvent-exponential-decay}]

Probability estimate \eqref{eq:resolvent-exponentially-decaying} with the arguments in \cite[Section 7]{bourgain2005localization} implies that Anderson localization happens in $ [0,8]\setminus Y_{\Bar{V}}$. See also \cite[Section 6,7]{germinet2011comprehensive} by Germinet and Klein.

Now we use symmetry to prove Anderson localization for spectrum range $$[\Bar{V},\Bar{V}+8]\setminus \bigcup_{i=1}^{n} \left[ \widetilde{\lambda^{(i)}}-\Bar{V}^{-\frac{1}{4}},\widetilde{\lambda^{(i)}}+\Bar{V}^{-\frac{1}{4}} \right],$$ where $\widetilde{\lambda^{(i)}}=\Bar{V}+8-\lambda^{(i)}$. Define $\Tilde{V}:\Z^2 \rightarrow \{0,\Bar{V}\}$ by $\Tilde{V}(a)=\Bar{V}-V(a)(a \in \Z^2)$ and let $\Tilde{H}=-\Delta + \Tilde{V}$. Let $\Tilde{\lambda}=\Bar{V}+8-\lambda$ for every $\lambda \in \mathbb{R}$. For each $u:\Z^{2} \rightarrow \mathbb{R}$, define $\Tilde{u}:\Z^2 \rightarrow \mathbb{R}$ by $\Tilde{u}(x,y)=(-1)^{x+y}u(x,y)$ for $x,y \in \Z$. This gives a bijection $u \mapsto \Tilde{u}$ from functions on $\Z^2$ to themselves. The properties \eqref{eq:polynomialbound} and \eqref{eq:exponentialdecaybound} in Theorem \ref{thm:1/2_main} are preserved under this bijection. Moreover, by direct calculations, we have 
\begin{equation}
    H u= \lambda u \text{\; if and only if \;} \Tilde{H}\Tilde{u}=\Tilde{\lambda} \Tilde{u}.  
\end{equation}

Since $\Tilde{H}$ has the same distribution as $H$, Anderson localization happens in $\{\Tilde{\lambda}:\lambda \in \left[0,8 \right]\setminus Y_{\Bar{V}}\}=[\Bar{V},\Bar{V}+8]\setminus \bigcup_{i=1}^{n} \left[ \widetilde{\lambda^{(i)}}-\Bar{V}^{-\frac{1}{4}},\widetilde{\lambda^{(i)}}+\Bar{V}^{-\frac{1}{4}} \right]$. Theorem $\ref{thm:1/2_main}$ follows. 
\end{proof}

\subsection{Outline}\label{sec:wegner-outline}
To prove localization, \cite{ding2020localization} and \cite{bourgain2005localization} used a multi-scale analysis to prove an estimate similar to \eqref{eq:resolvent-exponentially-decaying}. These two previous works considered the edge of the spectrum where the Lifshitz tail phenomenon happens and used this phenomenon to prove the initial step of the induction in the multi-scale analysis. Then they used an eigenvalue variation argument to prove the Wegner estimate which is crucial to the inductive steps.
The key to the eigenvalue variation argument is the unique continuation principle (see \cite[Theorem 1.6]{ding2020localization} and \cite[Lemma 3.10]{bourgain2005localization}).

The current paper follows the multi-scale analysis framework in \cite{ding2020localization} and \cite{bourgain2005localization}, and studies the spectrum range beyond the edge by taking advantage of site percolation (see Section \ref{sec:site-percolation} and \cite[Chapter 1.6]{grimmett1999percolation} for the former definition of site percolation). Informally, the condition $p=\frac{1}{2}$ (or more generally, $p \in (1-p_{c},p_{c})$) implies that the sites with the same potential rarely form large connected components (Proposition \ref{sta:sharpness}). By this fact, the initial scale case (Proposition \ref{prop:L_1good}) for the multi-scale analysis is proved
for energies away from $\lambda^{(i)}$'s which are eigenvalues of the minus Laplacian restricted on small finite subsets of $\Z^{2}$ (thus away from $\lambda^{(i)}$'s, the DOS measure is small). When $p \in (p_{c},1)$, the subset $\{a \in \Z^{2}:V(a) = 0\}$ contains an infinite connected component (see \cite[Chapter 1.6]{grimmett1999percolation}) on which it is not obvious whether the resolvent decays exponentially or not. By symmetry, the same is true for the subset $\{a\in\Z^{2}:V(a) = \Bar{V}\}$ in the case where $p \in (0,1-p_{c})$. The critical cases where $p \in \{1-p_{c},p_{c}\}$ are more subtle.  

The most important and difficult part for the induction of the multi-scale analysis is to prove the Wegner estimate (Proposition \ref{prop:Wegner}) which indicates log-H\"{o}lder continuity of the DOS measure (see e.g. \cite[Section 6]{bourgain2005anderson}). Our Wegner estimate states that, for an interval of length less than $O(\Bar{V}^{-L^{1-\varepsilon'}})$, the probability that it contains an eigenvalue of $H_{Q_{L}}$ is less than $O(L^{-\kappa'})$ for some $\kappa',\varepsilon'>0$. 

To prove the Wegner estimate, we prove an upper bound and a lower bound on how far an eigenvalue of $H_{Q_{L}}$ will move after perturbing the potential function $V$. Here, ``perturb'' means changing the value of $V$ at some vertices from $0$ to $\Bar{V}$ or $\Bar{V}$ to $0$. 

The upper bound estimate requires to show that if the $j$-th smallest eigenvalue is close to a given real number $\lambda_{0}$, then one can perturb the potential $V$ on a $(1-\varepsilon)$ portion of $Q_{L}$ such that the $j$-th smallest eigenvalue will not move too far (less than $O(\Bar{V}^{-L^{1-\varepsilon''}})$ with $\varepsilon''>\varepsilon'$). Here and throughout this section, when we say the $j$-th smallest eigenvalue, we always count with multiplicity.  While this upper bound estimate was proved for $\lambda_{0}$ near the bottom of the spectrum in \cite{ding2020localization}, it is simply not true for $\lambda_{0}$ away from the bottom. For example, suppose $H_{Q_{L}}$ has $k>0$ eigenvalues (with multiplicities) in $[0,8]$. Pick an arbitrary $a\in Q_{L}$ with $V(a)=0$ and let the perturbed operator $H'_{Q_{L}}$ be obtained by changing the potential $V$ from $0$ to $\Bar{V}$ only at vertex $a$. It can be shown that the $k$-th smallest eigenvalue of $H'_{Q_{L}}$ is in $[\Bar{V},\Bar{V}+8]$ and thus is far from the $k$-th smallest eigenvalue of $H_{Q_{L}}$ which is in $[0,8]$. Hence we can not expect the upper bound estimate to hold in its original version.

It turns out that a different version of the upper bound estimate still holds. In that version, we will not compare the $j$-th smallest eigenvalue of an operator with the $j$-th smallest eigenvalue of its perturbation. We will make another correspondence between the eigenvalues of an operator and the eigenvalues of its perturbation. To clarify, in the previous example, the $k$-th eigenvalue of $H_{Q_{L}}$ will actually correspond to the $(k-1)$-th eigenvalue of $H'_{Q_{L}}$ and the distance between these two eigenvalues will be shown to be small, provided one of them is close to $\lambda_{0}$. 
To rigorously find the correspondence between eigenvalues of an operator and eigenvalues of its perturbation, we will introduce the \emph{``cutting procedure''} which continuously ``transforms'' the operator $H_{Q_{L}}$ (and $H'_{Q_{L}}$) to a direct sum operator $\bigoplus_{i} H_{\Lambda_{i}}$ (and $\bigoplus_{i} H'_{\Lambda_{i}}$) respectively. Here, $\bigcup_{i} \Lambda_{i}=Q_{L}$ is a disjoint union.
The $j$-th eigenvalue of the operator $H_{Q_{L}}$ corresponds to $j'$-th eigenvalue of $H'_{Q_{L}}$ only if the $j$-th eigenfunction of $\bigoplus_{i} H_{\Lambda_{i}}$ equals the $j'$-th eigenfunction of $\bigoplus_{i} H'_{\Lambda_{i}}$. Under this correspondence of eigenvalues, the upper bound estimate is stated as Claim \ref{cla:total-number-bad-eigenvalues}. 
The formal definition of the cutting procedure is given in Definition \ref{def:cutting} and \ref{def:extended-cutting} by using percolation clusters. 

The lower bound estimate requires to show that there is an enough portion of points in $Q_{L}$ such that, when the potential increases on any of these points, a given eigenvalue will move a decent distance (at least $\Omega(\Bar{V}^{-L^{1-\varepsilon'}})$). Based on the heuristic that increasing the potential at vertices where an eigenfunction $u$ has large absolute values will increase the associated eigenvalue fast, one only needs to show that the eigenfunction $u$ has a decent lower bound on an enough portion of points in $Q_{L}$. This is guaranteed by a discrete version of the unique continuation principle Theorem \ref{thm:unique-continuation-intro} which is analog of \cite[Theorem 1.6]{ding2020localization}.
However, under the new correspondence of eigenvalues, the $j$-th eigenvalue of $H_{Q_{L}}$ may correspond to either the $j$-th eigenvalue or the $(j-1)$-th eigenvalue of the perturbation $H'_{Q_{L}}$ (here $H'_{Q_{L}}$ is obtained from $H_{Q_{L}}$ by changing the potential $V$ from $0$ to $\Bar{V}$ only at one vertex). If it corresponds to the $j$-th eigenvalue of $H'_{Q_{L}}$, then by monotonicity, the eigenvalue will increase. Otherwise, if it corresponds to the $(j-1)$-th eigenvalue of $H'_{Q_{L}}$, then by Cauchy interlacing theorem, the eigenvalue will decrease. Either way, the lower bound estimate can be proved for rank one perturbation, provided we can have a quantitative estimate on the difference. This is considered in Lemma \ref{lem:eigenvariation} and Lemma \ref{lem:eigen-obstruction}.

To have a polynomial bound on the probability (i.e. the left-hand side of \eqref{eq:resolvent-exponentially-decaying}), we need to consider the perturbation on a large set of vertices rather than only one vertex. For this purpose, the previous works \cite{ding2020localization} and \cite{bourgain2005localization} used the Sperner lemma which deals with monotone functions. However,
as seen in the argument above, under the new correspondence, the eigenvalue is no longer a monotone function of the potential. Thus the original Sperner lemma (\cite[Theorem 4.2]{ding2020localization}) can not be applied to our case. Instead, we generalize Sperner lemma to deal with directed graph products and prove Lemma \ref{lem:combinatorics} which is another new ingredient. The original Sperner lemma (\cite[Theorem 4.2]{ding2020localization}) can be seen as a special case of Lemma \ref{lem:combinatorics} when each directed graph consists of two vertices and one directed edge. The details are given in Section \ref{sec:Sperner-theorem}.
\subsection{Discrete unique continuation principle}\label{sec:unique-continuation-intro}
Let us denote $Q_{l}=Q_{l}(\mathbf{0})$. Given a real number $k>0$, we write $k Q_{l}(a)=Q_{kl}(a)$. 
The discrete unique continuation principle is the following
\begin{theorem}\label{thm:unique-continuation-intro}
  For every small $\varepsilon>0$, there exists $\alpha>1$ such that the following holds. If $\lambda_{0}\in [0,8]$ is an energy, $\Bar{V}\geq 2$ and $Q\subset\Z^2$ is a box of side length $L\geq \alpha$, then $\Prob[\mathcal{E}]\geq 1-\exp(-\varepsilon L^{\frac{2}{3}})$, where $\mathcal{E}$ denotes the event that
  \begin{equation}\label{eq:unique_cont_support}
      \left|\left\{a\in Q: |u(a)|\geq (\Bar{V}L)^{-\alpha L}\|u\|_{\ell^{\infty}(\frac{Q}{100})}\right\}\right|\geq \varepsilon^{3} L^{2}
  \end{equation}
  holds whenever $\lambda\in \R$, $u:\Z^{2}\rightarrow \R$, $|\lambda-\lambda_{0}|\leq (\Bar{V}L)^{-\alpha L}$ and $H u=\lambda u$ in $Q$.
\end{theorem}
 In fact, the multi-scale analysis framework requires us to prove a slightly stronger version, Lemma \ref{lem:unique-continuation}, which accommodates a sparse ``frozen set''. An important feature of Theorem \ref{thm:unique-continuation-intro} and Lemma \ref{lem:unique-continuation} is that the probability estimate does not depend on $\Bar{V}$. This feature is one of the major reasons why the set of energies $\lambda^{(i)}$'s in Theorem \ref{thm:1/2_main} does not grow when $\Bar{V}$ increases.

Theorem \ref{thm:unique-continuation-intro} generalizes \cite[Theorem 1.6]{ding2020localization} to deal with large $\Bar{V}$ and proves a $\Omega(L^{2})$ lower bound on the cardinality of the support of $u$. This improves the previous $\Omega(L^{\frac{3}{2}}(\log L)^{-\frac{1}{2}})$ lower bound in \cite[Theorem 1.6]{ding2020localization}. The price we pay is that the energy window needs to be $O((\Bar{V}L)^{-\alpha L})$ (while it was $O(\exp(-\alpha(L\log L)^{\frac{1}{2}}))$ in \cite{ding2020localization}).

We refer the reader to the beginning of Section \ref{sec:proof-Lemma-3.5} for a proof outline and a comparison between proofs of \cite[Theorem 1.6]{ding2020localization} and Theorem \ref{thm:unique-continuation-intro}. Here, we only mention
the main new ingredient Lemma \ref{lem:Boolean-cube} which is proved in Section \ref{sec:auxiliary-boolean}. In fact, a weaker form of Lemma \ref{lem:Boolean-cube} will suffice for the proof of Theorem \ref{thm:unique-continuation-intro}.
\begin{lemma}\label{lem:Boolean-cube}
Given positive integers $k<n$, we denote the $n$ dimensional Boolean cube by $B^{n}=\left\{(x_{1},x_{2},\cdots,x_{n})\in \mathbb{R}^{n}:x_{i}\in \{0,1\} \text{ for each $1\leq i\leq n$}\right\}$. 
Then for any $k$ dimensional affine space $\Gamma\subset \mathbb{R}^{n}$, 
\begin{equation}\label{eq:number-of-intersection}
    \#\{a \in B^{n}: \min_{b\in \Gamma} \|a-b\|_{2}< \frac{1}{4} n^{-\frac{1}{2}} (n-k)^{-\frac{1}{2}} \} \leq 2^{k+1}.
\end{equation}
\end{lemma}
Lemma \ref{lem:Boolean-cube} can be seen as a quantitative version of Odlyzko Lemma (see e.g. \cite{odlyzko1988subspaces}). To prove it, we will find a subset $\mathcal{S}\subset \{1,\cdots,n\}$ with $|\mathcal{S}|=n-k-1$ such that, the projection operator onto the orthogonal complement of $\Gamma$ is ``well invertible'' when it is restricted on $\mathbb{R}^{\mathcal{S}}$. The existence of $\mathcal{S}$ is
a direct consequence of the following ``Restricted Invertibility Theorem'' for matrices with isotropic columns which was previously proved in \cite{marcus2014ramanujan} by Marcus, Spielman, and Srivastava.
\begin{lemma}[Theorem 3.1 in \cite{marcus2014ramanujan}]\label{lem:restricted-invertibility}
Suppose $v_{1},v_{2},\cdots,v_{l}\in \mathbb{C}^{m}$ are vectors with $\sum_{i=1}^{l} v_{i} v_{i}^{\dag} =I_{m}$ where $v_{i}^{\dag}$ is the dual vector of $v_{i}$ and $I_{m}$ is the identity matrix.

Then for every $m'<m$ there is a subset $\mathcal{S}\subset \{1,2,\cdots,m\}$ of size $m'$ such that the $m'$-th largest eigenvalue of 
$\sum_{i\in \mathcal{S}}  v_{i} v_{i}^{\dag}$ is at least $\left(1-\sqrt{\frac{m'}{m}}\right)^{2} \frac{m}{l}$.
\end{lemma}
For general restricted invertibility principles and their history, we refer to \cite{naor2017restricted} by Naor and Youssef.
\subsection{Notations}
We set up some notations in this subsection. Throughout the paper, we regard $\Z^2$ as a graph with vertices $\{(x,y):x,y \in \Z\}$ and there is an edge connecting $a,b \in \Z^2$ if and only if $|a-b|=1$ (in this case, we also write $a\sim b$).

Given any subset $S \subset \Z^2$ and function $f:\Z^2 \rightarrow \mathbb{R}$, we define the restriction $f|_{S}:S \rightarrow \mathbb{R}$ by $f|_{S}(a)=f(a)$ for $a \in S$.
We denote $P_{S}:\ell^{2}(\Z^2)\rightarrow \ell^{2}(S)$ to be the projection operator defined by $P_{S}f=f|_{S}$ for each $f \in \ell^{2}(\Z^2)$. For simplicity, we write $\|f\|_{\ell^{2}(S)}=\|P_{S}f\|_{\ell^{2}(S)}$.  For an operator $A$ on $\ell^{2}(\Z^2)$, we denote $A_{S}=P_{S} A P_{S}^{\dag}$ where $P_{S}^{\dag}$ is the adjoint operator of $P_{S}$.

Throughout the rest of the paper, $H$ always denotes the operator defined in \eqref{eq:andersonmodel}. Given $\lambda \in \mathbb{C}\setminus \sigma(H_{S})$, we write $G_{S}(a,b;\lambda)=(H_{S}-\lambda)^{-1}(a,b)$ for $S\subset \Z^2$ and $a,b\in S$. 

For any real function $u$ defined on a set $D$ and any real number $c$, we use $\{u\geq c\}$ as shorthand for the set $\{a \in D: u(a) \geq c\}$.

\subsection*{Organization of remaining text}
In Section \ref{sec:initial-scale}, we define the cutting procedure. Along this way, we prove the induction base case (Proposition \ref{prop:L_1good}) for the multi-scale analysis. The sharpness of site percolation (Proposition \ref{sta:sharpness}) plays a key role there.

In Section \ref{sec:Wegner-estimate}, we prove the Wegner estimate Proposition \ref{prop:Wegner}. We will first collect all needed lemmas in Section \ref{sec:auxiliary-lemmas} and prove a generalized Sperner lemma in Section \ref{sec:Sperner-theorem}. The proof of the Wegner estimate is given in Section \ref{sec:proof-Wegner}.

In Section \ref{sec:larger-scale}, we perform the multi-scale analysis by using the Wegner estimate and prove Theorem \ref{thm:resolvent-exponential-decay}.

In Section \ref{sec:proof-Lemma-3.5}, we prove the unique continuation principle Theorem \ref{thm:unique-continuation-intro} and its stronger version Lemma \ref{lem:unique-continuation}.

Among these four sections, Section \ref{sec:larger-scale} follows closely the existing framework in \cite{ding2020localization} and \cite{bourgain2005localization} while other sections contain the new ingredients as follows:
\begin{itemize}
    \item A ``cutting procedure'' which allows us to match eigenvalues under different potential functions (Section \ref{sec:initial-scale}).
    \item The use of the sharpness of site percolation in the proof of the initial case of the multi-scale analysis (Section \ref{sec:initial-scale}).
    \item A generalized Sperner lemma for directed graph products (Section \ref{sec:Wegner-estimate}).
    \item A 2D unique continuation theorem with an improved lower bound (see \eqref{eq:unique_cont_support}) and a smaller energy window (Section \ref{sec:proof-Lemma-3.5}).
\end{itemize}

\section{Initial scale}\label{sec:initial-scale}
In this section, we use site percolation (Section \ref{sec:site-percolation}) to define the cutting procedure described in the introduction. We will first define \emph{$r$-bits} that are boxes centered in a sublattice with a certain edge length (Definition \ref{def:r-bit-admissible}). We then define the cutting procedure for Hamiltonian restricted on $r$-bits by using percolation clusters (Definition \ref{def:cutting}). These $r$-bits will also be used as ``basic units'' for eigenvalue variation arguments in the proof of the Wegner estimate Proposition \ref{prop:Wegner} in Section \ref{sec:Wegner-estimate}. Then we will extend the cutting procedure to boxes with larger length scales (Definition \ref{def:extended-cutting}). Finally, we will prove the induction base case for the multi-scale analysis (Proposition \ref{prop:L_1good}).

\subsection{Site percolation}\label{sec:site-percolation}
Consider the Bernoulli site percolation on $\Z^2$. Let $p \in (0,1)$, suppose each vertex in $\Z^2$ is independently occupied with probability $p$. It is well known that there exists a critical probability $p_{c}\in (0,1)$ such that, for $p>p_{c}$, almost surely, there exists an infinite connected subset of $\Z^2$ whose vertices are occupied; for $p<p_{c}$, almost surely, there does not exist an infinite connected subset of $\Z^2$ whose vertices are occupied. It is known that $p_{c} >\frac{1}{2}$, see e.g. \cite{grimmett1998critical} by Grimmett and Stacey.

\begin{defn}
For any $S \subset \Z^2$, denote 
$$\partial^{+}S=\{a \in \Z^2 \setminus S: \text{$a \sim b$ for some $b\in S$}\}$$ to be the outer boundary of $S$; and $$\partial^{-}S=\{a \in S: \text{$a \sim b$ for some $b\in \Z^2 \setminus S$}\}$$ to be the inner boundary of $S$. Denote $$\partial S=\left\{\{a,b\}: a \in \partial^{+}S \text{, } b \in \partial^{-}S \text{ and } a \sim b \right\}$$ to be the set of edges connecting elements in $\partial^{-}S$ and $\partial^{+}S$.
\end{defn}

The following sharpness proposition follows directly from $p_{c}>\frac{1}{2}$ and \cite[Theorem 7.3]{AizenmanSharpness} by Aizenman and Barsky:
\begin{prop}\label{sta:sharpness}
Suppose $V:\Z^2 \rightarrow \{0,\Bar{V}\}$ is a random function such that $\{V(a): a \in \Z^2\}$ is a family of i.i.d. random variables with $\Prob\left[V(a)=0\right]=\frac{1}{2}$ and $\Prob\left[V(a)=\Bar{V}\right]=\frac{1}{2}$. 
There is a numerical constant $c_{0}>0$ such that, for each $l>10$ and $b\in \Z^2$,
\begin{equation}
    \Prob \left[\mathcal{E}^{l}_{per}(b)\right]< \exp(-c_{0}l).
\end{equation}
Here, $\mathcal{E}_{per}^{l}(b)$ denotes the event that there is a path in $\Z^2$ joining $b$ to some vertex in $\partial^{-} Q_{l}(b)$ such that $V$ equals $0$ on all vertices in this path.
\end{prop}

\subsection{$r$-bit}
Let $\varepsilon_{0}>0$ be a fixed small constant such that 
\begin{equation}\label{eq:def-of-epsilon0}
    \varepsilon_{0}<\varepsilon_{1}^{10}
\end{equation}
where $\varepsilon_{1}$ is the numerical constant appeared in Lemma \ref{lem:unique-continuation} below.

The inequality \eqref{eq:def-of-epsilon0} will only be used in the proof of Proposition \ref{prop:Wegner}. At this moment, the reader can think of $\varepsilon_{0}$ as a small numerical constant.
\begin{defn}\label{def:define-omega}
For any large odd number $r$, denote $\dot{r}=\left\lceil (1-\frac{\varepsilon_{0}}{2})(r-1) \right\rceil$.
For any vertex $a\in \dot{r}\Z^2$ where $\dot{r}\Z^2=\{(\dot{r} x,\dot{r} y):x,y\in \Z\}$, let $\Omega_{r}(a)=Q_{ (1-2\varepsilon_{0})r }(a)$, $\Tilde{\Omega}_{r}(a)=Q_{(1-\frac{3}{2}\varepsilon_{0})r}(a)$ and $F_{r}(a)=Q_{r}(a)\setminus \Omega_{r}(a)$.
\end{defn}

\begin{defn}\label{def:r-bit-admissible}
Given a large odd number $r$, a vertex $a\in \dot{r}\Z^2$ and a potential function $V' :F_{r}(a) \rightarrow \{0,\Bar{V}\}$, we call $(Q_{r}(a),V')$ an \emph{$r$-bit}.
We say $(Q_{r}(a),V')$ is \emph{admissible} if the following two items hold:
\begin{itemize}
    \item For each $x \in \partial^{-}Q_{r}(a)$ and $y \in F_{r}(a)$ with $|x-y| \geq \frac{\varepsilon_{0}}{30}r$, there is no path in $F_{r}(a)$ joining $x$ to $y$ such that $V'$ equals $0$ on all vertices in the path.
    \item There is no path in $F_{r}(a)$ joining some vertex in $\partial^{+}\Omega_{r}(a)$ to some vertex in $\partial^{-}\Tilde{\Omega}(a)$ such that $V'$ equals $0$ on all vertices in the path.
\end{itemize}

With a little abuse of notations, we also call $Q_{r}(a)$ an $r$-bit if $a \in \dot{r}\Z^2$. When $V':F_{r}(a)\rightarrow \{0,\Bar{V}\}$ is obviously given, we also say $Q_{r}(a)$ is admissible if $(Q_{r}(a),V')$ is admissible.

Given an $r$-bit $Q_{r}(a)$, we say it is \emph{inside} some $S \subset \Z^2$ if $Q_{r}(a) \subset S$. We say it \emph{does not affect} $S$ if $\Omega_{r}(a) \cap S=\emptyset$.
\end{defn}

\begin{rem}\label{rem:geo-of-r-bit}
We give here three remarks on $r$-bits, the first two are from Definition \ref{def:define-omega} and the third one is obvious by Definition \ref{def:r-bit-admissible}. See also Figure \ref{fig:r-bits} for an illustration.
\begin{enumerate}
    \item For two different $r$-bits $Q_{r}(a_{1})$ and $Q_{r}(a_{2})$, we have $$\Tilde{\Omega}_{r}(a_{1}) \cap (\partial^{+}Q_{r}(a_{2})\cup Q_{r}(a_{2}))=\emptyset.$$
    Note that, $\Tilde{\Omega}_{r}(a_{1})$ is a scaling image of $r$-bit $Q_{r}(a_{1})$ with the scaling constant slightly smaller than $1$. Thus the equation above means $\Tilde{\Omega}_{r}(a_{1})$ is disjoint from other $r$-bits and their outer boundaries.
    \item  For any $a \in \Z^2$, there exists an $r$-bit $Q_{r}(b)$ with $a \in Q_{(1-\frac{2}{5}\varepsilon_{0})r}(b)$.
    \item Suppose $r$-bits $(Q_{r}(a),V')$ and $(Q_{r}(a'),V'')$ satisfy $V'(b)=V''(b-a+a')$ for each $b\in F_{r}(a)$, then $(Q_{r}(a),V')$ is admissible if and only if $(Q_{r}(a'),V'')$ is admissible.
\end{enumerate}
\end{rem}
\begin{figure}
    \centering
    \includegraphics{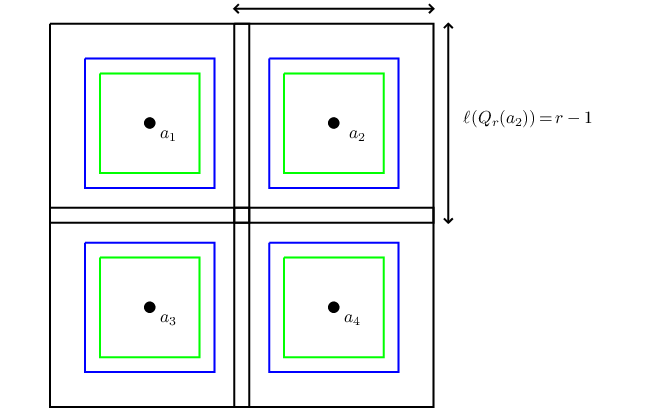}
    \caption{The black squares represent $r$-bits $Q_{r}(a_{i})(i=1,2,3,4)$ with overlaps, the blue squares represent $\Tilde{\Omega}_{r}(a_{i})(i=1,2,3,4)$ and the green squares represent $\Omega_{r}(a_{i})(i=1,2,3,4)$.}
    \label{fig:r-bits}
\end{figure}
The following Proposition \ref{prop:r-admissible} is the place where we use the sharpness of site percolation (Proposition \ref{sta:sharpness}).
\begin{prop}\label{prop:r-admissible}
Suppose odd number $r$ is large enough. Let $V:\Z^2 \rightarrow \{0,\Bar{V}\}$ be the i.i.d. Bernoulli random potential with $\Prob (V(a)=0)=\Prob (V(a)=\Bar{V})=\frac{1}{2}$ for each $a\in \Z^{2}$. Then for each $a\in \dot{r}\Z^2$, we have
\begin{equation}
    \Prob\left[\text{$(Q_{r}(a),V|_{F_{r}(a)})$ is admissible}\right] >  1-\exp(-8c_{1}r)
\end{equation}
where $c_{1}$ is a numerical constant.
\end{prop}
\begin{proof}
Let $\mathcal{E}_{nad}(a)$ be the event that $(Q_{r}(a),V|_{F_{r}(a)})$ is not admissible. Then by Definition \ref{def:r-bit-admissible},
\begin{equation}
    \mathcal{E}_{nad}(a) \subset \bigcup_{b \in \partial^{-}\Tilde{\Omega}(a) \cup \partial^{-}Q_{r}(a)} \mathcal{E}^{\frac{\varepsilon_{0}}{60} r}_{per}(b).
\end{equation}
Here, the notation $\mathcal{E}^{l}_{per}(b)$ is defined in Proposition \ref{sta:sharpness}.
Assume $r$ is large enough, by Proposition \ref{sta:sharpness},
\begin{equation}
    \Prob\left[\mathcal{E}_{nad}(a)\right] \leq 8r \exp(-\frac{c_{0} \varepsilon_{0}}{60} r)< \exp(-8c_{1}r),
\end{equation}
where $c_{1}<\frac{c_{0} \varepsilon_{0}}{480}$ is a numerical constant.
\end{proof}

\begin{defn}\label{def:define-S-r}
For any $r$-bit $(Q_{r}(a),V|_{F_{r}(a)})$, 
we denote by $S_{r}(a)$ the maximal connected subset of $\Omega_{r}(a)\cup \{b \in F_{r}(a):V(b)=0\}$ that contains $\Omega_{r}(a)$. 
\end{defn}
\begin{lemma}\label{lem:property-of-Sr}
Given $V_{0}:Q_{r}(a)\rightarrow \{0,\Bar{V}\}$, suppose $(Q_{r}(a),V_{0}|_{F_{r}(a)})$ is an admissible $r$-bit. Then we have the following properties:
\begin{enumerate}
    \item $\Omega_{r}(a)\subset S_{r}(a)\subset \Tilde{\Omega}_{r}(a)\setminus \partial^{-}\Tilde{\Omega}_{r}(a)$.
    \item $S_{r}(a)$ is $V_{0}|_{F_{r}(a)}$-measurable.
    \item $V_{0}(b)=\Bar{V}$ for each $b\in \partial^{+}S_{r}(a)$.
\end{enumerate}
\end{lemma}
\begin{proof}
The first property is due to the second item in Definition \ref{def:r-bit-admissible}. The second property follows directly from Definition \ref{def:define-S-r}. The third property follows from the maximality of $S_{r}(a)$.
\end{proof}

We now define the ``cutting procedure'' on an admissible $r$-bit $Q_{r}(a)$. Intuitively, the cutting procedure on $Q_{r}(a)$ continuously modifies the edge weight of $\partial S_{r}(a)$ and finally splits $S_{r}(a)$ and $Q_{r}(a)\setminus S_{r}(a)$.
\begin{defn}\label{def:cutting}
Given $V:Q_{r}(a)\rightarrow \{0,\Bar{V}\}$, suppose $(Q_{r}(a),V|_{F_{r}(a)})$ is an admissible $r$-bit. For $t\in \left[0,1\right]$, define operator $H^{t}_{Q_{r}(a)}:\ell^{2}(Q_{r}(a))\rightarrow \ell^{2}(Q_{r}(a))$ as follows: $H^{t}_{Q_{r}(a)}(b,c)=t-1$ if $\{b,c\} \in \partial S_{r}(a)$; $H^{t}_{Q_{r}(a)}(b,c)=H_{Q_{r}(a)}(b,c)$ otherwise. Denote $G^{t}_{Q_{r}(a)}(b,c;\lambda)=(H^{t}_{Q_{r}(a)}-\lambda)^{-1}(b,c)$ for any $b,c \in Q_{r}(a)$.
\end{defn}
\begin{rem}
From Definition \ref{def:cutting}, $H^{t}_{Q_{r}(a)}$ is self-adjoint for each $t$. We have $H^{0}_{Q_{r}(a)}=H_{Q_{r}(a)}$ and $H^{1}_{Q_{r}(a)}= H_{S_{r}(a)} \bigoplus H_{Q_{r}(a) \setminus S_{r}(a)}$.
\end{rem}
\begin{lemma}\label{lem:spectrum-range-of-H}
Given $V:Q_{r}(a)\rightarrow \{0,\Bar{V}\}$, suppose $(Q_{r}(a),V|_{F_{r}(a)})$ is an admissible $r$-bit. Then for each $t\in [0,1]$ and each connected subset $S\subset Q_{r}(a)$, we have
\begin{equation}\label{eq:spectrum-range-Ht}
    \sigma\left(H^{t}_{Q_{r}(a)}\right) \subset \left[0,8\right]\cup \left[\Bar{V},\Bar{V}+8\right],
\end{equation}
and 
\begin{equation}\label{eq:spectrum-range-HS}
    \sigma\left(H_{S}\right) \subset \left[0,8\right]\cup \left[\Bar{V},\Bar{V}+8\right].
\end{equation}
\end{lemma}
\begin{proof}
We first prove \eqref{eq:spectrum-range-Ht}.
Suppose $\lambda \in \sigma\left(H^{t}_{Q_{r}(a)}\right)$, let $u$ be an eigenfunction with $H^{t}_{Q_{r}(a)}u=\lambda u$. Pick $b\in Q_{r}(a)$ with $|u(b)| \geq |u(b')|$ for each $b'\in Q_{r}(a)$. Then we have
\begin{equation}\label{eq:maximal-point}
    (V(b)+4-\lambda) u(b)= -\sum_{\substack{ b' \sim b\\ b' \in Q_{r}(a)}} H^{t}_{Q_{r}(a)}(b,b')u(b').
\end{equation}
Since $|H^{t}_{Q_{r}(a)}(b,b')| \leq 1$ for each $b\not= b'$, we have \eqref{eq:maximal-point} implies $$
|(V(b)+4-\lambda) u(b)| \leq 4 |u(b)|,
$$
and thus $|(V(b)+4-\lambda)| \leq 4$. The conclusion follows from $V(b) \in \{0,\Bar{V}\}$.

Finally, to prove \eqref{eq:spectrum-range-HS}, substitute $H^{t}_{Q_{r}(a)}$ by $H_{S}$ and repeat the above argument.
\end{proof}
\subsection{Resolvent estimate on $r$-bits}
Now we define the exceptional energies $\lambda^{(i)}$'s in Theorem \ref{thm:1/2_main} and Theorem \ref{thm:resolvent-exponential-decay}. They are exactly the eigenvalues of the minus Laplacian restricted on connected subsets of $Q_{r}$. A small neighbourhood of them (the set $J_{r}^{\Bar{V}}$ in Definition \ref{def:J_eig}) is excluded so that the resolvent is bounded on admissible $r$-bits (Proposition \ref{prop:0thscale}).
\begin{defn}\label{def:J_eig}
Given an odd number $r$ and a real number $U>1$, let $$Eig_{r}=\bigcup_{\substack{S \subset Q_{r}\\ \text{$S$ is connected}\\ }} \sigma\big( (-\Delta)_{S} \big)$$ and $$J_{r}^{U}=\bigcup_{x\in Eig_{r}} \left[x-U^{-\frac{1}{4}},x+U^{-\frac{1}{4}} \right].$$
\end{defn}

\begin{prop}\label{prop:0thscale}
Given $r$ a large odd number, $a\in \dot{r}\Z^{2}$ and $V':F_{r}(a)\rightarrow \{0,\Bar{V}\}$, we assume $\Bar{V}>\exp(r^{2})$. 
Suppose $r$-bit $(Q_{r}(a),V')$ is admissible and $\lambda_{0} \in [0,8]\setminus J_{r}^{\Bar{V}}$. Then
for each $V:Q_{r}(a)\rightarrow \{0,\Bar{V}\}$ with $V|_{F_{r}(a)}=V'$, each $t \in [0,1]$ and each connected subset $S \subset Q_{r}(a)$, we have the following:

\begin{itemize}
    \item $\|(H^{t}_{Q_{r}(a)}-\lambda_{0})^{-1}\| \leq  2 \Bar{V}^{\frac{1}{4}}$.
    \item $\|(H_{S}-\lambda_{0})^{-1}\| \leq  2 \Bar{V}^{\frac{1}{4}}$.
    \item $|G^{t}_{Q_{r}(a)}(b,b';\lambda_{0})| \leq \Bar{V}^{-\frac{1}{4}}$ for each $b\in \partial^{-}Q_{r}(a)$, $b' \in Q_{r}(a)$ such that $|b-b'| \geq \frac{\varepsilon_{0}}{8}r$.
\end{itemize}

\end{prop}
\begin{proof}
We first prove the first item.  The strategy here is to prove that for any eigenvalue $\lambda$ of $H^{t}_{Q_{r}(a)}$, there is some $W'\subset Q_{r}(a)$ such that $\lambda$ is close to an eigenvalue of $H_{W'}$.

If there is no eigenvalue of $H^{t}_{Q_{r}(a)}$ in $[0,8]$, then by Lemma \ref{lem:spectrum-range-of-H}, $\|(H^{t}_{Q_{r}(a)}-\lambda_{0})^{-1}\| \leq  (\Bar{V}-8)^{-1}<2 \Bar{V}^{\frac{1}{4}}$ and the first item holds.

Now assume there is an eigenvalue $\lambda$ of $H^{t}_{Q_{r}(a)}$ in $[0,8]$ and we need to prove $|\lambda-\lambda_{0}|\geq \frac{1}{2}\Bar{V}^{-\frac{1}{4}}$. Let $v$ be an $\ell^{2}(Q_{r}(a))$ normalised eigenfunction of $H^{t}_{Q_{r}(a)}$ with eigenvalue $\lambda$. Write $T=\{a' \in Q_{r}(a): V(a')=\Bar{V}\}$. For each $a' \in T$, we have 
\begin{equation}
    -\sum_{\substack{b' \sim a'\\ b' \in Q_{r}(a)}} H^{t}_{Q_{r}(a)}(a',b') v(b')=(\Bar{V}+4-\lambda)v(a').
\end{equation} 
Since $\|v\|_{\ell^{2}(Q_{r}(a))}=1$ and $|H^{t}_{Q_{r}(a)}(b',b'')| \leq 1$ for any $b' \not= b''$, we have $|v(a')|\leq 4/(\Bar{V}-4)$ for $a' \in T$. This implies $\|v\|_{\ell^{2}(T)} \leq 4r/(\Bar{V}-4)<\frac{1}{2}$ since $\Bar{V}>\exp(r^{2})$. Consider all maximal connected subsets $W \subset Q_{r}(a)\setminus T$. The number of them is less than $r^{2}$, 
thus there exists one of these subsets $W' \subset Q_{r}(a)$ with 
\begin{equation}\label{eq:lower-bound-W'-v}
    \|v\|_{\ell^{2}(W')} \geq \frac{1}{2 r}.
\end{equation}
Since $V=0$ on $W'$, by Lemma \ref{lem:property-of-Sr}, we have $\partial^{+} S_{r}(a) \cap W'=\emptyset$ and $(H_{Q_{r}(a)}^{t})_{W'}=H_{W'}$. 
Thus for each $b \in W'$,
\begin{equation}\label{eq:region-W'}
    (H_{W'}-\lambda) (v|_{W'}) (b)=(H^{t}_{Q_{r}(a)} -\lambda)v(b) -\sum_{\substack{b' \sim b\\ b' \in \partial^{+}W' \cap Q_{r}(a)}} H^{t}_{Q_{r}(a)}(b,b')v(b').
\end{equation}
By maximality of $W'$, for each $a' \in \partial^{+}W' \cap Q_{r}(a)$, $a' \in T$ and thus $$|v(a')| \leq 4/(\Bar{V}-4).$$
Since $(H^{t}_{Q_{r}(a)}- \lambda)v=0$ and $|H^{t}_{Q_{r}(a)}(b,b')| \leq 1$ when $b \not = b'$, we have \eqref{eq:region-W'} implies $$|(H_{W'}-\lambda) (v|_{W'}) (b)| \leq 16/(\Bar{V}-4)$$ for each $b \in W'$. Thus
\begin{equation}\label{eq:eigenvalue-bound-1}
    \| (H_{W'}-\lambda) (v|_{W'})\|_{\ell^{2}(W')} \leq  16 r/(\Bar{V}-4) \leq 32 r^{2}/(\Bar{V}-4) \|v\|_{\ell^{2}(W')} 
\end{equation}
by \eqref{eq:lower-bound-W'-v}. Writing $v|_{W'}$ as a linear combination of eigenfunctions of $H_{W'}$, \eqref{eq:eigenvalue-bound-1} provides an eigenvalue $\lambda'$ of $H_{W'}$ such that $|\lambda-\lambda'| \leq 32 r^{2}/(\Bar{V}-4)$. Since $\lambda' \in Eig_{r}$ and $\lambda_{0} \not \in J_{r}^{\Bar{V}}$, by Definition \ref{def:J_eig}, we have
$$ |\lambda_{0}-\lambda|\geq |\lambda_{0}-\lambda'|-|\lambda'-\lambda| > \Bar{V}^{-\frac{1}{4}}-32 r^{2}/(\Bar{V}-4)>\frac{1}{2} \Bar{V}^{-\frac{1}{4}}.
$$
Here, we used $\Bar{V}>\exp(r^{2})$.
The first item follows. 

The proof of the second item is similar to the proof of the first item. Assume there is an eigenvalue $\lambda_{*}$ of $H_{S}$ in $[0,8]$. Let $v_{*}$ be an $\ell^{2}(S)$ normalised eigenfunction of $H_{S}$ with eigenvalue $\lambda_{*}$, we need to prove $|\lambda_{*}-\lambda_{0}|\geq \frac{1}{2}\Bar{V}^{-\frac{1}{4}}$. For each $a' \in T\cap S$, we have 
\begin{equation}
    -\sum_{\substack{b' \sim a'\\ b' \in S}} H_{S}(a',b') v_{*}(b')=(\Bar{V}+4-\lambda_{*})v_{*}(a').
\end{equation} 
Since $|H_{S}(b',b'')| \leq 1$ for any $b' \not= b''$, we have $|v_{*}(a')|\leq 4/(\Bar{V}-4)$ for $a' \in T\cap S$. This implies $\|v_{*}\|_{\ell^{2}(T)} \leq 4r/(\Bar{V}-4)<\frac{1}{2}$. Consider all maximal connected subsets $W \subset S\setminus T$. The number of them is less than $r^{2}$, 
thus there exists one of these subsets $W'' \subset S$ with $\|v_{*}\|_{\ell^{2}(W'')} \geq \frac{1}{2 r}$. 
For each $b \in W''$,
\begin{equation}\label{eq:region-W''}
    (H_{W''}-\lambda_{*}) (v_{*}|_{W''}) (b)=(H_{S} -\lambda_{*})v_{*}(b) -\sum_{\substack{b' \sim b\\ b' \in \partial^{+}W'' \cap S}} H_{S}(b,b')v_{*}(b').
\end{equation}
By maximality of $W''$, for each $a' \in \partial^{+}W'' \cap S$, $a' \in T$ and thus $$|v_{*}(a')| \leq 4/(\Bar{V}-4).$$
Since $(H_{S}- \lambda_{*})v_{*}=0$ and $|H_{S}(b,b')| \leq 1$ when $b \not = b'$, we have \eqref{eq:region-W''} implies $$\left|(H_{W''}-\lambda_{*}) (v_{*}|_{W''}) (b)\right| \leq 16/(\Bar{V}-4)$$ for each $b \in W''$. Thus
\begin{equation}\label{eq:eigenvalue-bound-2}
    \| (H_{W''}-\lambda_{*}) (v_{*}|_{W''})\|_{\ell^{2}(W'')} \leq  16 r/(\Bar{V}-4) \leq 32 r^{2}/(\Bar{V}-4) \|v_{*}\|_{\ell^{2}(W'')}.
\end{equation}
Writing $v_{*}|_{W''}$ as a linear combination of eigenfunctions of $H_{W''}$, \eqref{eq:eigenvalue-bound-2} provides an eigenvalue $\lambda_{*}'$ of $H_{W''}$ such that $|\lambda_{*}-\lambda_{*}'| \leq 32 r^{2}/(\Bar{V}-4)$. Since $\lambda_{*}' \in Eig_{r}$ and $\lambda_{0} \not \in J_{r}^{\Bar{V}}$, by the same argument for the first item, we have
$$ |\lambda_{0}-\lambda_{*}|>\frac{1}{2} \Bar{V}^{-\frac{1}{4}}
$$
and the second item follows.

Now we prove the third item and the strategy here is to exploit the resolvent identity. Pick $b,b'\in Q_{r}(a)$ with $b \in \partial^{-} Q_{r}(a)$ and $|b-b'|\geq \frac{\varepsilon_{0}r}{8}$. We claim that, there exists connected $S_{0} \subset Q_{r}(a)\cap Q_{\frac{\varepsilon_{0} r}{9} (b)}$ with $b \in S_{0} $ such that, for any $c \in S_{0}$ and $c' \in Q_{r}(a)\setminus S_{0}$ with $c \sim c'$, we have $c \in T$. To see this, if $V(b)=\Bar{V}$, then simply let $S_{0}=\{b\}$; otherwise, let $S_{1}$ be the maximal connected subset of $Q_{r}(a)\setminus T$ that contains $b$. Since $Q_{r}(a)$ is admissible, the first item in Definition \ref{def:r-bit-admissible} implies $S_{1} \subset Q_{r}(a)\cap Q_{\frac{\varepsilon_{0} r}{10} (b)}$. Let $S_{0}=S_{1} \cup (\partial^{+}S_{1}\cap Q_{r}(a))$ and our claim follows from the maximality of $S_{1}$.

By Lemma \ref{lem:property-of-Sr}, $S_{r}(a)\subset \Tilde{\Omega}_{r}(a)$ and thus $S_{r}(a) \cap (S_{0}\cup \partial^{+}S_{0}) =\emptyset$. By resolvent identity,
\begin{equation}\label{eq:G5}
    G^{t}_{Q_{r}(a)}(b,b';\lambda_{0})=\sum_{\substack{c \in S_{0},c \sim c'\\ c' \in Q_{r}(a) \setminus S_{0}}} G_{S_{0}}(b,c;\lambda_{0})G^{t}_{Q_{r}(a)}(c',b';\lambda_{0}).
\end{equation}
By definition of resolvent,
\begin{equation}
   (V(c)-\lambda_{0}+4) G_{S_{0}}(b,c;\lambda_{0})=  \delta_{c,b} + \sum_{c'' \sim c, c'' \in S_{0}} G_{S_{0}}(b,c'';\lambda_{0})
\end{equation}
where $\delta_{c,b}=1$ if $c=b$ and $\delta_{c,b}=0$ otherwise.
Hence 
\begin{equation}\label{eq:control-G-S}
    |G_{S_{0}}(b,c;\lambda_{0})| \leq \frac{1}{|V(c)-\lambda_{0}+4|} (1+4 \|(H_{S_{0}}-\lambda_{0})^{-1}\|).
\end{equation}
The second item of this proposition implies $\|(H_{S_{0}}-\lambda_{0})^{-1}\| \leq 2\Bar{V}^{\frac{1}{4}}$.
Assume $c \sim c'$ for some $c \in S_{0}$ and $c' \in Q_{r}(a) \setminus S_{0}$, then the property of $S_{0}$ and inequality \eqref{eq:control-G-S} together imply $V(c)=\Bar{V}$ and 
\begin{equation}\label{eq:bounding-G-S_0-b-c}
    |G_{S_{0}}(b,c;\lambda_{0})| \leq 20\Bar{V}^{-\frac{3}{4}}.
\end{equation}
Finally, in \eqref{eq:G5}, by the first item in this proposition and inequality \eqref{eq:bounding-G-S_0-b-c},
\begin{align*}
    |G^{t}_{Q_{r}(a)}(b,b';\lambda_{0})|
    \leq &\|(H^{t}_{Q_{r}(a)}-\lambda_{0})^{-1}\| \sum_{\substack{c \in S_{0},c \sim c'\\ c' \in Q_{r}(a) \setminus S_{0}}} |G_{S_{0}}(b,c;\lambda_{0})|\\
    \leq &2\Bar{V}^{\frac{1}{4}}\sum_{\substack{c \in S_{0},c \sim c'\\ c' \in Q_{r}(a) \setminus S_{0}}} |G_{S_{0}}(b,c;\lambda_{0})|\\
    \leq &320 r^{2} \Bar{V}^{-\frac{1}{2}}\\
    <&\Bar{V}^{-\frac{1}{4}}.
\end{align*}
\end{proof}
\subsection{Initial scale analysis}
In this subsection, we extend the cutting procedure to larger boxes and prove the induction base case for the multi-scale analysis (Proposition \ref{prop:L_1good}).
\begin{defn}
Suppose $r$ is an odd number, $a \in \Z^2$ and $L \in \Z_{+}$. We say $Q_{L}(a)$ is \emph{$r$-dyadic} if there exists $k \in \Z_{+}$ such that $a \in 2^{k}\dot{r}\Z^2$ and $L=2^{k+1}\dot{r}+r$. In this case, $L$ is called an $r$-dyadic scale.
\end{defn}

\begin{lemma}\label{lem:r_dyadic_cube-ele}
Suppose $Q_{L}(a)$ is an $r$-dyadic box. Then we have $$Q_{L}(a)=\bigcup_{b \in \dot{r}\Z^2 \cap Q_{L}(a)} Q_{r}(b).$$ If $r$-bit $Q_{r}(b') \not\subset Q_{L}(a)$, then $\Tilde{\Omega}_{r}(b')\cap Q_{L}(a)=\emptyset$.  
\end{lemma}
\begin{proof}
    We assume without loss of generality that $a=\mathbf{0}$. Write $L=2^{k+1}\dot{r}+r$ with $k\in \Z_{+}$. By \eqref{eq:def-cube-Q_l}, 
    \begin{equation}\label{eq:explicit-Q_L}
        Q_{L}(\mathbf{0})= \{(x,y)\in \Z^{2}: -2^{k}\dot{r} -\frac{r-1}{2} \leq x, y\leq 2^{k}\dot{r}+ \frac{r-1}{2}\}.
    \end{equation}
    By Definition \ref{def:define-omega}, $\dot{r}\geq (1-\frac{\varepsilon_{0}}{2})(r-1)>\frac{r}{2}$ and thus 
    \begin{equation}\label{eq:intersect-r'Z-Q_L}
            \dot{r}\Z^{2} \cap Q_{L}(\mathbf{0}) = \{(\dot{r} x,\dot{r} y): |x|, |y|\leq 2^{k},\text{ $x,y\in \Z$}\}.
    \end{equation}
    Hence by \eqref{eq:explicit-Q_L} and $\dot{r}<r-1$, $$Q_{L}(\mathbf{0})= \bigcup_{b \in \dot{r}\Z^2 \cap Q_{L}(\mathbf{0})} \{b' \in \Z^{2}:|b-b'|\leq \frac{r-1}{2}\}=\bigcup_{b \in \dot{r}\Z^2 \cap Q_{L}(\mathbf{0})} Q_{r}(b).$$
    Assume $r$-bit $Q_{r}(b') \not \subset Q_{L}(\mathbf{0})$, then $b'\not \in \dot{r}\Z^{2}\cap Q_{L}(\mathbf{0})$. Write $b'=(\dot{r} x',\dot{r} y')$ with $x',y'\in \Z$. By \eqref{eq:intersect-r'Z-Q_L}, without loss of generality, we assume $|x'|\geq 2^{k}+1$. By \eqref{eq:explicit-Q_L}, $$\inf_{b''\in Q_{L}(\mathbf{0})} |b' - b''|\geq \dot{r}-\frac{r-1}{2}\geq (1-\frac{\varepsilon_{0}}{2})(r-1)-\frac{r-1}{2}> \frac{(1-\frac{3}{2}\varepsilon_{0})r-1}{2}.$$ By Definition \ref{def:define-omega}, $\Tilde{\Omega}_{r}(b')\cap Q_{L}(\mathbf{0})=\emptyset$.
\end{proof}
We now extend the ``cutting procedure'' to $r$-dyadic boxes. It will be used in the proof of Proposition \ref{prop:Wegner}.
\begin{defn}\label{def:extended-cutting}
Given an $r$-dyadic box $Q_{L}(a)$ and $V:Q_{L}(a)\rightarrow \{0,\Bar{V}\}$,
let $\mathcal{R}$ be a subset of admissible $r$-bits inside $Q_{L}(a)$. For each $t\in [0,1]$, define $H_{Q_{L}(a)}^{\mathcal{R},t}:\ell^{2}(Q_{L}(a)) \rightarrow \ell^{2}(Q_{L}(a))$ as follows:
    $$\left\{\begin{array}{l}
    \text{$H_{Q_{L}(a)}^{\mathcal{R},t}(b,c)=t-1$ \;\;\;\;\;\;\;\;\;\;\;\;\; if $\{b,c\} \in \bigcup_{Q_{r}(a') \in \mathcal{R}} \partial S_{r}(a')$;}\\
    \text{$H_{Q_{L}(a)}^{\mathcal{R},t}(b,c)=H_{Q_{L}(a)}(b,c)$ \;\;\;otherwise.} 
    \end{array}\right.$$
%%%
Denote $G^{\mathcal{R},t}_{Q_{L}(a)}(b,c;\lambda)=(H^{\mathcal{R},t}_{Q_{L}(a)}-\lambda)^{-1}(b,c)$ for $b,c\in Q_{L}(a)$.
\end{defn}

\begin{defn}\label{theta_1}
For any large odd number $r$, denote $\Theta^{r}=\cup_{a \in \dot{r}\Z^2}F_{r}(a)$. For simplicity, we also denote $\Theta^{r}$ by $\Theta_{1}$ if $r$ is already given in context.
\end{defn}
The reason to define $\Theta^{r}$ is that, one only needs to know the value of $V$ on $\Theta^{r}$ to decide whether each $r$-bit is admissible or not. The sub-index of ``$\Theta_{1}$'' is for the consistency of notations in later multi-scale analysis Theorem \ref{thm:multi-2}.
\begin{defn}
Given an odd number $r$, an $r$-dyadic box $Q_{L}(a)$ and a potential function $V':\Theta_{1}\cap Q_{L}(a) \rightarrow \{0,\Bar{V}\}$, we say $Q_{L}(a)$ is \emph{perfect} if for any $r$-bit $Q_{r}(b)\subset Q_{L}(a)$, $(Q_{r}(b),V'|_{F_{r}(b)})$ is admissible.
\end{defn}
\begin{prop}\label{prop:prob-of-perfect}
Suppose odd number $r$ is large enough and $c_{1}$ is the constant in Proposition \ref{prop:r-admissible}. Given $r$-dyadic box $Q_{L}(a)$ with $L\leq \exp(c_{1} r)$, the event that \text{$Q_{L}(a)$ is perfect} only depends on $V|_{\Theta_{1}\cap Q_{L}(a)}$ and
\begin{equation}
    \Prob[\text{$Q_{L}(a)$ is perfect}] \geq 1-L^{-6}.
\end{equation}
\end{prop}

\begin{proof}
Since for each $r$-bit $Q_{r}(b) \subset Q_{L}(a)$, the event that it is admissible only depends on $V|_{F_{r}(b)}$, thus the event that \text{$Q_{L}(a)$ is perfect} only depends on $V|_{\Theta_{1}\cap Q_{L}(a)}$.

By Proposition \ref{prop:r-admissible}, we have
\begin{equation}
    \Prob[\text{$Q_{L}(a)$ is perfect}]\geq 1-L^{2}\exp(-8c_{1}r) \geq 1-L^{-6}
\end{equation}
since $L\leq \exp(c_{1} r)$.
\end{proof}
\begin{defn}
Given $S_{1},S_{2}\subset \Z^2$, we define $$\dist(S_{1},S_{2})= \inf_{a \in S_{1},b \in S_{2}} |a-b|.
$$
\end{defn}
We now prove the exponential decaying property of resolvent for perfect $r$-dyadic boxes. It will serve as the induction base case for the multi-scale analysis in Section \ref{sec:larger-scale}.
\begin{prop}\label{prop:L_1good}
Suppose odd number $r$ is large enough and $\Bar{V}>\exp(r^{2})$. If $V':\Theta_{1}\cap Q_{L}(a) \rightarrow \{0,\Bar{V}\}$ such that $Q_{L}(a)$ is a perfect $r$-dyadic box, then for any $V:Q_{L}(a)\rightarrow \{0,\Bar{V}\}$ with $V|_{\Theta_{1}\cap Q_{L}(a)}=V'$, any $\lambda_{0} \in [0,8]\setminus J_{r}^{\Bar{V}}$, any subset $\mathcal{R}$ of $r$-bits inside $Q_{L}(a)$, any $t \in [0,1]$, and each $b,c \in Q_{L}(a)$, we have
\begin{equation}\label{eq:G1}
    |G^{\mathcal{R},t}_{Q_{L}(a)}(b,c;\lambda_{0})| \leq \Bar{V}^{-\frac{|b-c|}{8r}+1}.
\end{equation}
\end{prop}
\begin{proof}
For simplicity of notations, we assume $a=\mathbf{0}$.

Let $g=\max_{b,c \in Q_{L}} |G^{\mathcal{R},t}_{Q_{L}}(b,c;\lambda_{0})|\Bar{V}^{\frac{|b-c|}{8r}}$ and
assume for some $b,b' \in Q_{L}$, $|G^{\mathcal{R},t}_{Q_{L}}(b,b';\lambda_{0})|\Bar{V}^{\frac{|b-b'|}{8r}}=g$.

Note that $Q_{L}=\bigcup_{a' \in \dot{r}\Z^2 \cap Q_{L}} Q_{r}(a')$ by Lemma \ref{lem:r_dyadic_cube-ele}. By Definition \ref{def:define-omega}, $\dot{r}=\left\lceil \left(1-\frac{\varepsilon_{0}}{2}\right)(r-1)\right\rceil$. By elementary geometry,
there is an $r$-bit $Q_{r}(c) \subset Q_{L}$ such that $b' \in Q_{r}(c)$ and $\dist(b',Q_{L} \setminus Q_{r}(c))\geq \frac{\varepsilon_{0} r}{7}$. By resolvent identity,
\begin{multline}\label{eq:solvent2}
    G^{\mathcal{R},t}_{Q_{L}}(b,b';\lambda_{0})\\
    =\sum_{\substack{b'' \in Q_{r}(c)\\ b'' \sim b'''\\ b''' \in Q_{L}\setminus Q_{r}(c)}} \widetilde{G_{Q_{r}(c)}}(b',b'';\lambda_{0})G^{\mathcal{R},t}_{Q_{L}}(b''',b;\lambda_{0}) + \mathbbm{1}_{b\in Q_{r}(c)} \widetilde{G_{Q_{r}(c)}}(b,b';\lambda_{0}).
\end{multline}
Here, $\widetilde{G_{Q_{r}(c)}}(b',b'';\lambda_{0})=G^{t}_{Q_{r}(c)}(b',b'';\lambda_{0})$ if $Q_{r}(c) \in \mathcal{R}$ and $\widetilde{G_{Q_{r}(c)}}(b',b'';\lambda_{0})=G_{Q_{r}(c)}(b',b'';\lambda_{0})$ otherwise.
Note that, if $b'' \sim b'''$ for some $b'' \in Q_{r}(c)$ and $b''' \in Q_{L}\setminus Q_{r}(c)$, then $|b'-b''|\geq \frac{\varepsilon_{0}r}{7} -1 \geq  \frac{\varepsilon_{0}r}{8}$.
In this case, by Proposition \ref{prop:0thscale}, $|\widetilde{G_{Q_{r}(c)}}(b',b'';\lambda_{0})| \leq \Bar{V}^{-\frac{1}{4}} \leq \Bar{V}^{-\frac{|b'-b''|}{4r}}$ since $|b'-b''|\leq r$ . Thus, we can estimate the first term in the right hand side of \eqref{eq:solvent2} by
\begin{align}
    &\left|\sum_{\substack{b'' \in Q_{r}(c)\\ b'' \sim b'''\\ b''' \in Q_{L}\setminus Q_{r}(c)}} \widetilde{G_{Q_{r}(c)}}(b',b'';\lambda_{0})G^{\mathcal{R},t}_{Q_{L}}(b''',b;\lambda_{0})\right|\\
    &\leq \sum_{\substack{b'' \in Q_{r}(c)\\ b'' \sim b'''\\ b''' \in Q_{L}\setminus Q_{r}(c)}}  g \Bar{V}^{-\frac{|b'-b''|}{4r}-\frac{|b-b'''|}{8r}}\\
    &<\frac{1}{2} \Bar{V}^{-\frac{|b-b'|}{8r}} g \label{mul:G1}
\end{align}
The second inequality is because, by triangle inequality with $|b''-b'''|=1$ and $|b'-b''| \geq \frac{\varepsilon_{0} r}{8}$,
\begin{align}
    \Bar{V}^{-\frac{|b'-b''|}{4r}-\frac{|b-b'''|}{8r}} 
    &\leq \Bar{V}^{-\frac{|b-b'|}{8r}+\frac{1}{8r}-\frac{\varepsilon_{0}}{64}} \\ \label{eq:the third one}
    &\leq  \exp(\frac{1}{8}r-\frac{\varepsilon_{0}}{64}r^{2}) \Bar{V}^{-\frac{|b-b'|}{8r}}\\
    &< \frac{1}{16}r^{-1}  \Bar{V}^{-\frac{|b-b'|}{8r}}
\end{align}
for large enough $r$, where \eqref{eq:the third one} is due to $\Bar{V}> \exp(r^{2})$. Since $$|G^{\mathcal{R},t}_{Q_{L}}(b,b';\lambda_{0})|=\Bar{V}^{-\frac{|b-b'|}{8r}} g ,$$ by \eqref{eq:solvent2} and \eqref{mul:G1}, we have
\begin{equation}
    g \leq \frac{1}{2} g + \Bar{V}^{\frac{|b-b'|}{8r}} \mathbbm{1}_{b\in Q_{r}(c)} |\widetilde{G_{Q_{r}(c)}}(b,b';\lambda_{0})|.
\end{equation}
If $b \not \in Q_{r}(c)$, then we have $g=0$. Otherwise, $|b-b'|\leq r$. By the first item in Proposition \ref{prop:0thscale}, $|\widetilde{G_{Q_{r}(c)}}(b,b';\lambda_{0})|\leq 2 \Bar{V}^{\frac{1}{4}}$ and thus
\begin{equation}
    g \leq 2 \Bar{V}^{\frac{|b-b'|}{8r}} |\widetilde{G_{Q_{r}(c)}}(b,b';\lambda_{0})| \leq 4 \Bar{V}^{\frac{3}{8}}<\Bar{V},
\end{equation}
which is equivalent to \eqref{eq:G1}.
\end{proof}

\section{Wegner Estimate}\label{sec:Wegner-estimate}
In this section we prove the Wegner estimate (Proposition \ref{prop:Wegner}) which will be used in the multi-scale analysis Theorem \ref{thm:multi-2}. In Section \ref{sec:auxiliary-lemmas}, we collect several lemmas on the unique continuation principle (Lemma \ref{lem:unique-continuation}), the eigenvalue variation (Lemma \ref{lem:eigenvariation} and Lemma \ref{lem:eigen-obstruction}), and the almost orthonormal vectors (Lemma \ref{lem:ortho-bdd}). In Section \ref{sec:Sperner-theorem}, a generalized Sperner lemma (Lemma \ref{lem:combinatorics}) for directed graph products is proved. All these lemmas will be used in Section \ref{sec:proof-Wegner} to prove the Wegner estimate Proposition \ref{prop:Wegner}.
\subsection{Auxiliary lemmas}\label{sec:auxiliary-lemmas}
We first need some geometry notations from \cite{ding2020localization}. The following Definition $\ref{def:geo1}$ to $\ref{def:geo4}$ are the same as Definition 3.1 to 3.4 in \cite{ding2020localization}.
\begin{defn}\label{def:geo1}
Given two subsets $I,J \subset \Z$, we define
\begin{equation}
    R_{I,J}=\{(x,y)\in \Z^2: \text{$x+y \in I$ and $x-y \in J$}\}.
\end{equation}
We call $R_{I,J}$ a \emph{tilted rectangle} if $I,J$ are intervals. A \emph{tilted square} $\Tilde{Q}$ is a tilted rectangle $R_{I,J}$ with $|I|=|J|$. With a little abuse of notations, we denote $\ell(\Tilde{Q})=|I|$ for a tilted square $\Tilde{Q}=R_{I,J}$.
\end{defn}
\begin{defn}\label{def:diagonals}
Given $k \in \Z$, we define the diagonals
\begin{equation}
    \mathcal{D}_{k}^{\pm}=\{(x,y)\in\Z^2:x\pm y=k\}.
\end{equation}
\end{defn}
\begin{defn}\label{def:sparsity}
Suppose $\Theta \subset \Z^2$, $\eta>0$ a density, and $R$ a tilted rectangle. Say that $\Theta$ is $(\eta,\pm)$-sparse in $R$ if 
\begin{equation}
    |\mathcal{D}_{k}^{\pm}\cap \Theta \cap R|\leq \eta |\mathcal{D}_{k}^{\pm} \cap R| \text{ for all diagonals $\mathcal{D}_{k}^{\pm}$}.
\end{equation}
We say that $\Theta$ is $\eta$-sparse in $R$ if it is both $(\eta,+)$-sparse and $(\eta,-)$-sparse in $R$.
\end{defn}

\begin{defn}\label{def:geo4}
A subset $\Theta \subset \Z^2$ is called $\eta$-regular in a set $E \subset \Z^2$ if $\sum_{k} |Q_{k}|\leq \eta |E|$ holds whenever $\Theta$ is not $\eta$-sparse in each of the disjoint tilted squares $Q_{1},Q_{2},\cdots,Q_{n}\subset E$. 
\end{defn}

The following Lemma \ref{lem:unique-continuation} and \ref{lem:findbox} are used to find an enough portion of the box where an eigenfunction has a decent lower bound.
In particular, Lemma \ref{lem:unique-continuation} is analog of \cite[Theorem 3.5]{ding2020localization} 
and its proof is given in Section \ref{sec:proof-Lemma-3.5}.
\begin{lemma}\label{lem:unique-continuation}
There exists numerical constant $0<\varepsilon_{1}<\frac{1}{100}$ such that the following holds. For every $\varepsilon\leq \varepsilon_{1}$, there is a large $\alpha>1$ depending on $\varepsilon$ such that, if
\begin{enumerate}
    \item $Q \subset\Z^2$ a box with $\ell(Q)\geq \alpha$
    \item $\Theta \subset Q$ is $\varepsilon$-regular in $Q$
    \item $\Bar{V}\geq 2$ and $\lambda_{0} \in [0,8]$
    \item $V': \Theta \rightarrow \{0,\Bar{V}\}$
    \item $\mathcal{E}^{\varepsilon,\alpha}_{uc}(Q,\Theta)$ denotes the event that
    \begin{equation}\label{eq:unique-conti}
            \left\{\begin{array}{l}|\lambda-\lambda_{0}|\leq (\ell(Q)\Bar{V})^{-\alpha \ell(Q)}\\\text{$H u= \lambda u$ in Q}\\ \text{$|u| \leq 1$ in a $1-\varepsilon^{3}$ fraction of $Q\setminus \Theta$} \end{array}\right. 
    \end{equation}
    implies $|u|\leq (\ell(Q)\Bar{V})^{\alpha\ell(Q)}$ in $\frac{1}{100}Q$,
\end{enumerate}
then $\Prob[\mathcal{E}^{\varepsilon,\alpha}_{uc}(Q,\Theta)\big| V|_{\Theta}=V'] \geq 1- \exp(-\varepsilon \ell(Q)^{\frac{2}{3}})$.
\end{lemma}
The following lemma is a rewrite of \cite[Lemma 5.3]{ding2020localization} and its proof is the same as the proof of \cite[Lemma 5.3]{ding2020localization}.
\begin{lemma}\label{lem:findbox}
For every integer $K\geq 1$, there exists $C_{K}>0$ depending on $K$ such that the following holds. If 
\begin{enumerate}
    \item $\Bar{V}\geq 2$ and $\lambda \in [0,8]$
    \item $L\geq C_{K}L' \geq L' \geq C_{K}$
    \item box $Q \subset \Z^2$ with $\ell(Q)=L$
    \item boxes $Q'_{k} \subset Q$ with $\ell(Q'_{k}) =L'$ for $k=1,2,\cdots,K$
    \item $H_{Q} u =\lambda u$,
\end{enumerate}
then,
\begin{equation}
    \|u\|_{\ell^{\infty}(Q')} \geq \Bar{V}^{-C_{K}L'}\|u\|_{\ell^{\infty}(Q)}
\end{equation}
holds for some $2 Q' \subset Q\setminus \cup_{k}Q'_{k}$ with $\ell(Q')=L'$.
\end{lemma}
\begin{defn}
Given a self-adjoint matrix $A$ and $\lambda \in \mathbb{R}$, we define $$n(A;\lambda)=\trace \mathbbm{1}_{(-\infty,\lambda)}(A).$$ i.e. $n(A;\lambda)$ is the number of $A$'s eigenvalues (with multiplicities) which are less than $\lambda$.
\end{defn}
The following Lemma \ref{lem:eigenvariation} and \ref{lem:eigen-obstruction} will provide a lower bound of the eigenvalue variation under a rank one perturbation of an operator. Lemma \ref{lem:eigenvariation} was proved in \cite[Lemma 5.1]{ding2020localization}.
\begin{lemma}\label{lem:eigenvariation}
Suppose real symmetric $n \times n$ matrix $A$ has eigenvalues $\lambda_{1}\leq \lambda_{2} \leq \cdots \leq \lambda_{n} \in \mathbb{R}$ with orthonormal eigenbasis $v_{1},v_{2},\cdots,v_{n} \in \mathbb{R}^{n}$. Let integers $k\in \{1,2,\cdots,n\}$ and $1\leq j\leq i\leq n-1$. If 
\begin{enumerate}
    \item $0<r_{1}<r_{2}<r_{3}<r_{4}<r_{5}<1$
    \item $r_{1} \leq C \min\{r_{3}r_{5},r_{2}r_{3}/r_{4}\}$
    \item $0<\lambda_{j}\leq \lambda_{i} <r_{1}<r_{2}<\lambda_{i+1} $
    \item $v^{2}_{j}(k)\geq r_{3}$
    \item $\sum_{r_{2}<\lambda_{l}<r_{5}} v^{2}_{l}(k) \leq r_{4}$
\end{enumerate}
then $n(A;r_{1})> n(A+tP_{k};r_{1})$ for $t\geq 1$, where $P_{k}$ is the projection operator defined by  $(P_{k} u)(i)=0$ if $i\not =k$ and $(P_{k} u)(k)=u(k)$ for any $u\in \R^{n}$.
\end{lemma}
\begin{proof}
\cite[Lemma 5.1]{ding2020localization} implies the conclusion for the case where $t=1$. The conclusion still holds for $t\geq 1$ by monotonicity.
\end{proof}

\begin{lemma}\label{lem:eigen-obstruction}
Let $k \in \{1,2,\cdots,n\}$ and $P_{k}$ be the projection operator defined by  $(P_{k} u)(i)=0$ if $i\not =k$ and $(P_{k} u)(k)=u(k)$ for any $u\in \R^{n}$. Suppose self-adjoint operator $A: \R^{n} \rightarrow \R^{n}$ has eigenvalues $\lambda_{1}\leq \lambda_{2} \leq \cdots \leq \lambda_{n} \in \mathbb{R}$ with orthonormal eigenbasis $v_{1},v_{2},\cdots,v_{n} \in \mathbb{R}^{n}$.

If $\lambda \not \in \sigma(A)$ and $\sum_{i=1}^{n} \frac{v_{i}(k)^{2}}{\lambda_{i}-\lambda}>0(<0)$, then $\lambda \not \in \sigma(A+t P_{k})$ for each $t >0(<0)$, respectively.
\end{lemma}
\begin{proof}
We only consider the case where $\sum_{i=1}^{n} \frac{v_{i}(k)^{2}}{\lambda_{i}-\lambda}>0$, the other case follows the same argument.

For each $t\in \R$, let $v^{t}_{1},v^{t}_{2},\cdots, v^{t}_{n}$ be the orthonormal eigenbasis of $A+t P_{k}$ with eigenvalues $\lambda_{1}^{t},\lambda_{2}^{t}\cdots,\lambda_{n}^{t}$. Then
the resolvent of $A+t P_{k}$ at $k$ with energy $\lambda_{0} \not\in \sigma(A+tP_{k})$ is 
\begin{equation}\label{eq:resolvent-at-t}
    G_{t}(k,k;\lambda_{0})=\left\langle \mathbbm{1}_{k}, (A+t P_{k}-\lambda_{0})^{-1} \mathbbm{1}_{k} \right\rangle =\sum_{i=1}^{n} \frac{v^{t}_{i}(k)^{2}}{\lambda^{t}_{i}-\lambda_{0}}. 
\end{equation}
Let $\mathrm{i}$ denote the imaginary unit. By resolvent identity, for each $t,\eta>0$, 
\begin{equation}\label{eq:resolvent-obstruction}
G_{t}(k,k;\lambda+\mathrm{i} \eta)=\frac{1}{t+G_{0}(k,k;\lambda+\mathrm{i}\eta)^{-1}}.    
\end{equation}
Since $G_{0}(k,k;\lambda)=\sum_{i=1}^{n} \frac{v_{i}(k)^{2}}{\lambda_{i}-\lambda}>0$, \eqref{eq:resolvent-obstruction} implies $\lim_{\eta \rightarrow 0}G_{t}(k,k;\lambda+\mathrm{i} \eta)>0$ for any $t>0$. Since $G_{t}(k,k;\lambda+\mathrm{i} \eta)=\sum_{i=1}^{n} \frac{v^{t}_{i}(k)^{2}}{\lambda^{t}_{i}-\lambda-\mathrm{i}\eta}$, we have
\begin{equation}\label{eq:finite-limit-green-function}
   \lim_{\eta\rightarrow 0} \sum_{i=1}^{n} \frac{v^{t}_{i}(k)^{2}}{\lambda^{t}_{i}-\lambda-\mathrm{i}\eta} < \infty.
\end{equation}
Assume $\lambda \in \sigma(A+t P_{k})$, then there exists $i_{0}$ with $\lambda^{t}_{i_{0}}=\lambda$. Equation \eqref{eq:finite-limit-green-function} implies $v^{t}_{i_{0}}(k)=0$. Since $(A + t P_{k}) v^{t}_{i_{0}}= \lambda^{t}_{i_{0}} v^{t}_{i_{0}}$, we have $A v^{t}_{i_{0}} = \lambda v^{t}_{i_{0}}$. This contradicts with $\lambda \not\in \sigma(A)$.
\end{proof}
We also need the following bound on the number of almost orthonormal vectors which was proved in \cite{ding2020localization}. A similar version of the following lemma was also proved in \cite{tao2013cheap}.
\begin{lemma}[Lemma 5.2 in \cite{ding2020localization}] \label{lem:ortho-bdd}
If $v_{1},\cdots,v_{m} \in \R^{n}$ satisfy $|\langle v_{i},v_{j} \rangle  - \delta_{i j}| \leq (5 n)^{-\frac{1}{2}}$, then $m \leq (5-\sqrt{5})n/2$.
\end{lemma}
\subsection{Sperner Lemma}\label{sec:Sperner-theorem}
We prove a generalization of \cite[Theorem 4.2]{ding2020localization} which will be used in an eigenvalue variation argument in the proof of Proposition \ref{prop:Wegner}. 
\begin{defn}\label{def:rho-Sperner}
Suppose $\rho\in (0,1]$. A set $\mathcal{A}$ of subsets of $\{1,2,\cdots,n\}$ is $\rho$-Sperner if, for every $A\in \mathcal{A}$, there is a set $B(A)\subset \{1,2,\cdots,n\}\setminus A$ such that $|B(A)|\geq \rho (n-|A|)$ and $A \subset A' \in \mathcal{A}$ implies $A' \cap B(A)=\emptyset$.
\end{defn}
The following lemma is proved in \cite[Theorem 4.2]{ding2020localization}.
\begin{lemma}[Theorem 4.2 in \cite{ding2020localization}] \label{lem:muti-Sperner}
If $\rho \in (0,1]$ and $\mathcal{A}$ is a $\rho$-Sperner set of subsets of $\{1,2,\cdots,n\}$, then 
$$|\mathcal{A}|\leq 2^{n} n^{-\frac{1}{2}} \rho^{-1}.
$$
\end{lemma}

\begin{defn}
Suppose $A=(T,E)$ is a simple directed graph (without multi-edges or self-loops) with vertex set $T$ and edge set $E$. For each $e\in E$, we denote by $e^{-}(e^{+})$ the starting (ending) vertex of $e$ respectively. i.e. $e=(e^{-},e^{+})$. For two $e_{1},e_{2}\in E$, we say $e_{1}$ and $e_{2}$ \emph{have no intersection} if $e_{1}^{\pm},e_{2}^{\pm}$ are four different vertices; otherwise, we say $e_{1}$ and $e_{2}$ \emph{have intersection}. 
\end{defn}
\begin{defn}
Given $k\in \Z_{+}$ and a simple directed graph $A=(T,E)$, $A$ is called \emph{$k$-colourable} if $E$ can be written as a disjoint union $E=\bigcup_{j=1}^{k}E_{j}$ such that for each $j \in \{1,\cdots,k\}$ and $e_{1}\not=e_{2}\in E_{j}$, $e_{1}$ and $e_{2}$ have no intersection.
\end{defn}
\begin{lemma}\label{lem:coloring-graph}
Suppose $A=(T,E)$ is a simple directed graph and $m\in \Z_{+}$. Assume for each $x\in T$, 
\begin{equation}\label{eq:degree-of-vertices}
    |\{e\in E: e^{+}=x\}\cup \{e\in E:e^{-}=x\}| \leq m.
\end{equation}
Then $A$ is $2m-1$-colourable. 
\end{lemma}
\begin{proof}
By \eqref{eq:degree-of-vertices}, each $e\in E$ has intersection with at most $2m-2$ other edges. Thus we can color the edges of $A$ by at most $2m-1$ colors such that any two edges with the same color have no intersection. 
\end{proof}
The following lemma is a generalization of \cite[Theorem 4.2]{ding2020localization} in the sense that \cite[Theorem 4.2]{ding2020localization} is the special case where each graph $A_{i}$ (see below) has two vertices and one directed edge.
\begin{lemma}\label{lem:combinatorics}
Given $N,k,K_{0}\in \Z_{+}$, suppose $A_{i}=(T_{i},E_{i})$ are simple directed graphs for $1\leq i\leq N$. Assume $A_{i}$ is $k$-colourable for each $1\leq i\leq N$.

Suppose $B \subset T_{1}\times T_{2}\times \cdots \times T_{N}$ satisfies the following:
\begin{enumerate}
    \item Each $\vec{x}=(x_{1},x_{2},\cdots,x_{N}) \in B$ is associated with $K_{0}$ indices $1\leq I_{1}(\vec{x})<I_{2}(\vec{x})<\cdots<I_{K_{0}}(\vec{x}) \leq N$ and $K_{0}$ edges $e_{j}(\vec{x})\in E_{I_{j}(\vec{x})}$ $(j=1,\cdots,K_{0})$ such that $e_{j}(\vec{x})^{-}=x_{I_{j}(\vec{x})}$ $(j=1,\cdots,K_{0})$.\label{item:first}
    \item $|B|> K_{0}^{-1} k^{2} N^{\frac{1}{2}} |T_{1}| |T_{2}| \cdots |T_{N}|$,\label{item:second}
\end{enumerate}
then there exist $\vec{x},\vec{y}\in B$ such that the following properties hold:
\begin{enumerate}
    \item[$(a)$] for each $i=1,2,\cdots,N$, either $x_{i}=y_{i}$ or $(x_{i},y_{i})\in E_{i}$,
    \item[$(b)$] there exists $j \in \{1,2,\cdots,K_{0}\}$ such that $(x_{I_{j}(\vec{x})},y_{I_{j}(\vec{x})})=e_{j}(\vec{x})$.
\end{enumerate}
\end{lemma}
\begin{proof}
Let us first consider an easier case where each $A_{i}$ consists of two vertices and one directed edge (thus we can assume $k=1$). Let $e_{i}$ denote the single directed edge in $A_{i}$ for $1\leq i \leq N$. Then there is a bijection between $T_{1}\times T_{2}\times \cdots \times T_{N}$ and the power set of $\{1,\cdots,N\}$:
\begin{equation}\label{eq:bijection_simple_case}
    \vec{x} \longmapsto Y_{\vec{x}} = \{i : 1\leq i \leq N, x_{i} = e_{i}^{+}\}.
\end{equation}
We prove the lemma by contradiction. We assume
there are no two elements in $B$ satisfying both $(a)$ and $(b)$. Note that in our case, for any $\vec{x},\vec{y} \in B$, condition $(a)$ is equivalent to $Y_{\vec{x}}\subset Y_{\vec{y}}$ and condition $(b)$ is equivalent to $Y_{\vec{y}}\cap \{I_{j}(\vec{x}):1\leq j\leq K_{0}\} \not= \emptyset$. Hence, for any $\vec{x},\vec{y} \in B$, $Y_{\vec{x}}\subset Y_{\vec{y}}$ implies $Y_{\vec{y}}\cap \{I_{j}(\vec{x}):1\leq j\leq K_{0}\} = \emptyset$. By Definition \ref{def:rho-Sperner}, $\{Y_{\vec{z}}: \vec{z} \in B\}$ is $K_{0}/N$-Sperner. By Lemma \ref{lem:muti-Sperner}, $|B| \leq  2^{N} N^{-\frac{1}{2}} K_{0}^{-1} = K_{0}^{-1} N^{\frac{1}{2}} |T_{1}| |T_{2}| \cdots |T_{N}|$ which contradicts with assumption \ref{item:second}.

Now we consider the general case and we first prove that we can assume $k=1$ without loss of generality.
By assumption \ref{item:second}, $$K_{0}^{-1} k^{2} N^{\frac{1}{2}} |T_{1}| |T_{2}| \cdots |T_{N}|<|B|\leq |T_{1}| |T_{2}| \cdots |T_{N}|,$$ thus we have $k^{2} N^{\frac{1}{2}}<K_{0}$. In particular, $k<K_{0}$.
\begin{cla}
We can assume $k=1$.
\end{cla}
\begin{proof}[Proof of the claim]
For each $i$, since $A_{i}$ is $k$-colourable,
we can write $E_{i}$ as a disjoint union $E_{i}=\bigcup_{m=1}^{k} E^{(m)}_{i}$ such that any two edges in $E^{(m)}_{i}$ have no intersection. For each $\vec{x} \in B$, by pigeonhole principle, there exists $m(\vec{x})\in \{1,2,\cdots,k\}$, such that  
$$\left|\left\{1\leq j\leq K_{0}:e_{j}(\vec{x})\in E_{I_{j}(\vec{x})}^{(m(\vec{x}))}\right\}\right|\geq \left\lceil \frac{K_{0}}{k} \right\rceil.$$
Since $B=\bigcup_{m=1}^{k}B_{m}$ with $B_{m}=\{\vec{x}\in B:m(\vec{x})=m\}$, by pigeonhole principle again, there exists $m' \in \{1,\cdots,k\}$ with $|B_{m'}|\geq \left\lceil \frac{1}{k} |B| \right\rceil$ and thus $$|B_{m'}|\geq \left\lceil \frac{K_{0}}{k} \right\rceil^{-1}  N^{\frac{1}{2}} |T_{1}| |T_{2}| \cdots |T_{N}|.
$$

We substitute $A_{i}=(T_{i},E_{i})$ by $A'_{i}=(T_{i},E_{i}^{(m')})$ for $i=1,\cdots,N$, substitute $B$ by $B_{m'}$, $K_{0}$ by $\left\lceil \frac{K_{0}}{k} \right\rceil$ and $k$ by $1$.
\end{proof}

Now we assume $k=1$. We prove the lemma by contradiction. We assume
\begin{equation}\label{assum:the-opposite}
    \textit{\parbox{.85\textwidth}{there are no two elements in $B$ satisfying both $(a)$ and $(b)$.}}
\end{equation}

Given $i\in \{1,\cdots,N\}$, we write $E_{i}=\{e_{i s}:s=1,\cdots,n_{i}\}$ and denote $T'_{i}=T_{i}\setminus \bigcup_{s=1}^{n_{i}} \{e_{i s}^{-},e_{i s} ^{+}\}$. For simplicity, denote $$\prod_{i=1}^{N} T_{i} = T_{1}\times T_{2}\times \cdots \times T_{N}.$$ Let $F_{i}=E_{i}\cup T'_{i}$ which consists of some edges and vertices. 
For each $\vec{f}=(f_{1},\cdots,f_{N}) \in F_{1}\times F_{2} \times \cdots \times F_{N}$, denote 
$$C_{\vec{f}}=\left\{\vec{x}\in \prod_{i=1}^{N} T_{i}:\text{$\forall 1\leq i\leq N$, $x_{i}=f_{i}$ if $f_{i}\in T'_{i}$; $x_{i}\in \left\{f_{i}^{-},f_{i}^{+}\right\}$ if $f_{i}\in E_{i}$}\right\}.$$
Then 
\begin{equation}\label{eq:decompose-T}
    T_{1}\times T_{2}\times \cdots \times T_{N}=\bigcup_{\vec{f}\in F_{1}\times F_{2} \times \cdots \times F_{N}} C_{\vec{f}}.
\end{equation}
For each $1\leq i\leq N$, since $A_{i}$ is $1$-colourable, any two edges in $E_{i}$ have no intersection. Thus the union in \eqref{eq:decompose-T} is a disjoint union. Since $|B|/(|T_{1}| |T_{2}| \cdots |T_{N}|)> K_{0}^{-1} N^{\frac{1}{2}} $, by pigeonhole principle again, there exists $\vec{f'}=(f'_{1},\cdots,f'_{N})\in F_{1}\times F_{2} \times \cdots \times F_{N}$ such that 
\begin{equation}\label{eq:estimate-B}
    |B\cap C_{\vec{f'}}|/|C_{\vec{f'}}|> K_{0}^{-1} N^{\frac{1}{2}}.
\end{equation}
Let $\mathcal{I}=\left\{1\leq i\leq N: f'_{i} \in E_{i}\right\}$ be the set of coordinates such that $f'_{i}$ is an edge. By assumption \ref{item:first}, for each $\vec{x}\in B\cap C_{\vec{f'}}$ and $j\in\{1,\cdots,K_{0}\}$, we have $e_{j}(\vec{x})=f'_{I_{j}(\vec{x})}$, $e_{j}^{-}(\vec{x})=x_{I_{j}(x)}$ and $I_{j}(\vec{x})\in \mathcal{I}$.
Denote $$Y_{\vec{x}}=\{i\in \mathcal{I}:x_{i}=(f'_{i})^{+}\}$$ for each $\vec{x}\in B\cap C_{\vec{f'}}$. Then $\vec{z}\mapsto Y_{\vec{z}}$ is an injection from $B\cap C_{\vec{f'}}$ to the power set of $\mathcal{I}$. Note that, the definition of set $Y_{\vec{x}}$ is analog to \eqref{eq:bijection_simple_case} except that we are now restricting on the subset $\mathcal{I}$.

We claim that $\{Y_{\vec{x}}:\vec{x}\in B\cap C_{\vec{f'}}\}$ is $K_{0}/|\mathcal{I}|$-Sperner as a set of subsets of $\mathcal{I}$. To see this,
suppose $Y_{\vec{x}}\subset Y_{\vec{y}}$ for some $\vec{x},\vec{y}\in B\cap C_{\vec{f'}}$. Then $\vec{x},\vec{y}$ satisfy property $(a)$. By assumption \eqref{assum:the-opposite}, $\vec{x},\vec{y}$ do not satisfy property $(b)$, thus $$Y_{\vec{y}} \cap \{I_{j}(\vec{x}):j=1,\cdots,K_{0}\}=\emptyset.$$
Since $\{I_{j}(\vec{x}):j=1,\cdots,K_{0}\}\subset \mathcal{I}\setminus Y_{\vec{x}}$, our claim follows from Definition \ref{def:rho-Sperner}.

Now Lemma \ref{lem:muti-Sperner} implies $$|B\cap C_{\vec{f'}}|=|\{Y_{\vec{x}}:\vec{x}\in B\cap C_{\vec{f'}}\}| \leq 2^{|\mathcal{I}|} K_{0}^{-1} |\mathcal{I}|^{\frac{1}{2}} \leq |C_{\vec{f'}}|  K_{0}^{-1} N^{\frac{1}{2}},$$ which contradicts with \eqref{eq:estimate-B}.
\end{proof}
\subsection{Proof of the Wegner estimate}\label{sec:proof-Wegner}
We now prove analog of the Wegner estimate \cite[Lemma 5.6]{ding2020localization}. 
\begin{prop}[Wegner estimate]\label{prop:Wegner}
Assume\\
(1) $\varepsilon>\delta>0$ are small and $c_{2}>0$ is a numerical constant\\
(2) integer $K\geq 1$, odd number $r>C_{\varepsilon,\delta,K}$ and real $\Bar{V}>\exp(r^{2})$\\
(3) $\lambda_{0} \not \in J_{r}^{\Bar{V}}$ which is defined in Definition \ref{def:J_eig}\\
(4) scales $R_{0}\geq R_{1}\geq \cdots \geq R_{6}\geq \exp(c_{2} r)$ with $R_{k}^{1-2\delta}\geq R_{k+1} \geq R_{k}^{1-\frac{1}{2}\varepsilon}$ and $R_{0},R_{3}$ are $r$-dyadic \\
(5) $Q\subset \Z^2$ an $r$-dyadic box with $\ell(Q)=R_{0}$\\
(6) $Q'_{1},\cdots,Q'_{K}\subset Q$ $r$-dyadic boxes, each with length $R_{3}$ (called ``defects'')\\
(7) $G \subset \cup_{k}Q'_{k}$ with $0<|G|\leq R_{0}^{\delta}$\\
(8) $\Theta\subset Q$ and $Q\setminus \Theta=\cup_{b \in D}\Omega_{r}(b)$ for some $D\subset \dot{r}\Z^2 \cap Q$\\
(9) $\Theta$ is $\varepsilon_{0}^{\frac{1}{5}}$-regular in every box $Q'\subset Q\setminus \cup_{k}Q'_{k}$ with $\ell(Q')=R_{6}$, where $\varepsilon_{0}$ is defined by \eqref{eq:def-of-epsilon0}\\
(10) potential $V':\Theta \rightarrow \{0,\Bar{V}\}$ satisfies the following: for any $V:Q \rightarrow \{0,\Bar{V}\}$ with $V|_{\Theta}=V'$, any $\lambda \in \left[\lambda_{0}- \Bar{V}^{-R_{5}},\lambda_{0}+ \Bar{V}^{-R_{5}}\right]$, any $t \in [0,1]$ and any subset $\mathcal{R}$ of $r$-bits that do not affect $\Theta\cup \bigcup_{k}Q'_{k}$, each $Q_{r}(b) \in \mathcal{R}$ is admissible and $H^{\mathcal{R},t}_{Q} u= \lambda u$ implies
    \begin{equation}\label{eq:control-wegner}
        \Bar{V}^{R_4}\|u\|_{\ell^2(Q\setminus \bigcup_{k}Q'_{k})} \leq \|u\|_{\ell^{2}(Q)}\leq (1+R_{0}^{-\delta})\|u\|_{\ell^2(G)}.
    \end{equation}
Then we have
  \begin{equation}
      \mathbb{P}\left[\|(H_{Q}-\lambda_{0})^{-1}\|\leq \Bar{V}^{R_1}\big| \; V|_{\Theta}=V'\right]\geq 1-R_{0}^{10\varepsilon-\frac{1}{2}}.
  \end{equation}

\end{prop}
As mentioned in Section \ref{sec:wegner-outline}, in order to prove the Wegner estimate, we need to prove the upper bound estimate and the lower bound estimate, conditioning on $V|_{\Theta}=V'$. In particular, the upper bound estimate is proved in Claim \ref{cla:total-number-bad-eigenvalues} and it provides a significantly smaller subset $\Lambda_{V}$ of eigenvalues (depending on potential $V$) such that eigenvalues outside $\Lambda_{V}$ have zero conditional probability to be close to $\lambda_{0}$. Thus we only need to consider eigenvalues in $\Lambda_{V}$. %An important remark here is that eigenvalues in $\Lambda$ have different indices for different potential functions but they correspond to the same set of eigenvalues after the cutting procedure.
The lower bound estimate is proved in Claim \ref{cla:wegner-bound} and it implies that any eigenvalue in $\Lambda_{V}$ can be perturbed to move away from $\lambda_{0}$ by changing the potential function on any vertex in a significant portion of the box. By combining this fact and the Sperner lemma (Lemma \ref{lem:combinatorics}), we prove a probability upper bound for the event that there is an eigenvalue in $\Lambda_{V}$ which is close to $\lambda_{0}$ and thus prove the Wegner estimate. 
\begin{proof}[Proof of Proposition \ref{prop:Wegner}]
Throughout the proof, we allow constants $C>1>c>0$ to depend on $\varepsilon$, $\delta$, $K$.
\begin{cla}
We can assume without loss of generality that $\cup_{k}Q'_{k} \subset \Theta$.
\end{cla}
\begin{proof}[Proof of the claim]
Let $\Theta'=\cup_{k}Q'_{k}\setminus \Theta$ and observe that for any event $\mathcal{E}$,
\begin{equation}
    \Prob\left[\mathcal{E} \big| V|_{\Theta}=V'\right]=\mathbb{E} \left[ \Prob\left[ \mathcal{E} \big| V|_{\Theta \cup \Theta'}=V'\cup V''\right] \right]
\end{equation}
where the expectation is taking over all $V'':\Theta' \rightarrow \{0,\Bar{V}\}$. Thus, it suffices to estimate the term in the expectation. Now we replace $\Theta$ by $\Theta \cup \Theta'$ and check the assumptions. Except for assumption $(8)$, other assumptions are straightforward. As for assumption $(8)$, note that $Q\setminus (\Theta\cup \Theta') = (Q\setminus\Theta)\setminus (\cup_{k}Q'_{k})$. For any $a \in \dot{r}\Z^2$, by Lemma \ref{lem:r_dyadic_cube-ele} and the assumption that $Q'_{k}$'s are $r$-dyadic, either $\Omega_{r}(a) \subset (\cup_{k}Q'_{k})$ or $\Omega_{r}(a) \cap (\cup_{k}Q'_{k}) = \emptyset$. Thus $Q\setminus (\Theta\cup\Theta')=\cup_{b \in D'}\Omega_{r}(b)$ where $D' = \{b\in D:\Omega_{r}(a) \cap (\cup_{k}Q'_{k}) = \emptyset\}$. The assumption follows.
\end{proof}
Now we assume $\cup_{k}Q'_{k} \subset \Theta$, then by Lemma \ref{lem:r_dyadic_cube-ele}, $\Tilde{\Omega}_{r}(b) \cap \left(\cup_{k}Q'_{k}\right)=\emptyset$ for each $b\in D$. 

We fix $\mathcal{R}=\{Q_{r}(b):b\in D\}$. By assumption $(10)$, when $Q_{r}(b) \in \mathcal{R}$, $(Q_{r}(b),V|_{F_{r}(b)})$ is admissible. For each $V:Q\rightarrow \{0,\Bar{V}\}$ with $V|_{\Theta}=V'$ and $t \in [0,1]$, denote all the eigenvalues of $H^{\mathcal{R},t}_{Q}$ by $$\lambda^{t}_{1}(V)\leq \lambda^{t}_{2}(V)\leq \cdots \leq \lambda^{t}_{R_{0}^{2}}(V).$$ In particular, $\lambda^{0}_{1}(V)\leq \lambda^{0}_{2}(V)\leq \cdots \leq \lambda^{0}_{R_{0}^{2}}(V)$ are all the eigenvalues of $H_{Q}$. Let $u_{V,k}(k=1,\cdots,R_{0}^{2})$ be an orthonormal eigenbasis such that for each $k$, 
$$H_{Q} u_{V,k}=\lambda_{k}^{0}(V) u_{V,k}.$$

Since $H^{\mathcal{R},1}_{Q}(x,y)=0$ whenever $\{x,y\} \in \bigcup_{b \in D}\partial S_{r}(b)$, we have
\begin{equation}\label{eq:cutting-free-sites1}
    H^{\mathcal{R},1}_{Q}=\bigoplus_{b \in D}H_{S_{r}(b)} \bigoplus H_{Q \setminus \left(\cup_{b \in D}S_{r}(b)\right)}.
\end{equation}
Here, we also used the fact that $S_{r}(b)\cap S_{r}(b')=\emptyset$ whenever $b \not= b' \in D$ (see Remark \ref{rem:geo-of-r-bit}).

Thus eigenvalues of $H^{\mathcal{R},1}_{Q}$ consist of eigenvalues of $H_{S_{r}(b)}(b \in D)$ and eigenvalues of $H_{Q \setminus \left(\cup_{b \in D}S_{r}(b)\right)}$. By item $1$ in Lemma \ref{lem:property-of-Sr}, $Q \setminus \left(\cup_{b \in D}S_{r}(b)\right)\subset \Theta$. Thus $H_{Q \setminus \left(\cup_{b \in D}S_{r}(b)\right)}$ only depends on $V|_{\Theta}=V'$.
Let $\lambda_{1} \leq \lambda_{2} \leq \cdots \leq \lambda_{m}$ be all the eigenvalues of $H_{Q \setminus \left(\cup_{b \in D}S_{r}(b)\right)}$. Let $\lambda_{q} \leq \lambda_{q+1} \leq \cdots \leq \lambda_{q+p}$ be all the eigenvalues of $H_{Q \setminus \left(\cup_{b \in D}S_{r}(b)\right)}$ inside the closed interval $[\lambda_{0}-\Bar{V}^{-R_{4}},\lambda_{0}+ \Bar{V}^{-R_{4}}]$. Then $$\lambda_{q-1}<\lambda_{0}-\Bar{V}^{-R_{4}}$$ if $q>1$. Denote $$i(V)=|\{k: \lambda^{1}_{k}(V)< \lambda_{0}-\Bar{V}^{-R_{4}}\}|+1=n(H^{\mathcal{R},1}_{Q};\lambda_{0}-\Bar{V}^{-R_{4}})+1.$$ 
Because $\lambda_{0} \not \in J_{r}^{\Bar{V}}$, by item $2$ in Proposition \ref{prop:0thscale}, any eigenvalue of $H_{S_{r}(b)}(b \in D)$ is outside the interval $[\lambda_{0}-\frac{1}{2}\Bar{V}^{-\frac{1}{4}},\lambda_{0}+\frac{1}{2}\Bar{V}^{-\frac{1}{4}}]$. Thus by \eqref{eq:cutting-free-sites1},
\begin{equation}\label{eq:eigenvalue-on-frozen-set}
\begin{split}
    &\sigma(H^{\mathcal{R},1}_{Q})\cap [\lambda_{0}-\Bar{V}^{-R_{4}},\lambda_{0}+ \Bar{V}^{-R_{4}}]\\
    &=\sigma(H_{Q \setminus \left(\cup_{b \in D}S_{r}(b)\right)})\cap [\lambda_{0}-\Bar{V}^{-R_{4}},\lambda_{0}+ \Bar{V}^{-R_{4}}]\\
    &=\{\lambda_{q+j}:0\leq j \leq p\}\\
    &=\{\lambda^{1}_{i(V)+j}(V): 0\leq j \leq p\}. 
\end{split}
\end{equation}
    
\begin{cla}\label{cla:bound-p}
    $p \leq C R_{0}^{\delta}$.
\end{cla}
\begin{proof}[Proof of the claim]
Let $$\left\{v_{i} \in \ell^{2}\big(Q \setminus \left(\cup_{b \in D}S_{r}(b)\right)\big):0\leq i \leq p\right\}$$ be an orthonormal set with $H_{Q \setminus \left(\cup_{b \in D}S_{r}(b)\right)} v_{i}=\lambda_{q+i} v_{i}$ for each $0\leq i\leq p$. Consider the function $v'_{i}$ on $Q$ defined by $v'_{i}|_{\cup_{b \in D}S_{r}(b)}=0$ and $v'_{i}|_{Q \setminus \left(\cup_{b \in D}S_{r}(b)\right)}=v_{i}$. By \eqref{eq:cutting-free-sites1}, $v'_{i}$ is an eigenfunction of $H^{\mathcal{R},1}_{Q}$ with the eigenvalue $\lambda_{q+i}$. By assumption $(10)$, $\|v'_{i}\|_{\ell^{2}(G)} \geq (1+R_{0}^{-\delta})^{-1}\geq 1-R_{0}^{-\delta}$. From $\langle v'_{i},v'_{j} \rangle_{\ell^{2}(Q)}=\delta_{i j}$ we deduce that $$|\langle v'_{i},v'_{j} \rangle_{\ell^{2}(G)}-\delta_{i j}| \leq R_{0}^{-\delta}\leq (5 |G|)^{-\frac{1}{2}}.$$
Thus, $\left\{v'_{i}|_{G}:0 \leq i \leq p\right\}$ is a set of almost orthogonal vectors and Lemma \ref{lem:ortho-bdd} implies the claim.
\end{proof}
\begin{cla}\label{cla:total-number-bad-eigenvalues}
Suppose $\lambda^{0}_{k}(V) \in [\lambda_{0}-\Bar{V}^{-R_{2}},\lambda_{0}+ \Bar{V}^{-R_{2}}]$ for some $1\leq k\leq R_{0}^{2}$. Then
there exists $j \in \{0,1,\cdots,p\}$ such that $k=i(V)+j$.
\end{cla}
\begin{proof}[Proof of the claim]
Fix such $V$ and for simplicity, when $t\in [0,1]$ we denote $\lambda_{k}^{t}=\lambda_{k}^{t}(V)$ and choose $u_{k}^{t}$ to be an $\ell^{2}$-normalised eigenfunction of $H_{Q}^{\mathcal{R},t}$ with eigenvalue $\lambda_{k}^{t}$. Denote $X=\cup_{b \in D} \partial S_{r}(b)$. The first order variation implies (see \cite[Chapter 2, Section 6.5]{kato2013perturbation})
\begin{equation}\label{eq:first-order-variation}
    |\lambda_{k}^{t}-\lambda_{k}^{0}|=\left|\int_{0}^{t} \sum_{ \substack{x \sim y\\ \{x,y\} \in X} } u^{s}_{k}(x) u^{s}_{k}(y) ds\right|.
\end{equation}
By Lemma \ref{lem:r_dyadic_cube-ele},
$\bigcup_{b\in D} \Tilde{\Omega}_{r}(b)\cap \left(\cup_{k}Q'_{k}\right)=\emptyset$.
Since $X \subset \bigcup_{b\in D}\Tilde{\Omega}_{r}(b)$, assumption $(10)$ and equation \eqref{eq:control-wegner} imply 
$$|\int_{0}^{t} \sum_{ \substack{x \sim y\\ \{x,y\} \in X} } u^{s}_{k}(x) u^{s}_{k}(y) ds| \leq 2t |X| \Bar{V}^{-2 R_{4}}  \leq 4 t  R_{0}^{2} \Bar{V}^{-2 R_{4}}<\frac{1}{2}\Bar{V}^{-R_{4}}$$ 
as long as $|\lambda_{k}^{t}-\lambda_{0}|\leq \Bar{V}^{-R_{5}}$. Thus \eqref{eq:first-order-variation} implies
\begin{equation}\label{eq:lambda-not-moving-far}
    |\lambda_{k}^{t}-\lambda_{0}|\leq \Bar{V}^{-R_{2}}+ \frac{1}{2}\Bar{V}^{-R_{4}} + 4R_{0}^{2} \mathbbm{1}_{\max_{0\leq s \leq t}|\lambda_{k}^{s}-\lambda_{0}| \geq \Bar{V}^{-R_{5}}}.  
\end{equation}
Since $\lambda_{k}^{t}$ is continuous with respect to $t$, by continuity, \eqref{eq:lambda-not-moving-far} implies $|\lambda^{t}_{k}-\lambda_{0}|\leq \Bar{V}^{-R_{4}}$ for each $t \in [0,1]$. In particular, $|\lambda^{1}_{k}-\lambda_{0}|\leq \Bar{V}^{-R_{4}}$. Thus by $\eqref{eq:eigenvalue-on-frozen-set}$, $k=i(V)+j$ for some $j \in \{0,1,\cdots,p\}$.
\end{proof}
By Claim \ref{cla:total-number-bad-eigenvalues}, we only need to consider eigenvalues in set $\{\lambda_{i(V)+k}:0\leq k \leq p\}$. We will prove that, with high probability, these $(p+1)$ eigenvalues are away from $\lambda_{0}$ with distance at least $\Bar{V}^{-R_{1}}$. We first prove that each of these $(p+1)$ eigenvalues' eigenfunctions has a large support (Claim \ref{cla:unique-continuation-bound}). Then we use these supports of eigenfunctions to do an eigenvalue perturbation argument which, combined with the Sperner lemma, proves that eigenvalues in $\{\lambda_{i(V)+k}:0\leq k \leq p\}$ are away from $\lambda_{0}$ with high probability (Claim \ref{cla:wegner-bound}).
\begin{cla}\label{cla:unique-continuation-bound}
$\Prob\left[\mathcal{E}_{uc} \big| V|_{\Theta}=V'\right]\geq 1-\exp(-R_{0}^{\varepsilon})$, where $\mathcal{E}_{uc}$ denotes the event that $$\left|\left\{a \in Q:|u(a)|\geq \Bar{V}^{-\frac{1}{2}R_{2}}\|u\|_{2}\right\}\setminus \Theta \right|\geq R_{4}^{\frac{3}{2}}
$$ holds whenever $|\lambda-\lambda_{0}|\leq \Bar{V}^{-R_{5}}$ and $H_{Q} u= \lambda u$.
\end{cla}
\begin{proof}[Proof of the claim]
Our strategy here is that we first use Lemma \ref{lem:findbox} to find a vertex $a_{*}$ with $|u(a_{*})|$ being lower bounded, then use the unique continuation principle Lemma \ref{lem:unique-continuation} to find $R_{4}^{\frac{3}{2}}$ vertices in $Q_{R_{6}}(a_{*})$.

Recall the definition of $\mathcal{E}^{\varepsilon,\alpha}_{uc}(Q,\Theta)$ in Lemma \ref{lem:unique-continuation} and that equation \eqref{eq:def-of-epsilon0} implies $\varepsilon_{0}^{\frac{1}{5}} < \varepsilon_{1}$. By Lemma \ref{lem:unique-continuation} and assumption $(9)$, there exists $\alpha'>1$ such that the event 
$$\mathcal{E}'_{uc}=\bigcap_{\substack{Q' \subset Q\setminus \cup_{k}Q'_{k}\\ \ell(Q')=R_{6}}} \mathcal{E}^{\varepsilon_{0}^{\frac{1}{5}},\alpha'}_{uc}(Q',\Theta \cap Q')
$$ satisfies 
\begin{equation}
    \Prob\left[ \mathcal{E}'_{uc} \big| V_{\Theta}=V' \right] \geq 1-\exp(-\varepsilon_{0}^{\frac{1}{5}} R_{6}^{\frac{2}{3}}+C\log(R_{0})) \geq 1-\exp(-R_{0}^{\varepsilon}).
\end{equation}
Thus it suffices to show $\mathcal{E}'_{uc} \subset \mathcal{E}_{uc}$. Suppose $\mathcal{E}'_{uc}$ holds, $|\lambda-\lambda_{0}|\leq \Bar{V}^{-R_{5}}$, and $H_{Q}u =\lambda u$. 

Lemma \ref{lem:findbox} provides an $R_{3}$-box $Q_{*}$ with $Q_{*} \subset Q\setminus \cup_{k}Q'_{k}$ and $a_{*}\in \frac{1}{2}Q_{*}$ such that,
\begin{equation}
    |u(a_{*})| \geq \Bar{V}^{-C_{K}R_{3}} \|u\|_{\ell^{\infty}(Q)} \geq \Bar{V}^{-C'_{K}R_{3}} \|u\|_{\ell^{2}(Q)}.
\end{equation}
Since $\mathcal{E}^{\varepsilon_{0}^{\frac{1}{5}},\alpha'}_{uc}(Q_{R_{6}}(a_{*}),\Theta \cap Q_{R_{6}}(a_{*}))$ holds and $$|\lambda-\lambda_{0}|\leq \Bar{V}^{-R_{5}}\leq  (R_{6}\Bar{V})^{-\alpha' R_{6}},$$ we see that 
\begin{equation}
    |\{|u|\geq (R_{6}\Bar{V})^{-\alpha' R_{6}}|u(a_{*})| \} \cap Q_{R_{6}}(a_{*}) \setminus \Theta| \geq \frac{1}{2} \varepsilon_{0}^{\frac{3}{5}} R_{6}^{2}.
\end{equation}
Here, recall that for any real function $u'$ defined on a set $W$ and any real number $c$, we use $\{u'\geq c\}$ as shorthand for the set $\{a \in W: u'(a) \geq c\}$.

Thus by taking $r> C_{\varepsilon,\delta,K}$ large and observing $R_{6}\geq \exp(c_{2}r)$, we have
\begin{equation}\label{eq:unique-continuation2}
    |\{|u|\geq \Bar{V}^{-\frac{1}{2}R_{2}}\|u\|_{\ell^{2}(Q)} \}\cap Q \setminus \Theta| \geq \frac{1}{2} \varepsilon_{0}^{\frac{3}{5}} R_{6}^{2}\geq R_{4}^{\frac{3}{2}}.
\end{equation}
\eqref{eq:unique-continuation2} provides the inclusion and the claim.
\end{proof}
\begin{cla}\label{cla:wegner-bound}
    For $0 \leq k_1 \leq k_2 \leq p$ and $0 \leq \ell \leq C R_{0}^{\delta}$, we have
    \begin{equation}
        \mathbb{P}\left[\mathcal{E}_{k_1,k_2,\ell} \cap \mathcal{E}_{uc} \big|\; V|_{\Theta} = V' \right] \leq Cr^{6} R_{0} R_{4}^{-\frac{3}{2}} 
    \end{equation}
    where $\mathcal{E}_{k_1,k_2,\ell}$ denotes the event 
    \begin{align}
        &|\lambda^{0}_{i(V)+k_1}(V)-\lambda_{0}|,|\lambda^{0}_{i(V)+k_2}(V)-\lambda_{0}|< s_{\ell}\;\; and\\  
        &|\lambda^{0}_{i(V)+k_1 -1}(V)-\lambda_{0}|,|\lambda^{0}_{i(V)+k_2+1}(V)-\lambda_{0}|\geq s_{\ell+1},
    \end{align}
    where $s_{i}:=\Bar{V}^{-R_{1}+(R_2-\frac{1}{2} R_4+C)i}$ for each $i \in \Z$.
\end{cla}
\begin{proof}
Conditioning on $V|_{\Theta}=V'$, we view events $\mathcal{E}_{uc}$ and $\mathcal{E}_{k_1,k_2,\ell}$ as subsets of $\{0,\Bar{V}\}^{\cup_{b \in D} \Omega_{r}(b)}$. 
Given $\tau \in \{0,1\}$, we denote by $\mathcal{E}_{k_1,k_2,\ell,\tau}$ the intersection of $\mathcal{E}_{k_1,k_2,\ell}$ and the event
\begin{equation}\label{eq:number-points}
    |\{a'\in Q\setminus \Theta:|u_{V,i(V)+k_1}(a')|\geq \Bar{V}^{-\frac{1}{2}R_{2}} \text{ and } V(a')=\tau \Bar{V}\}|
    \geq \frac{1}{2} R_{4}^{\frac{3}{2}}.
\end{equation}
Then $$ \mathcal{E}_{k_1,k_2,\ell}\cap \mathcal{E}_{uc} \subset \mathcal{E}_{k_1,k_2,\ell,0}\cup \mathcal{E}_{k_1,k_2,\ell,1}.
$$

It suffices to prove that
\begin{equation}\label{eq:estimate-tau}
        \mathbb{P}\left[\mathcal{E}_{k_1,k_2,\ell,\tau} \big|\; V|_{\Theta} = V' \right] \leq 200 r^{6} R_{0} R_{4}^{-\frac{3}{2}}
    \end{equation}
for each $\tau \in \{0,1\}$.

We prove it for $\tau=0$, the case where $\tau=1$ is by symmetry. We prove by contradiction, assume \eqref{eq:estimate-tau} does not hold for $\tau=0$. That is,
\begin{equation}\label{eq:estimate-0-opposite}
        \mathbb{P}\left[\mathcal{E}_{k_1,k_2,\ell,0} \big|\; V|_{\Theta} = V' \right] > 200 r^{6} R_{0} R_{4}^{-\frac{3}{2}}.
    \end{equation}

Given $V \in \mathcal{E}_{k_1,k_2,\ell,0}$ with $V|_{\Theta}=V'$ and $a \in \Omega_{r}(b)$ with some $b\in D$, we say $a$ is a \emph{``crossing'' site} with respect to $V$ (or \emph{w.r.t.} $V$) if $V(a)=0$ and $$n\big((-\Delta+V+\Bar{V} \delta_{a})_{S_{r}(b)};\lambda_{0}\big)=n\big((-\Delta+V)_{S_{r}(b)};\lambda_{0}\big)-1;$$ we say $a$ is a \emph{``non-crossing'' site} with respect to $V$ (or w.r.t. $V$) if $V(a)=0$ and 
$$n\big((-\Delta+V+\Bar{V} \delta_{a})_{S_{r}(b)};\lambda_{0}\big)=n\big((-\Delta+V)_{S_{r}(b)};\lambda_{0}\big).$$
Note that by rank one perturbation, for any $a \in Q\setminus \Theta$ with $V(a)=0$, either $a$ is a crossing site w.r.t. $V$ or $a$ is a non-crossing site w.r.t. $V$.

Denote by $\mathcal{E}_{k_1,k_2,\ell,0,cro}$ the intersection of $\mathcal{E}_{k_1,k_2,\ell,0}$ and the event
$$|\{|u_{V,i(V)+k_1}|\geq \Bar{V}^{-\frac{1}{2}R_{2}}\}\cap \{a'\in Q\setminus \Theta :\text{$a'$ is a crossing site w.r.t. $V$}\} |
    \geq \frac{1}{4} R_{4}^{\frac{3}{2}}.
$$

Denote by $\mathcal{E}_{k_1,k_2,\ell,0,ncr}$ the intersection of $\mathcal{E}_{k_1,k_2,\ell,0}$ and the event
$$|\{|u_{V,i(V)+k_1}|\geq \Bar{V}^{-\frac{1}{2}R_{2}}\}\cap \{a'\in Q\setminus \Theta :\text{$a'$ is a non-crossing site w.r.t. $V$}\} |
    \geq \frac{1}{4} R_{4}^{\frac{3}{2}}.
$$
Then by \eqref{eq:number-points}, $$\mathcal{E}_{k_1,k_2,\ell,0} \subset \mathcal{E}_{k_1,k_2,\ell,0,cro} \cup \mathcal{E}_{k_1,k_2,\ell,0,ncr}.
$$
By \eqref{eq:estimate-0-opposite}, one of the following holds:
\begin{equation}\label{eq:estimate-crossing-site}
        \mathbb{P}\left[\mathcal{E}_{k_1,k_2,\ell,0,cro} \big|\; V|_{\Theta} = V' \right] > 100r^{6} R_{0} R_{4}^{-\frac{3}{2}},
    \end{equation}
    or
\begin{equation}\label{eq:estimate-non-crossing-site}
        \mathbb{P}\left[\mathcal{E}_{k_1,k_2,\ell,0,ncr} \big|\; V|_{\Theta} = V' \right] > 100r^{6} R_{0} R_{4}^{-\frac{3}{2}}.
    \end{equation}

We will arrive at a contradiction in each case.

\noindent {\bf Case 1.} \eqref{eq:estimate-crossing-site} holds.

For each $b\in D$, 
we define a directed graph $A_{b}=(T_{b},E_{b})$ with vertex set $T_{b}=\{0,\Bar{V}\}^{\Omega_{r}(b)}$, and the edge set $E_{b}$ is defined as follows. For each $w \in T_{b}$, let $\widetilde{w} \in \{0,\Bar{V}\}^{S_{r}(b)}$ be $\widetilde{w}=w$ in $\Omega_{r}(b)$ and $\widetilde{w}=V'$ in $S_{r}(b)\setminus \Omega_{r}(b)$.  
Given $w_{1},w_{2}\in T_{b}$, 
there is an edge which starts from $w_{1}$ and ends at $w_{2}$ if $w_{2}=w_{1}+\Bar{V}\delta_{b'}$ for some $b'\in \Omega_{r}(b)$ and $n\big( (-\Delta)_{S_{r}(b)}+\widetilde{w_{2}};\lambda_{0}\big)=n\big( (-\Delta)_{S_{r}(b)}+\widetilde{w_{1}};\lambda_{0}\big)-1$. 

For each $w \in T_{b}$, there are less than $2 r^{2}$ edges which start from or end at $w$. By Lemma \ref{lem:coloring-graph}, $A_{b}$ is $4r^{2}$-colourable.

For each $V \in \mathcal{E}_{k_1,k_2,\ell,0,cro}\cap\{V:V|_{\Theta}=V'\}$, by pigeonhole principle,
we can find a subset $D_{0}(V)\subset D$ with $|D_{0}(V)|= \lceil\frac{1}{4}r^{-2} R_{4}^{\frac{3}{2}}\rceil$ such that for each $b\in D_{0}(V)$, there is a crossing site $b' \in \Omega_{r}(b)$ w.r.t. $V$ with $|u_{V,i(V)+k_1}(b')|\geq \Bar{V}^{-\frac{1}{2}R_{2}}$. This provides, for each $b\in D_{0}(V)$, an edge $e_{b}(V) \in E_{b}$ with $e_{b}(V)^{-}=V|_{\Omega_{r}(b)}$, $e_{b}(V)^{+}=V|_{\Omega_{r}(b)}+\Bar{V}\delta_{b'}$ and $|u_{V,i(V)+k_1}(b')|\geq \Bar{V}^{-\frac{1}{2}R_{2}}$. We use Lemma \ref{lem:combinatorics} with directed graphs $A_{b}=(T_{b},E_{b})$ $(b\in D)$, subset $B=\mathcal{E}_{k_1,k_2,\ell,0,cro} \subset \prod_{b\in D} T_{b}$, $N=|D|\leq R_{0}^{2}$, $K_{0}= \lceil\frac{1}{4}r^{-2} R_{4}^{\frac{3}{2}}\rceil$, $k=4 r^{2}$, associated index set $D_{0}(V)$ and edge set $\{e_{b}(V):b\in D_{0}(V)\}$ for each $V\in B$. Here, equation \eqref{eq:estimate-non-crossing-site} serves as assumption \ref{item:second} in Lemma \ref{lem:combinatorics}. Lemma \ref{lem:combinatorics} provides $V_{1},V_{2}\in \mathcal{E}_{k_1,k_2,\ell}$ such that the following holds:
\begin{itemize}
    \item $\forall b \in D$, either $V_{1}|_{Q_{r}(b)}=V_{2}|_{Q_{r}(b)}$ or $V_{2}|_{Q_{r}(b)}=V_{1}|_{Q_{r}(b)}+\Bar{V} \delta_{b'}$ for some crossing site $b'$ w.r.t. $V_{1}$.
    \item There exists a crossing site $a_{0}\in Q\setminus \Theta$ w.r.t. $V_{1}$ such that $V_{2}(a_{0})=\Bar{V}$ and $|u_{V_{1},i(V_{1})+k_{1}}(a_{0})| \geq \Bar{V}^{-\frac{1}{2}R_{2}}$.
\end{itemize}
Denote $V_{3}=V_{1}+\Bar{V}\delta_{a_{0}}$. 
Then by definition of crossing site and \eqref{eq:cutting-free-sites1}, $i(V_{3})=i(V_{1})-1$ and
$$i(V_{2})=i(V_{1})-|\{a \in Q: V_{1}(a) \not = V_{2}(a)\}|.$$ 
By Cauchy interlacing theorem and the fact that $|\{a \in Q: V_{2}(a) \not = V_{3}(a)\}| = i(V_{3}) - i(V_{2})$, we have
\begin{equation}\label{eq:cauchy-interlacing}
    \lambda^{0}_{i(V_{1})+k_{1}}(V_{1})\geq \lambda^{0}_{i(V_{3})+k_{1}}(V_{3})\geq \lambda^{0}_{i(V_{2})+k_{1}}(V_{2}).
\end{equation}
By assumption $(10)$, for each $1\leq j\leq R_{0}^{2}$, we have $|u_{V_{1},j}(a_{0})|\leq \Bar{V}^{-R_{4}}$ when $|\lambda^{0}_{j}(V_{1})-\lambda_{0}|\leq \Bar{V}^{-R_{5}}$.
Since $|u_{V_{1},i(V_{1})+k_{1}}(a_{0})|\geq \Bar{V}^{-\frac{1}{2}R_{2}}$ and  $$\lambda_{0}-s_{\ell}<\lambda^{0}_{i(V_{1})+k_{1}}(V_{1})<\lambda_{0}+ s_{\ell},$$
we have
\begin{align}
\begin{split}\label{eq:condition-to-apply-3.9}
    &\sum_{j=1}^{R_{0}^{2}} \frac{u_{V_{1},j}(a_{0})^{2}}{\lambda^{0}_{j}(V_{1})-(\lambda_{0}-s_{\ell})}\\=
    &\sum_{j = i(V_{1}) + k_{1}}^{i(V_{1}) + k_{2}} \frac{u_{V_{1},j}(a_{0})^{2}}{\lambda^{0}_{j}(V_{1})-(\lambda_{0}-s_{\ell})} +
    \sum_{\substack{j\not \in[ i(V_{1}) + k_{1}, i(V_{1}) + k_{2}]\\|\lambda^{0}_{j}(V_{1})-\lambda_{0}|\leq \Bar{V}^{-R_{5}}}} \frac{u_{V_{1},j}(a_{0})^{2}}{\lambda^{0}_{j}(V_{1})-(\lambda_{0}-s_{\ell})}\\+
    &\sum_{|\lambda^{0}_{j}(V_{1})-\lambda_{0}|> \Bar{V}^{-R_{5}}}
    \frac{u_{V_{1},j}(a_{0})^{2}}{\lambda^{0}_{j}(V_{1})-(\lambda_{0}-s_{\ell})}\\
    \geq &\frac{u_{V_{1},i(V_{1})+k_{1}}(a_{0})^{2}}{\lambda^{0}_{i(V_{1})+k_{1}}(V_{1}) - (\lambda_{0} - s_{\ell})}-
    \sum_{\substack{j\not \in[ i(V_{1}) + k_{1}, i(V_{1}) + k_{2}]\\|\lambda^{0}_{j}(V_{1})-\lambda_{0}|\leq \Bar{V}^{-R_{5}}}} \frac{\Bar{V}^{-2R_{4}}}{s_{\ell+1}-s_{\ell}}\\
    -&\sum_{|\lambda^{0}_{j}(V_{1})-\lambda_{0}|> \Bar{V}^{-R_{5}}} 2 \Bar{V}^{R_{5}} \\\geq &\frac{\Bar{V}^{- R_{2}}}{2s_{\ell}}- R_{0}^{2} \frac{\Bar{V}^{-2R_{4}}}{s_{\ell+1}-s_{\ell}}-2 R_{0}^{2} \Bar{V}^{R_{5}} \\>&0.
\end{split}
\end{align}
Note that $$\lambda^{0}_{i(V_{3})+k_{1}}(V_{1})= \lambda^{0}_{i(V_{1})+k_{1}-1}(V_{1})< \lambda_{0}-s_{\ell}< \lambda^{0}_{i(V_{1})+k_{1}}(V_{1}).$$ 
By Lemma \ref{lem:eigen-obstruction} and \eqref{eq:condition-to-apply-3.9}, 
$\lambda^{0}_{i(V_{3})+k_{1}}(V_{1}+t\delta_{a_{0}})<\lambda_{0}-s_{\ell}$ for each $t>0$.
Let $t=\Bar{V}$, we have $ \lambda^{0}_{i(V_{3})+k_{1}}(V_{3})<\lambda_{0}-s_{\ell}$. Thus by \eqref{eq:cauchy-interlacing}, $\lambda^{0}_{i(V_{2})+k_{1}}(V_{2})<\lambda_{0}-s_{\ell}$ and hence $V_{2} \not \in \mathcal{E}_{k_{1},k_{2},\ell}$. We thus arrived at a contradiction.

\noindent {\bf Case 2.} \eqref{eq:estimate-non-crossing-site} holds.

For each $b\in D$, 
we define a directed graph $A_{b}=(T_{b},E_{b})$ with vertex set $T_{b}=\{0,\Bar{V}\}^{\Omega_{r}(b)}$, and edge set $E_{b}$ is defined as follows. For each $w \in T_{b}$, let $\widetilde{w} \in \{0,\Bar{V}\}^{S_{r}(b)}$ be $\widetilde{w}=w$ in $\Omega_{r}(b)$ and $\widetilde{w}=V'$ in $S_{r}(b)\setminus \Omega_{r}(b)$.  
Given $w_{1},w_{2}\in T_{b}$, 
there is an edge which starts from $w_{1}$ and ends at $w_{2}$ if $w_{2}=w_{1}+\Bar{V}\delta_{b'}$ for some $b'\in \Omega_{r}(b)$ and $n\big( (-\Delta)_{S_{r}(b)}+\widetilde{w_{2}};\lambda_{0}\big)=n\big( (-\Delta)_{S_{r}(b)}+\widetilde{w_{1}};\lambda_{0}\big)$.

By the similar arguments used in Case 1, there exist $V_{1},V_{2}\in \mathcal{E}_{k_1,k_2,\ell}$ such that the following holds:
\begin{itemize}
    \item $\forall b \in D$, either $V_{1}|_{Q_{r}(b)}=V_{2}|_{Q_{r}(b)}$ or $V_{2}|_{Q_{r}(b)}=V_{1}|_{Q_{r}(b)}+\Bar{V} \delta_{b'}$ for some non-crossing site $b'$ w.r.t. $V_{1}$.
    \item There exists a non-crossing site $a_{0}$ w.r.t. $V_{1}$ such that $V_{2}(a_{0})=\Bar{V}$ and $|u_{V_{1},i(V_{1})+k_{1}}(a_{0})| \geq \Bar{V}^{-\frac{1}{2}R_{2}}$.
\end{itemize}
Denote $V_{3}=V_{1}+\Bar{V}\delta_{a_{0}}$. Then by \eqref{eq:cutting-free-sites1} and definition of non-crossing site, $i(V_{3})=i(V_{1})=i(V_{2})$. Since $V_{1}\leq V_{3}\leq V_{2}$, by monotonicity, we have
\begin{equation}\label{eq:monotonicity}
    \lambda^{0}_{i(V_{1})+k_{2}}(V_{1})\leq \lambda^{0}_{i(V_{3})+k_{2}}(V_{3})\leq \lambda^{0}_{i(V_{2})+k_{2}}(V_{2}).
\end{equation}
 Now we apply Lemma \ref{lem:eigenvariation} to $H_{Q}-\lambda_{0}+s_{\ell}$
        with $r_1=2s_{\ell}$, $r_2=s_{\ell+1}$, $r_3=\Bar{V}^{- R_2}$, $r_4=\Bar{V}^{-c R_4}$ and $r_5=\Bar{V}^{-R_5}$. 
        Then $\lambda^{0}_{i(V_{3})+k_{2}}(V_{3})\geq\lambda_{0}+s_{\ell}$. By \eqref{eq:monotonicity}, $\lambda^{0}_{i(V_{2})+k_{2}}(V_{2})\geq\lambda_{0}+s_{\ell}$ and thus $V_{2} \not \in \mathcal{E}_{k_{1},k_{2},\ell}$. We thus arrived at a contradiction.
\end{proof}

\begin{cla}\label{cla:inclusion}
\begin{equation}
    \{\|(H_{Q}-\lambda_{0})^{-1}\| > \Bar{V}^{R_1}\} \cap \{V|_{\Theta}=V'\} \subset \bigcup_{\substack{0\leq k_{1}\leq k_{2}\leq p\\ 0\leq \ell \leq CR_{0}^{\delta}}} \mathcal{E}_{k_{1},k_{2},\ell}
\end{equation}
\end{cla}
\begin{proof}[Proof of the claim]
By Claim \ref{cla:bound-p} and Claim \ref{cla:total-number-bad-eigenvalues}, we can always find $0\leq \ell \leq CR_{0}^{\delta}$ such that the annulus $(\lambda_{0}-s_{\ell+1},\lambda_{0}+s_{\ell+1})\setminus(\lambda_{0}-s_{\ell},\lambda_{0}+s_{\ell})$ contains no eigenvalue of $H_{Q}$. The claim follows.
\end{proof}
 Finally by Claim \ref{cla:inclusion},
        \begin{multline}
            \Prob[\|(H_{Q}-\lambda_{0})^{-1}\| > \Bar{V}^{R_1}| \; V|_{\Theta} = V'] \\
            \leq \sum_{0\leq k_1,k_2 \leq p}\sum_{1 \leq \ell \leq CR_{0}^{\delta}} \Prob[\mathcal{E}_{k_1,k_2,\ell} \cap \mathcal{E}_{uc}|\; V|_{\Theta} = V'] + \Prob[\mathcal{E}_{uc}^{c}\big|\; V|_{\Theta}=V']. 
        \end{multline}
         By Claim \ref{cla:unique-continuation-bound}, \ref{cla:wegner-bound} and let $C_{\varepsilon,\delta,K}$ be large enough,
        \begin{align*}
            &\mathbb{P}[\|(H_{Q}-\lambda_{0})^{-1}\| > \Bar{V}^{R_1}| \; V|_{\Theta} = V']\\
            &\leq Cr^{6} R_{0}^{1+3\delta} R_{4}^{-\frac{3}{2}}+\exp(- R_{0}^{-\varepsilon})\\
            &\leq R_{0}^{10\varepsilon-\frac{1}{2}}.
        \end{align*}
        We used here $r\geq C_{\varepsilon,\delta,K}$ and $R_{0} \geq \exp(c_{2} r)$.
\end{proof}

\section{Larger scales}\label{sec:larger-scale}
We now prove Theorem \ref{thm:resolvent-exponential-decay} by a multi-scale analysis based on \cite[Lemma 8.3]{ding2020localization} with the Wegner estimate Proposition \ref{prop:Wegner}. 
\begin{defn}
Suppose $r$ is an odd number, $\mathcal{R}$ is a set of $r$-bits and $E \subset \Z^2$. We denote 
\begin{equation}
    \mathcal{R}_{E}=\{Q_{r}(b) \in \mathcal{R}: Q_{r}(b) \subset E\}.
\end{equation}
\end{defn}
The following gluing lemma in the multi-scale analysis is a direct modification of \cite[Lemma 6.2]{ding2020localization} and it follows from the same proof as \cite[Lemma 6.2]{ding2020localization}.
\begin{lemma}[Gluing lemma]\label{lem:gluing}
If
\begin{enumerate}
    \item $\varepsilon>\delta>0$ small and $c_{3}>0$ a numerical constant
    \item $K \geq 1$ an integer, $r>C_{\varepsilon,\delta,K}$ a large odd number and $\Bar{V}>\exp(r^{2})$
    \item $t \in [0,1]$ and $\lambda_{0}\in \left[0,8\right]$
    \item scales $R_{0}\geq \cdots \geq R_{6} \geq \exp(c_{3}r)$ with $R_{k}^{1-\varepsilon}\geq R_{k+1}$
    \item $1 \geq m \geq 2R_{5}^{-\delta}$ represents the exponential decay rate
    \item $Q=Q_{R_{0}}(a) \subset \Z^2$ an $r$-dyadic box
    \item $Q_{1}',\cdots,Q_{K}' \subset Q$ disjoint $r$-dyadic $R_{2}$-boxes with $\|(H_{Q_{k}'}-\lambda_{0})^{-1}\|\leq \Bar{V}^{R_{4}}$ (they are called ``defects'')
    \item $\mathcal{R}$ a subset of admissible $r$-bits inside $Q$ which do not affect $\cup_{k}Q_{k}'$ 
    \item for all $b \in Q$ one of the following holds
    \begin{itemize}
        \item there is $Q_{k}'$ such that $b \in Q_{k}'$ and $\dist (b,Q\setminus Q'_{k})\geq \frac{1}{8} \ell(Q_{k}')$
        \item there is an $r$-dyadic $R_{5}$-box $Q'' \subset Q$ such that $b \in Q''$, $\dist(b,Q\setminus Q'') \geq \frac{1}{8}\ell(Q'')$, and $|G_{Q''}^{\mathcal{R}_{Q''},t}(b',b'';\lambda_{0})| \leq \Bar{V}^{R_{6}-m |b'-b''|}$ for $b',b'' \in Q''$,
    \end{itemize}
\end{enumerate}
then $|G_{Q}^{\mathcal{R},t}(b,b';\lambda_{0})|\leq \Bar{V}^{R_{1}-\Bar{m}|b-b'|}$ for $b,b' \in Q$ where $\Bar{m}=m-R_{5}^{-\delta}$.
\end{lemma}
\begin{rem}
As in \cite[Remark 6.3]{ding2020localization}, the scales $R_{0},\cdots,R_{6}$ have the following interpretations:
\begin{enumerate}
    \item $R_{0}$: large scale
    \item $\exp(R_{1})$: large scale resolvent bound
    \item $R_{2}$: defect scale
    \item $-R_{3}$: defect edge weight
    \item $\exp(R_{4})$: defect resolvent bound
    \item $R_{5}$: small scale
    \item $\exp(R_{6})$: small scale resolvent bound
\end{enumerate}
They are set up to be compatible with the multi-scale analysis (Theorem \ref{thm:multi-2}) below.
\end{rem}

The following covering lemma is a direct modification of \cite[Lemma 8.1]{ding2020localization} and it follows from the same proof as \cite[Lemma 8.1]{ding2020localization}.
\begin{lemma}\label{lem:findlargerbox}
Suppose $K\geq 1$ is an integer, $r$ is a large odd number, $\alpha \geq C^{K}$ is a power of $2$,  $R_{0}\geq R_{1}\geq R_{2}$ are $r$-dyadic scales with $R_{i}\geq \alpha R_{i+1}(i=0,1)$, $Q\subset \Z^2$ is an $r$-dyadic $R_{0}$-box and $Q''_{1},\cdots,Q''_{K} \subset Q$ are $r$-dyadic $R_{2}$-boxes. Then there is an $r$-dyadic scale $R_{3}\in \left[R_{1},\alpha R_{1}\right]$ and disjoint $r$-dyadic $R_{3}$-boxes $Q'_{1},\cdots,Q'_{K}\subset Q$ such that, 
\begin{equation}\label{eq:covering-condition}
    \text{for each $Q_{k}''$, there is $Q'_{j}$ with $Q_{k}'' \subset Q'_{j}$ and $\dist(Q''_{k},Q\setminus Q'_{j})\geq \frac{1}{8}R_{3}$}.
\end{equation}

\end{lemma}
The following lemma provides the continuity of resolvent bounds and its proof was given in \cite{ding2020localization}.
\begin{lemma}[Lemma 6.4 in \cite{ding2020localization}]\label{lem:continuity_resol}
If square $Q \subset \Z^{2}$, $\lambda \in \R$, $\alpha > \beta >0$, and 
$$
|(H_{Q} - \lambda)^{-1}(x,y)| \leq \exp(\alpha - \beta |x-y|)
$$
for any $x,y \in Q$, then for all $|\lambda' - \lambda|\leq c\beta |Q|^{-1} \exp(-\alpha)$, we have
$$
|(H_{Q} - \lambda')^{-1}(x,y)| \leq 2\exp(\alpha - \beta |x-y|)
$$
for any $x,y\in Q$.
\end{lemma}

\begin{defn}\label{def:goodness}
Suppose $\gamma,\varepsilon>0$, large odd number $r$, real $\Bar{V}>\exp(r^{2})$, energy $\lambda_{0}$, $r$-dyadic box $Q_{L}(a)$, $\Theta \subset Q_{L}(a)$ and $V':\Theta \rightarrow \{0,\Bar{V}\}$. We say $(Q_{L}(a),\Theta,V')$ is $(\gamma,\varepsilon)$-good if the following holds:

Whenever we have
\begin{itemize}
    \item $V:Q_{L}(a) \rightarrow \{0,\Bar{V}\}$ with $V|_{\Theta}=V'$,
    \item $b,c \in Q_{L}(a)$,
    \item $t \in [0,1]$,
    \item $\mathcal{R}$ a subset of $r$-bits inside $Q_{L}(a)$ such that each $Q_{r}(b) \in \mathcal{R}$ does not affect $\Theta$,
\end{itemize}
\;\;\;\; then 
\begin{itemize}
    \item  for each $Q_{r}(b)\in \mathcal{R}$, $(Q_{r}(b),V|_{F_{r}(b)})$ is admissible,
    \item  the following inequality holds:
\begin{equation}
    |G^{\mathcal{R},t}_{Q_{L}(a)}(b,c;\lambda_{0})| \leq \Bar{V}^{-\gamma |b-c|+L^{1-\varepsilon}}.
\end{equation}
\end{itemize}

\end{defn}

The following multi-scale analysis is a direct modification of \cite[Lemma 8.3]{ding2020localization}. By using a standard argument (see, e.g. \cite[Proof of Theorem 3.1]{li2022anderson}), it implies Theorem \ref{thm:resolvent-exponential-decay} with $Y_{\Bar{V}}=J_{r}^{\Bar{V}}$.

Recall that in Definition \ref{theta_1},  we defined $\Theta^{r}=\cup_{a \in \dot{r}\Z^2}F_{r}(a)$ for any large odd number $r$.
\begin{theorem}[Multi-scale Analysis]\label{thm:multi-2}
  For each $\kappa <\frac{1}{2}$, we can pick $\varepsilon >\delta>0$ such that, for each odd number $r>C_{\varepsilon,\delta}$, $\Bar{V}>\exp(r^{2})$ and $\lambda_{0} \not \in J_{r}^{\Bar{V}}$, the following holds. 
  
  There exist
  \begin{enumerate}
      \item $r$-dyadic scales $L_k$ for $k\geq 1$ with $L_{k+1}\in \left[\frac{1}{2} L_{k}^{\frac{1}{1-6\varepsilon}},L_{k}^{\frac{1}{1-6\varepsilon}}\right]$ and the first scale satisfying $\frac{1}{2}\exp(\frac{1}{2}c_{1}\delta r)\leq L_{1}\leq \exp(\frac{1}{2}c_{1}\delta r) $ where $c_{1}$ is the constant in Proposition \ref{prop:r-admissible}, \label{multi-cond-1}
      \item decay rates $ \gamma_{k} \geq \frac{1}{10 r}$ with $\gamma_{1}=\frac{1}{8r}$ and $\gamma_{k+1}=\gamma_{k}-L_{k}^{-\delta}$, \label{multi-cond-2}
      \item densities $\eta_{k}<\varepsilon_{0}^{\frac{1}{5}}$ with $\eta_{1}=\varepsilon_{0}^{\frac{1}{4}}$ and $\eta_{k+1}=\eta_{k}+L_{k}^{-\frac{1}{5}\varepsilon}$ where $\varepsilon_{0}$ is defined by \eqref{eq:def-of-epsilon0}, \label{multi-cond-3}
      \item random sets $\Theta_{k} \subset \Theta_{k+1}(k\geq 1)$ where $\Theta_{1}=\Theta^{r}$,\label{multi-cond-4}
    \end{enumerate}
     such that the following statements are true for $k\geq 1$, 
  \begin{enumerate}
      \item when $k\geq 2$, $\Theta_{k}\bigcap Q$ is $V|_{\Theta_{k-1}\bigcap 3Q}$-measurable for any $r$-dyadic box $Q$ with $\ell(Q)\geq L_k$,\label{multi-stat-1}
      \item when $k\geq 2$, $\Theta_{k}$ is a union of $\Theta_{k-1}$ and some $r$-bits,\label{multi-stat-2}
      \item $\Theta_{k}$ is $\eta_{k}$-regular in any $Q_{L}(a)\subset \Z^2$ with $L \geq L^{1-\frac{5}{2}\varepsilon}_{k}$, \label{multi-stat-3}
      \item for any $r$-dyadic box $Q$ with $\ell(Q)=L_{k}$,\label{multi-stat-4} 
      \begin{equation}
          \Prob\left[\text{$(Q,\Theta_{k}\cap Q, V|_{\Theta_{k}\cap Q})$ is $(\gamma_{k},\varepsilon)$-good}\right]\geq 1-L_{k}^{-\kappa}.
      \end{equation}

  \end{enumerate}
\end{theorem}
\begin{proof}
Assume $\varepsilon,\delta$ are small and we impose further constraints on these objects during the proof.  Set $r$-dyadic scale $$L_{1}\in \left[\frac{1}{2}\exp(\frac{1}{2} c_{1} \delta r),\exp(\frac{1}{2} c_{1} \delta r)\right]$$ where $c_{1}$ is the constant in Proposition \ref{prop:r-admissible}. Set $\gamma_{1} = \frac{1}{8r}$ and $\eta_{1}=\varepsilon_{0}^{\frac{1}{4}}$. By letting $r>C_{\varepsilon,\delta}$, we can pick $L_{k},\gamma_{k},\eta_{k}$ as in conditions \ref{multi-cond-1}, \ref{multi-cond-2} and \ref{multi-cond-3} for $k\geq 2$. Let $M_{0}$ be the largest integer such that $L_{M_{0}}\leq \exp(c_{1}r)$. Then $M_{0}\leq C'_{\varepsilon,\delta}$ for a constant $C'_{\varepsilon,\delta}$ depending on $\varepsilon,\delta$ and we have $L_{k-M_{0}}\leq L_{k}^{\delta}$ for each $k>M_{0}$. Set $\Theta_{k}=\Theta_{1}$ for $k=1,\cdots,M_{0}$.
\begin{figure}
    \centering
    \includegraphics{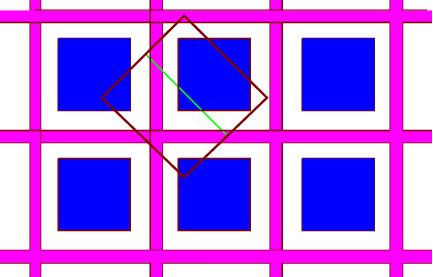}
    \caption{The figure illustrates the proof that $\Theta_{1}$ (the pink region) is $\eta_{1}$-regular in $Q_{L}(a)$. The blue region indicates $\bigcup\{Q_{(1-100\sqrt{\varepsilon_{0}})r}(b): b \in \dot{r}\Z^{2}\}$, the black tilted square indicates $\Tilde{Q}$ and the green line indicates the diagonal $\mathcal{D}$.}
    \label{fig:sparse}
\end{figure}
We prove by induction on $k$. We first prove the conclusion for $k\leq M_{0}$. Statements \ref{multi-stat-1} and \ref{multi-stat-2} are true since $\Theta_{k}=\Theta_{1}$ when $k\leq M_{0}$. To see Statement \ref{multi-stat-3}, let $Q_{L}(a)\subset \Z^2$ such that $L \geq L^{1-\frac{5}{2}\varepsilon}_{1}$. Suppose $\Tilde{Q} \subset Q_{L}(a)$ is a tilted square. We claim that, if there exists $b_{1} \in Q_{L}(a) \cap \dot{r}\Z^2$ such that $\Tilde{Q} \cap Q_{(1-100\sqrt{\varepsilon_{0}})r}(b_{1}) \not = \emptyset$, then $\Theta_{1}$ is $ \varepsilon_{0}^{\frac{1}{4}}$-sparse in $\Tilde{Q}$. We prove our claim by elementary geometry (see Figure \ref{fig:sparse}).
To see this, if $\Tilde{Q}\cap \Theta_{1}=\emptyset$ then our claim is obvious. Otherwise, note that we have
\begin{align}
    \begin{split}
        &\dist(\Theta_{1},Q_{(1-100\sqrt{\varepsilon_{0}})r}(b_{1}))\\= &\dist(F_{r}(b_{1}),Q_{(1-100\sqrt{\varepsilon_{0}})r}(b_{1}))\\=
&\dist(Q_{r}(b_{1})\setminus Q_{(1-2\varepsilon_{0})r}(b_{1}),Q_{(1-100\sqrt{\varepsilon_{0}})r}(b_{1}))\\=
& \left\lfloor \frac{(1-2\varepsilon_{0})r - 1}{2} \right\rfloor - \left\lfloor \frac{(1-100\sqrt{\varepsilon_{0}})r - 1}{2} \right\rfloor + 1\\
> &(50\sqrt{\varepsilon_{0}} - \varepsilon_{0})r.
\end{split}
\end{align}
Thus $\Tilde{Q}\cap \Theta_{1}\not=\emptyset$ implies that the edge length of $\Tilde{Q}$ is larger than $\frac{(50\sqrt{\varepsilon_{0}} - \varepsilon_{0})r}{\sqrt{2}} > 25\sqrt{\varepsilon_{0}}r$. Suppose $l\in \Z$ and $\varsigma\in\{+,-\}$ such that $\Tilde{Q}\cap \mathcal{D}_{l}^{\varsigma}\not=\emptyset$ where $\mathcal{D}_{l}^{\varsigma}$ is a diagonal defined in Definition \ref{def:diagonals}. Write $\mathcal{D}= \Tilde{Q} \cap \mathcal{D}_{l}^{\varsigma}$ and then 
\begin{equation}\label{eq:size-of-diagonal-1}
    |\mathcal{D}|>25\sqrt{\varepsilon_{0}}r.
\end{equation}
 By elementary geometry,
\begin{equation}\label{eq:regularity-2}
    |\{b\in \dot{r}\Z^2: \mathcal{D}\cap Q_{r}(b) \not=\emptyset\}|\leq 10+\frac{10|\mathcal{D}|}{r}.
\end{equation}
Since $\mathcal{D}$ has at most one intersection with any horizontal or vertical line, we have
\begin{equation}\label{eq:regularity-1}
    |\mathcal{D}\cap F_{r}(b)|\leq 10\varepsilon_{0}r
\end{equation}
for each $b\in  \dot{r}\Z^2$.
On the other hand, by Definition \ref{theta_1} we have $$|\Theta_{1}\cap \mathcal{D}|\leq \sum_{b\in \dot{r}\Z^2: \mathcal{D}\cap Q_{r}(b) \not=\emptyset} |\mathcal{D}\cap F_{r}(b)|.$$ Thus by \eqref{eq:regularity-1} and \eqref{eq:regularity-2}, we have $|\Theta_{1}\cap \mathcal{D}|\leq 100\varepsilon_{0}r + 100\varepsilon_{0}|\mathcal{D}| \leq \varepsilon_{0}^{\frac{1}{4}}|\mathcal{D}|$. The second inequality here is due to \eqref{eq:size-of-diagonal-1}. Our claim follows.

Thus any tilted square in which $\Theta_{1}$ is not $\varepsilon_{0}^{\frac{1}{4}}$-sparse is contained in
$$Q_{L}(a) \setminus \bigcup_{b \in Q_{L}(a) \cap \dot{r}\Z^2}  Q_{(1-100\sqrt{\varepsilon_{0}})r}(b),$$
whose cardinality is less than $ 10^{4}\sqrt{\varepsilon_{0}}L^{2}+8 r L \leq \varepsilon_{0}^{\frac{1}{4}} L^{2}$. Here, we used 
$L\geq \exp(c \delta r)$, $r>C_{\varepsilon,\delta}$ and $\varepsilon_{0}$ small enough (provided by \eqref{eq:def-of-epsilon0}).
Thus $\Theta_{1}$ is $\varepsilon_{0}^{\frac{1}{4}}$-regular in $Q_{L}(a)$ and Statement \ref{multi-stat-3} follows.

To see Statement \ref{multi-stat-4}, by Proposition \ref{prop:L_1good}, an $r$-dyadic $Q$ is perfect implies $(Q,\Theta_{1}\cap Q, V|_{\Theta_{1}\cap Q})$ is $(\frac{1}{8r},1)$-good. Thus Proposition \ref{prop:prob-of-perfect} implies Statement \ref{multi-stat-4} when $k\leq M_{0}$. 

Assume $k\geq M_{0}+1$ and our conclusions hold for any smaller $k$. We proceed to prove it for $k$. The general strategy is to apply Lemma \ref{lem:gluing}.

For each $j< k$, we call an $r$-dyadic box $Q_{L_{j}}(a)$ ``good'' if 
$$(Q_{L_{j}}(a),\Theta_{j} \cap Q_{L_{j}}(a),V|_{\Theta_{j} \cap Q_{L_{j}}(a)}) \text{ is } (\gamma_{j},\varepsilon)\text{-good}.$$ Otherwise, we call it ``bad''. We must control the number of bad boxes in order to apply Lemma \ref{lem:gluing}.

For any $0< k' < k$, by Lemma \ref{lem:gluing}, any bad $r$-dyadic $L_{k'}$-box $Q$ must contain a bad $L_{k'-1}$-box.
For any $0 < i \leq k$, and a bad $L_{k-i}$-box $Q'\subset Q$, we call $Q'$ a \emph{hereditary bad $L_{k-i}$-subbox of $Q$}, if there exists a sequence $Q'=\overline Q_i \subset \overline Q_{i-1} \subset \cdots \subset \overline Q_1 \subset Q$, where for each $j=1, \cdots, i$, $\overline Q_j$ is a bad $L_{k-j}$-box. We also call such sequence $\{\overline Q_{j}\}_{1\leq j \leq i}$ a hereditary bad chain of length $i$.
Note that the set of hereditary bad chains of $Q$ is $V_{\Theta_{k-1}\cap Q}$-measurable. 
\begin{cla}
If $\varepsilon<c_{\kappa}$ and $N> C_{M_{0},\kappa}$, then for all $k>M_{0}$, 
$$\Prob[\text{Q has no more than $N$ hereditary bad chains of length $M_{0}$}]\geq 1-L_{k}^{-1}.
$$
\end{cla}
\begin{proof}[Proof of the claim]
Let $N=(N')^{M_{0}}$ with $N'$ to be determined.  We can use the inductive hypothesis to estimate 
\begin{align}
    &\Prob\left[\text{$Q$ has more than $N$ hereditary bad chains of length $M_{0}$} \right]\\
    &\leq \sum_{\substack{\text{$r$-dyadic $Q' \subset Q$} \\  \ell(Q')=L_{j} \\ k-M_{0}<j\leq k }} \Prob \left[ \text{$Q'$ has more than $N'$ bad $L_{j-1}$-subboxes} \right]\\
     &\leq \sum_{\substack{ k-M_{0}<j\leq k }} L_{k}^{2} (L_{j}/L_{j-1})^{C N'} (L^{-\kappa}_{j-1})^{c N'} \\
    &\leq C M_{0} L_{k}^{2} (L_{k}^{(C\varepsilon-c \kappa) N' }+L_{k-M_{0}}^{(C\varepsilon-c \kappa) N' })\\ 
    &\leq C M_{0} L_{k}^{2} (L_{k}^{(C\varepsilon-c \kappa) N' }+L_{k}^{(C\varepsilon-c \kappa)(1-6\varepsilon)^{M_{0}} N' }).
\end{align}
Here, $c, C$ denote absolute constants. The claim follows by taking $c_{\kappa}=\frac{c \kappa}{2 C}$ and $C_{M_{0},\kappa} = (\frac{20 + 2M_{0}}{c\kappa (1-6c_{\kappa})^{M_{0}}})^{M_{0}}$, and letting $\varepsilon<c_{\kappa}$ and $N'>C^{\frac{1}{M_{0}}}_{M_{0},\kappa}$.
\end{proof}
Now fix $N$ as in the claim above. We call an $L_{k}$-box $Q$ \emph{ready} if $Q$ is $r$-dyadic and $Q$ contains no more than $N'$ hereditary bad chains of length $M_{0}$. Note that the event that $Q$ is ready is $V|_{\Theta_{k-1}\cap Q}$-measurable. 

Suppose the $L_{k}$-box $Q$ is ready. Let $Q'''_{1},\cdots,Q'''_{N} \subset Q$ be a complete list of $L_{k-M_{0}}$-boxes that includes every hereditary bad $L_{k-M_{0}}$-subboxes of $Q$. 
Let $Q''_{1},\cdots,Q''_{N} \subset Q$ be the corresponding bad $L_{k-1}$-subboxes of $Q$, such that $Q'''_{i} \subset Q''_{i}$ for each $i=1,2,\cdots,N$. These cubes are chosen in a way such that $\{Q''_{1},\cdots,Q''_{N}\}$ contains all the bad $L_{k-1}$-subboxes in $Q$.
Applying Lemma \ref{lem:findlargerbox}, we can choose an $r$-dyadic scale $L' \in [c_{N}L_{k}^{1-3\varepsilon},L_{k}^{1-3\varepsilon}]$ and disjoint $r$-dyadic $L'$-subboxes $$Q_{1}',\cdots,Q_{N}'\subset Q$$ such that, for each $Q_{i}''$, there is $Q_{j}'$ such that $Q_{i}''\subset Q'_{j}$ and $\dist(Q''_{i},Q \setminus Q'_{j})\geq \frac{1}{8}L'$. Note that we can choose $Q_{i}'$, $Q_{i}''$, $Q_{i}'''$ in a $V_{\Theta_{k-1}\cap Q}$-measurable way.

We define $\Theta_{k}$ to be the union of $\Theta_{k-1}$ and the subboxes $Q'_{1},\cdots,Q'_{N} \subset Q$ of each ready $L_{k}$-box $Q$.
We need to verify statements \ref{multi-stat-1} to \ref{multi-stat-4}. Note that Statement \ref{multi-stat-2} is true since each $r$-dyadic box is a union of $r$-bits (Lemma \ref{lem:r_dyadic_cube-ele}).
\begin{cla}
Statements \ref{multi-stat-1}, \ref{multi-stat-3} hold.
\end{cla}
\begin{proof}[Proof of the claim]
For each $L_{k}$-box $Q$, the event that $Q$ is ready, the scale $L'$ and $L'$-boxes $Q'_{i}\subset Q$ are all $V|_{Q \cap \Theta_{k-1}}$-measurable. Thus $\Theta_{k}\cap Q$ is $V|_{\Theta_{k-1}\cap 3Q}$-measurable. Note that we have $3Q$ in place of $Q$ because each $r$-dyadic $L_{k}$-box $Q$ intersects $24$ other $r$-dyadic $L_{k}$-boxes contained in $3 Q$.

As for Statement \ref{multi-stat-3}, for each $L^{1-\frac{5}{2}\varepsilon}_{k}$-box $Q\subset \Z^2$, the set $Q \cap \Theta_{k} \setminus \Theta_{k-1}$ is covered by at most $25N$ boxes $Q'_{i}$ with length at most $L_{k}^{1-3\varepsilon}$. Suppose $\Tilde{Q}$ is a tilted square such that $Q\cap \Theta_{k-1}$ is $\eta_{k-1}$-sparse in $\Tilde{Q}$ but $Q\cap \Theta_{k}$ is not $\eta_{k}$-sparse in $\Tilde{Q}$, then $\Tilde{Q}$ must intersect one of $Q'_{i}$'s and have length at most $L_{k}^{1-\frac{11}{4}\varepsilon}$. This implies $\Theta_{k}\cap Q$ is $\eta_{k}$-regular in $Q$.
\end{proof}
\begin{cla}\label{cla:inductive-hypothesis-good}
If the $L_{k}$-box $Q$ is ready, $\mathcal{R}$ a subset of $r$-bits inside $Q'_{i}$ that do not affect $\Theta_{k-1}\cup \bigcup_{j} Q_{j}''$, then each $Q_{r}(b)\in \mathcal{R}$ is admissible. Furthermore, if $|\lambda-\lambda_{0}|\leq \Bar{V}^{-2 L_{k-1}^{1-\varepsilon}}$,  $t \in [0,1]$ and $H^{\mathcal{R},t}_{Q_{i}'}u=\lambda u$, then
$$\Bar{V}^{c L_{k-1}^{1-\delta}} \|u\|_{\ell^{\infty}(E)} \leq \|u\|_{\ell^{2}(Q'_{i})}\leq (1+\Bar{V}^{-c L_{k-M_{0}}^{1-\delta}}) \|u\|_{\ell^{2}(G)},
$$
where $E=Q'_{i}\setminus \cup_{i}Q''_{j}$ and $G=Q'_{i} \cap \cup_{j}Q'''_{j}$.
\end{cla}
\begin{proof}[Proof of the claim]
If $r$-bit $Q_{r}(b)\subset Q$ does not affect $\Theta_{k-1}\cup \bigcup_{j} Q_{j}''$, then it is contained in a good $L_{k-1}$-box $Q_{L_{k-1}}(a')\subset Q$. By Definition \ref{def:goodness}, since $Q_{r}(b)$ does not affect $\Theta_{k-1}\cap Q_{L_{k-1}}(a')$, it is admissible.

If $a \in Q'_{i} \setminus G$, then there is $j \in \{1,\cdots,M_{0}\}$ and a good $L_{k-j}$-box $Q'' \subset Q'_{i}$ with $a \in Q''$ and $\dist(a,Q'_{i}\setminus Q'')\geq \frac{1}{8}L_{k-j}$. Moreover, if $a \in E$, then $j=1$. By Definition \ref{def:goodness} and Lemma \ref{lem:continuity_resol},
\begin{align*}
        |u(a)| &=\left|\sum_{\substack{b \in Q''\\ b'\in Q'_{i}\setminus Q''\\ b \sim b'}}G^{\mathcal{R_{Q''}},t}_{Q''}(a,b;\lambda) u(b')\right|\\
        &\leq 4 L_{k-j}\Bar{V}^{L_{k-j}^{1-\varepsilon}-\frac{1}{8}\gamma_{k-j}L_{k-j}}\|u\|_{\ell^{2}(Q'_{i})}\\ 
        &\leq \Bar{V}^{-c L_{k-j}^{1-\delta}} \|u\|_{\ell^{2}(Q'_{i})}.
\end{align*}
Here we used $\gamma_{k-j}\geq \frac{1}{10r}$ and $L_{k-j}\geq \exp(c\delta r)$.
In particular, we see that
$$\|u\|_{\ell^{\infty}(E)} \leq \Bar{V}^{-c L_{k-1}^{1-\delta}} \|u\|_{\ell^{2}(Q'_{i})}
$$
and
$$\|u\|_{\ell^{\infty}(Q'_{i}\setminus G)}\leq \Bar{V}^{-c L_{k-M_{0}}^{1-\delta}} \|u\|_{\ell^{2}(Q'_{i})}.$$
\end{proof}
\begin{cla}
If $Q$ is an $r$-dyadic $L_{k}$-box and $\mathcal{E}_{i}(Q)$ denotes the event that 
$$\text{$Q$ is ready and $\Prob[\|(H_{Q'_{i}}-\lambda_{0})^{-1}\|\leq \Bar{V}^{L_{k}^{1-4\varepsilon}} \big| V|_{\Theta_{k}\cap Q}]=1$},
$$
then $\Prob[\mathcal{E}_{i}(Q)]\geq 1-L_{k}^{10\varepsilon-\frac{1}{2}}$. 
\end{cla}
\begin{proof}[Proof of the claim]
Recall the event $Q$ ready and boxes $Q'_{i} \subset Q$ are $V|_{\Theta_{k-1}\cap Q}$-measurable. We may assume $i=1$. We apply Proposition \ref{prop:Wegner} to box $Q'_{1}$ with $5\varepsilon>\delta>0$, $K=N$, scales $$L' \geq L_{k}^{1-4\varepsilon}\geq L_{k}^{1-5\varepsilon} \geq L_{k-1} \geq L_{k-1}^{1-2\delta}\geq 2L_{k-1}^{1-\varepsilon}\geq L_{k-1}^{1-\frac{5}{2}\varepsilon},$$ $\Theta=\Theta_{k-1} \cap Q'_{1}$, defects $\{Q''_{j}:Q''_{j} \subset Q'_{1}\}$, and $G=\cup\{Q'''_{j}:Q'''_{j} \subset Q'_{1}\}$. Assume $\varepsilon>5\delta$ and note that $k\geq M_{0}+1$ and $L_{k-1}\geq L_{M_{0}}\geq \exp(\frac{1}{2}c_{1} r)$. The previous claims provide the conditions to verify the hypothesis of Proposition \ref{prop:Wegner}. Since $Q'_{1} \subset \Theta_{k}$ when $Q$ is ready, the claim follows. 
\end{proof}
\begin{cla}
If $Q$ is an $r$-dyadic $L_{k}$-box and $\mathcal{E}_{1}(Q),\cdots,\mathcal{E}_{N}(Q)$ hold, then $Q$ is good.
\end{cla}
\begin{proof}[Proof of the claim]

Suppose $\mathcal{R}$ is a subset of $r$-bits inside $Q$ that do not affect $\Theta_{k}$ and $t \in [0,1]$. By Claim \ref{cla:inductive-hypothesis-good}, each $Q_{r}(b')\in \mathcal{R}$ is admissible.
We apply Lemma \ref{lem:gluing} to the box $Q$ with small parameters $\frac{\varepsilon}{3}>\delta>0$, decay rate $\gamma_{k-1}$, scales $L_{k}\geq L_{k}^{1-\varepsilon}\geq L' \geq L_{k}^{1-\frac{7}{2}\varepsilon}\geq L_{k}^{1-4\varepsilon}\geq L_{k-1}\geq L_{k-1}^{1-\varepsilon}$, and defects $Q'_{1},\cdots,Q'_{N}$. We conclude that 
$$|G^{\mathcal{R},t}_{Q}(a,b)|\leq \Bar{V}^{L_{k}^{1-\varepsilon} -\gamma_{k} |a-b|}
$$
for each $a,b\in Q$.
Since the events $\mathcal{E}_{i}(Q)$ are $V|_{\Theta_{k}\cap Q}$-measurable, we see that $Q$ is good.
\end{proof}

Finally we verify Statement \ref{multi-stat-4}.
Combining the previous two claims, for any $r$-dyadic $L_{k}$-box $Q$, we have 
$$\Prob[\text{$(Q,\Theta_{k}\cap Q, V|_{\Theta_{k}\cap Q})$ is $(\gamma_{k},\varepsilon)$-good}]\geq 1-N L_{k}^{10\varepsilon-\frac{1}{2}} \geq 1-L_{k}^{-\kappa},
$$ provided $\kappa < \frac{1}{2} -10\varepsilon$. 
\end{proof}
\section{Proof of Lemma \ref{lem:unique-continuation}}\label{sec:proof-Lemma-3.5}
Our approach follows the scheme in \cite[Section 3]{ding2020localization} and \cite{buhovsky2017discrete}. The key for the proofs of \cite[Theorem 1.6]{ding2020localization}, \cite[Theorem (A)]{buhovsky2017discrete} and Lemma \ref{lem:unique-continuation} is the following observation for functions $u$ satisfying $H u =\lambda u$ on a tilted rectangle $R_{[1,a],[1,b]}$ defined in Definition \ref{def:geo1}.
\begin{obs}\label{obs:tilted}
Let $V:\Z^{2} \rightarrow \R$ and $u:R_{[1,a],[1,b]} \rightarrow \R$. Suppose $a\geq 10b$ and $-\Delta u +V u =\lambda u$ in $R_{[2,a-1],[2,b-1]}$. If $$\|u\|_{\ell^{\infty}(R_{[1,a],[1,2]})}\leq 1$$ and $|u|\leq 1$ on a $1-\varepsilon$ fraction of $R_{[1,a],[b-1,b]}$, then $\|u\|_{\ell^{\infty}(R_{[1,a],[1,b]})}$ is ``suitably'' bounded. 
\end{obs}
Observation \ref{obs:tilted} does not hold for arbitrary $V$ and $\lambda$. It was proved in \cite[Lemma 3.4]{buhovsky2017discrete} for the case where $V\equiv 0$ and $\lambda=0$ (i.e. $u$ is a harmonic function). In \cite{ding2020localization}, Observation \ref{obs:tilted} was also proved to hold with high probability for the case where $a\geq C b^{2}\log(a)$ and $\{V(x)\}_{x\in \Z^{2}}$ is a family of i.i.d. Bernoulli random variables taking values in $\{0,1\}$ (the ``key lemma'' \cite[Lemma 3.13]{ding2020localization}).

In Lemma \ref{lem:key-lemma} below, we will prove that Observation \ref{obs:tilted} holds with high probability only requiring $a\geq 10b$ and $\{V(x)\}_{x\in \Z^{2}}$ is a family of i.i.d. Bernoulli random variables taking values in $\{0,\Bar{V}\}$. 

In \cite{ding2020localization}, the authors used the ``key lemma'' \cite[Lemma 3.13]{ding2020localization} to prove a ``growth lemma'' (\cite[Lemma 3.18]{ding2020localization}) and then used a covering argument to conclude the proof of \cite[Theorem 1.6]{ding2020localization}.

Our Lemma \ref{lem:key-lemma} is analog to \cite[Lemma 3.13]{ding2020localization} with weaker assumptions.
Lemma \ref{lem:key-lemma} is the main new ingredient in the current proof. 
As long as Lemma \ref{lem:key-lemma} is proved, we follow the same scheme in \cite{ding2020localization} by proving a ``growth lemma'' (Lemma \ref{lem:growthlemma}) and using a covering argument (Section \ref{sec:covering}) to conclude the proof of Lemma \ref{lem:unique-continuation}. 

To prove Observation \ref{obs:tilted} (Lemma \ref{lem:key-lemma}), we first consider the case where $u=0$ on $R_{[1,a],[1,2]}$ (Lemma \ref{lem:transfer-u_3}).  
We use the triangular matrix structure of the operator $M_{[1,a]}^{k,k'}$ defined in Definition \ref{def:transfer-matrix}. Then we use Lemma \ref{lem:Boolean-cube} to estimate the probability. We refer the reader to the beginning of Section \ref{sec:key-lemmas} for an intuitive argument of the simple case where $u|_{R_{[1,a],[1,2]}\cup R_{[1,a],[b-1,b]}}=0$.

\subsection{Auxiliary lemmas}\label{sec:auxiliary-boolean}
We first prove Lemma \ref{lem:Boolean-cube} using Lemma \ref{lem:restricted-invertibility}. 
\begin{proof}[Proof of Lemma \ref{lem:Boolean-cube}]
Write $\{e_{j}\}_{j=1}^{n}$ to be the standard normal basis in $\mathbb{R}^{n}$.
Write $\Gamma=\Gamma_{0}+a_{0}$ where $\Gamma_{0}$ is a $k$ dimensional subspace and $a_{0}\in \mathbb{R}^{n}$. Let $\Gamma_{1}$ be the orthogonal complement of $\Gamma_{0}$ and let $P:\mathbb{R}^{n} \rightarrow \Gamma_{1}$ be the orthogonal projection. Define $v_{i}=P e_{i}$ for $i=1,2,\cdots,n$, then $\sum_{i=1}^{n} v_{i} v_{i}^{\dag} =I_{n-k}$ (the identity operator on $\Gamma_{1}$). 

Using Lemma \ref{lem:restricted-invertibility} with $l=n$, $m=n-k$ and $m'=n-k-1$, we can find $\mathcal{S}\subset\{1,2,\cdots,n\}$ with $|\mathcal{S}|=n-k-1$ such that the $n-k-1$-th largest eigenvalue of
\begin{equation}\label{eq:restricted-subset}
    \sum_{i\in \mathcal{S}} v_{i} v_{i}^{\dag}=\sum_{i\in \mathcal{S}} P e_{i}e_{i}^{\dag} P^{\dag} 
\end{equation}
is at least $\frac{1}{4n(n-k)}$. Assume without loss of generality that $\mathcal{S}=\{1,2,\cdots,n-k-1\}$. Denote by $\Gamma'$ the subspace generated by $\{e_{i}\}_{i=1}^{n-k-1}$ and let $Q:\mathbb{R}^{n}\rightarrow \Gamma'$ be the orthogonal projection onto $\Gamma'$. Then \eqref{eq:restricted-subset} is just $P Q^{\dag} Q P^{\dag}$. Note that the dimension of the range of $Q P^{\dag}$ is at most $n-k-1$, thus the rank of the operator $P Q^{\dag} Q P^{\dag}$ is at most $n-k-1$.
Hence the $n-k-1$-th largest eigenvalue (which is also the smallest eigenvalue) of the positive semi-definite operator $Q P^{\dag}P Q^{\dag}$ is at least $\frac{1}{4n(n-k)}$. This implies 
\begin{equation}\label{eq:invertibility}
    \|P Q^{\dag} a\|_{2}\geq \sqrt{\frac{1}{4n(n-k)}}\|a\|_{2} 
\end{equation}
for any $a\in \Gamma'$.

Consider the Boolean subcube $B'=\left\{\sum_{i=1}^{n-k-1} x_{i} e_{i}:x_{i}\in \{0,1\}\right\}\subset \Gamma'$. We claim that for any $v' \in \mathbb{R}^{n}$, 
\begin{equation}\label{eq:intersection-at-most-one}
    \#\{a \in B'+v': \min_{b\in \Gamma} \|a-b\|_{2}< \frac{1}{4} n^{-\frac{1}{2}} (n-k)^{-\frac{1}{2}}\} \leq 1.
\end{equation}
To see this, assume the claim does not hold. Then for some $v'' \in \mathbb{R}^{n}$, there are two different $a_{1},a_{2} \in (B'+v'')$ with $\min_{b\in \Gamma} \|a_{j}-b\|_{2}< \frac{1}{4} n^{-\frac{1}{2}} (n-k)^{-\frac{1}{2}}$ for $j=1,2$. Choose $b_{1},b_{2}\in \Gamma$ with $\|a_{j}-b_{j}\|_{2}< \frac{1}{4} n^{-\frac{1}{2}} (n-k)^{-\frac{1}{2}}$ for $j=1,2$. Let $a'=a_{1}-a_{2}$ and $b'=b_{1}-b_{2}$. Then $\|a'-b'\|_{2}<\frac{1}{2} n^{-\frac{1}{2}} (n-k)^{-\frac{1}{2}}$  and $a'\in \Gamma'$, $b'\in \Gamma_{0}$. 

Since any two vectors in $B'+v''$ has $\ell^{2}$ distance at least $1$, we have $\|a'\|_{2}\geq 1$. On the other hand, we have $\min_{b\in \Gamma_{0}} \|a'-b\|_{2}<\frac{1}{2} n^{-\frac{1}{2}} (n-k)^{-\frac{1}{2}}$ which is equivalent to $\|P Q^{\dag} a'\|_{2}< \frac{1}{2} n^{-\frac{1}{2}} (n-k)^{-\frac{1}{2}}$.  However, this contradicts with \eqref{eq:invertibility} and our claim \eqref{eq:intersection-at-most-one} follows.

Finally, $B=\bigcup \left\{B'+\sum_{j=n-k}^{n} x_{j}e_{j}: x_{j} \in \{0,1\}\text{ for } n-k\leq j\leq n\right\}$. Thus by \eqref{eq:intersection-at-most-one}, 
\begin{align*}
             &\#\left\{a \in B: \min_{b\in \Gamma} \|a-b\|_{2}< \frac{1}{4} n^{-\frac{1}{2}} (n-k)^{-\frac{1}{2}}\right\}\\ 
             &\leq \sum_{x_{j} \in \{0,1\} \text{ for } n-k\leq j \leq n} \#\left\{a \in B'+\sum_{j=n-k}^{n} x_{j}e_{j}: \min_{b\in \Gamma} \|a-b\|_{2}< \frac{1}{4} n^{-\frac{1}{2}} (n-k)^{-\frac{1}{2}}\right\}\\ &\leq \sum_{x_{j} \in \{0,1\} \text{ for }  n-k\leq j \leq n} 1\\
             &= 2^{k+1}.
\end{align*}
\end{proof}
We will also need the following lemma to bound the inverse norm of principal submatrices of a triangular matrix.
\begin{lemma}\label{lem:upper-triangular}
Let $d>0$ be an integer, $K>1$ be a real number and $\{m_{1}<m_{2}<\cdots<m_{d}\}$ be a set of positive integers. Let $A=\left(a_{i j}\right)_{1\leq i,j\leq d}$ be a lower (or upper) triangular matrix. Assume that $|a_{i i}|=1$ for each $i=1,\cdots,d$ and $|a_{i j}|\leq K^{|m_{i}-m_{j}|}$ for each $1\leq i,j\leq d$. Then the Euclidean operator norm of the inverse $A^{-1}$ satisfies $\|A^{-1}\|\leq d (2 K)^{m_{d}}$. 

\end{lemma}
\begin{proof}
We assume $A$ to be a lower triangular matrix, the case for upper triangular matrix follows the same argument.
Denote $A^{-1}=\left(a'_{i j}\right)_{1\leq i,j \leq d}$.

We prove that $|a'_{i j}| \leq (2 K) ^{|m_{i}-m_{j}|}$ by induction on $k=i-j$. For $k=0$, since $A$ is lower triangular, we have $a'_{i i}=(a_{i i})^{-1}$ and thus $|a'_{i i}|=1$. Assume our conclusion holds for $0\leq k<k'$, we prove the case where $i-j=k'$. Note that 
\begin{equation}
    \sum_{l=1}^{d} a_{i l} a'_{l j}=0.
\end{equation}
This implies 
\begin{equation}
    a_{i i} a'_{i j}=- \sum_{l=j}^{i-1} a_{i l} a'_{l j}.
\end{equation}
Since $|a_{i i}|=1$, by inductive hypothesis and $|a_{l l'}|\leq K^{|m_{l}-m_{l'}|}$, we have
\begin{equation}
    |a'_{i j}|\leq \sum_{l=j}^{i-1} K^{m_{i}-m_{l}} (2 K)^{m_{l}-m_{j}} =K^{m_{i}-m_{j}} \sum_{l=j}^{i-1} 2^{m_{l}-m_{j}}\leq (2 K)^{m_{i}-m_{j}}.
\end{equation}
Thus the induction proves $|a'_{i j}| \leq (2 K)^{m_{i}-m_{j}}$ for $1\leq i,j\leq d$. Finally, $$\|A^{-1}\|\leq (\sum_{1\leq i,j \leq d} |a'_{i j}|^{2} )^{\frac{1}{2}}\leq d (2 K)^{m_{d}}$$ since $0<m_{1}<\cdots < m_{d}$.
\end{proof}
\subsection{Tilted rectangles}
In this section, we collect basic lemmas on functions satisfying the equation $H u=\lambda u$ on a tilted rectangle (see Definition \ref{def:geo1}). The following Lemma \ref{lem:extension}, Lemma \ref{lem:extension-bound} and Lemma \ref{lem:variate-lambda} are rewrites of \cite[Lemma 3.8]{ding2020localization}, \cite[Lemma 3.10]{ding2020localization} and \cite[Lemma 3.11]{ding2020localization} respectively. They are modified to depend on $\Bar{V}$ explicitly. 

We will keep several notations from \cite[Section 3]{ding2020localization}. In particular, we work in the tilted coordinates of 
\begin{equation}\label{eq:def-tilted-coordinate}
    (s,t)=(x+y,x-y).
\end{equation}
Under coordinate transformation \eqref{eq:def-tilted-coordinate}, the transformed lattice is $\widetilde{\Z^{2}}=\{(s,t)\in \Z^{2}:\text{$s-t$ is even}\}$.
The equation 
\begin{equation}\label{eq:extension}
    H u= \lambda u
\end{equation}
becomes
\begin{equation}\label{eq:equ-on-tilted-coordinate}
    u(s,t)=(4+V(s-1,t-1)-\lambda) u(s-1,t-1) -u(s-2,t)-u(s-2,t-2)-u(s,t-2).
\end{equation}

Given two intervals $J_1,J_2\subset \Z$, by Definition \ref{def:geo1}, under the coordinate transformation, the tilted rectangle $R_{J_{1},J_{2}}\subset \Z^{2}$ is transformed to $$\widetilde{R_{J_{1},J_{2}}}=\{(s,t)\in J_{1}\times J_{2}: \text{$s-t$ is even}\}$$
in the new lattice $\widetilde{\Z^{2}}$. With a little abuse of notations, we also use $R_{J_{1},J_{2}}$ to denote $\widetilde{R_{J_{1},J_{2}}}$ for the rest of this section. 

\begin{defn}
Given integers $a_{1}<a_{2}$ and $b_{1}<b_{2}$,
the west boundary of the tilted rectangle is
$$\partial^{w}R_{[a_{1},a_{2}],[b_{1},b_{2}]}=R_{[a_{1},a_{2}],[b_{1},b_{1}+1]}\cup R_{[a_{1},a_{1}+1],[b_{1},b_{2}]}.$$
\end{defn}

The following lemma is a rewrite of \cite[Lemma 3.8]{ding2020localization} and it follows from the same proof of \cite[Lemma 3.8]{ding2020localization}.
\begin{lemma}\label{lem:extension}
Suppose energy $\lambda\in \R$, real number $\Bar{V}\in \R$ and integers $a_{1}< a_{2},b_{1}< b_{2}$. Then every function $u:\partial^{w}R_{[a_{1},a_{2}],[b_{1},b_{2}]} \rightarrow \R$ has a unique extension $$u^{0}=E^{(\lambda,\Bar{V})}_{R_{[a_{1},a_{2}],[b_{1},b_{2}]}}(u):R_{[a_{1},a_{2}],[b_{1},b_{2}]}\rightarrow \R$$ such that 
\begin{equation}\label{eq:equation-in-extension}
    H u^{0}=\lambda u^{0}
\end{equation}
in $R_{[a_{1}+1,a_{2}-1],[b_{1}+1,b_{2}-1]}$.
Moreover, $E^{(\lambda,\Bar{V})}_{R_{[a_{1},a_{2}],[b_{1},b_{2}]}}$ is a random linear operator and is $V|_{R_{[a_{1}+1,a_{2}-1],[b_{1}+1,b_{2}-1]}}$-measurable.
\end{lemma}

\begin{rem}
Given energy $\lambda$, real $\Bar{V}$ and integers $a_{1}<a_{2}$ and $b_{1}<b_{2}$, we also denote $E^{(\lambda,\Bar{V})}_{R_{[a_{1},a_{2}],[b_{1},b_{2}]}}$ by $E^{(\lambda,\Bar{V})}_{[a_{1},a_{1}],[b_{1},b_{2}]}$ for simplicity. When energy $\lambda$ and real number $\Bar{V}$ are given in context, we also omit $\lambda,\Bar{V}$ and denote $E^{(\lambda,\Bar{V})}_{R_{[a_{1},a_{2}],[b_{1},b_{2}]}}$ by $E_{R_{[a_{1},a_{2}],[b_{1},b_{2}]}}$ and $E^{(\lambda,\Bar{V})}_{[a_{1},a_{2}],[b_{1},b_{2}]}$ by $E_{[a_{1},a_{2}],[b_{1},b_{2}]}$.
\end{rem}
\begin{lemma}\label{lem:extension-bound}
Suppose we have real numbers $\lambda,\Bar{V}$ and integers $a_{1}\leq a_{2}$ and $b_{1}\leq b_{2}$. Assume $\lambda\in [-2,10]$ and $\Bar{V}\geq 2$.
If $H u= \lambda u$ in $R_{[a_{1}+1,a_{2}-1],[b_{1}+1,b_{2}-1]}$ and $\|u\|_{\ell^{\infty}(\partial^{w}R_{[a_{1},a_{2}],[b_{1},b_{2}]})}=1$, then 
\begin{align}
    &\|u\|_{\ell^{\infty}(R_{[a_{1},a_{2}],[b_{1},b_{2}]})} \leq (\Bar{V} (a_{2}-a_{1}+1))^{C_{1} (b_{2}-b_{1}-1)\vee 0} \label{eq:exten-bound-1}\\
    &\|u\|_{\ell^{\infty}(R_{[a_{1},a_{2}],[b_{1},b_{2}]})} \leq (\Bar{V} (b_{2}-b_{1}+1))^{C_{1} (a_{2}-a_{1}-1)\vee 0} \label{eq:exten-bound-2}
\end{align}
for a numerical constant $C_{1}$.
\end{lemma}
\begin{proof}
We only prove \eqref{eq:exten-bound-1}, and \eqref{eq:exten-bound-2} follows by symmetry.

Assume without loss of generality that $a_{1}=b_{1}=1$. We prove 
\begin{equation}\label{eq:induction-extension-bound}
    |u(s,t)|\leq (C \Bar{V} s)^{(t-2)\vee 0} 
\end{equation}
by induction on $(s,t)\in R_{[1,a_{2}],[1,b_{2}]}$. Here, $C\geq 10$ is a universal constant to be determined. Firstly, if $(s,t)\in R_{[1,a_{2}],[1,2]}$, then $t\leq 2$ and $$|u(s,t)|\leq 1 \leq (C \Bar{V} s)^{(t-2)\vee 0}$$ by assumption. Secondly, if $(s,t)\in R_{[1,2],[3,b_{2}]}$, then $|u(s,t)|\leq 1 \leq (C \Bar{V} s)^{(t-2)\vee 0}$ by assumption. Now suppose $(s,t)\in R_{[3,a_{2}],[3,b_{2}]}$ and assume $\eqref{eq:induction-extension-bound}$ holds for $(s',t')\in R_{[1,s],[1,t]}\setminus \{(s,t)\}$. We use \eqref{eq:equ-on-tilted-coordinate} to get
\begin{align*}
    &|u(s,t)|\\
    \leq&(14+\Bar{V})|u(s-1,t-1)| +|u(s-2,t)|+|u(s-2,t-2)|+|u(s,t-2)|\\
    \leq&(14+\Bar{V})(C \Bar{V} s)^{t-3} +(C \Bar{V} (s-2))^{t-2}+(C \Bar{V} (s-2))^{(t-4)\vee 0}+(C \Bar{V} s)^{(t-4)\vee 0}\\
    \leq&(16+\Bar{V})(C \Bar{V} s)^{t-3} +(C \Bar{V} (s-2))^{t-2}\\
    \leq&(C \Bar{V} s)^{t-2}\left(\frac{16+\Bar{V}}{C \Bar{V}}s^{-1} + \left(\frac{s-2}{s}\right)^{t-2}\right)\\
    \leq&(C \Bar{V} s)^{t-2}\left(\frac{16+\Bar{V}}{C \Bar{V}}s^{-1} + 1-2s^{-1}\right)\\
    \leq &(C \Bar{V} s)^{t-2}.
\end{align*}
Here, we used $|\lambda|\leq 10$, $\Bar{V}\geq 2$ and $C\geq 10$.
\end{proof}
The following lemma follows the same proof of \cite[Lemma 3.11]{ding2020localization}.
\begin{lemma}\label{lem:variate-lambda}
Suppose real numbers $\lambda_{1},\lambda_{2},\Bar{V}$ and positive integers $a,b>2$. Assume $\lambda_{1}, \lambda_{2}\in [-2,10]$ and $\Bar{V}\geq 2$.
If $H u_{1}=\lambda_{1} u_{1}$ and $H u_{2}=\lambda_{2} u_{2}$ in $R_{[2,a-1],[2,b-1]}$ and $u_{1}=u_{2}$ in $\partial^{w}R_{[1,a],[1,b]}$, then 
\begin{equation}
    \|u_{1}-u_{2}\|_{\ell^{\infty}(R_{[1,a],[1,b]})}\leq (a \Bar{V})^{C_{2} b}\|u_{1}\|_{\ell^{\infty}(\partial^{w} R_{[1,a],[1,b]})} |\lambda_{1}-\lambda_{2}|,
\end{equation}
where $C_{2}$ is a numerical constant.
\end{lemma}
\subsection{Key lemmas}\label{sec:key-lemmas}
The main task in this subsection is to prove the following Lemma \ref{lem:transfer-u_3} which will be used to prove the key estimate Lemma \ref{lem:key-lemma}. 
\begin{lemma}\label{lem:transfer-u_3}
There are constants $\alpha_{1}>1>c_{4}>0$ such that, if
\begin{enumerate}
    \item integers $a>b>\alpha_{1}$ with $10b \leq a \leq 60b$,
    \item $\lambda_{0}\in[0,8]$ and $\Bar{V}\geq 2$,
    \item $\Theta\subset \Z^{2}$ is $(c_{4},-)$-sparse in $R_{[1,a],[1,b]}$,
    \item $V':\Theta \rightarrow \{0,\Bar{V}\}$, 
    \item $\mathcal{E}_{tr}(R_{[1,a],[1,b]})$ denotes the event that,
    \begin{equation}\label{eq:cond-u-5.15}
        \left\{\begin{array}{l}
    \text{$H u= \lambda_{0} u$ in $R_{[2,a-1],[2,b-1]}$ }\\
    \text{$\|u\|_{\ell^{\infty}(R_{[1,2],[1,b]})} = 1$} \\
    \text{$u\equiv 0$ on $R_{[1,a],[1,2]}$} 
    \end{array}\right. 
    \end{equation}
    implies $|u|\geq (a\Bar{V})^{-\alpha_{1} a}$ on a $\frac{1}{10^{6}}$ fraction of $R_{[1,a],[b-1,b]}$,
\end{enumerate}
then $\mathbb{P}\left[\mathcal{E}_{tr}(R_{[1,a],[1,b]})\big|\; V|_{\Theta}=V'\right]\geq 1-\exp(-c_{4} a)$. 

\end{lemma}
In the proof of \cite[Lemma 3.13]{ding2020localization}, the authors considered an $\varepsilon$-net in the space of all functions satisfying \eqref{eq:cond-u-5.15} and then they proved that the desired property holds with high probability for each function in the $\varepsilon$-net and finally they used a probability union bound to conclude the proof of \cite[Lemma 3.13]{ding2020localization}. 

The $\varepsilon$-net method in \cite{ding2020localization} requires $a \geq  C b^{2} \log(a)$. Our method does not use $\varepsilon$-net and proves the desired property for all functions satisfying \eqref{eq:cond-u-5.15} in a single step. This allows to prove the lemma in the case where $a$ is roughly linear to $b$.

Our method exploits the exact formula for functions satisfying \eqref{eq:cond-u-5.15}. Let us give an intuitive argument here for the simple case where $H u=\lambda u$ in $R_{[2,a-1],[2,b-1]}$ with $a\geq 10 b$. We claim that, with high probability, $$u|_{R_{[1,a],[1,2]}\cup R_{[1,a],[b-1,b]}}=0$$ will force $u\equiv 0$ in $R_{[1,a],[1,b]}$ (which is implied by Observation \ref{obs:tilted} and linearity).

To see this, by Lemma \ref{lem:extension}, we can regard $u|_{R_{[1,a],[b-1,b]}}$ as the image of $u|_{R_{[1,a],[1,2]}\cup R_{[1,2],[3,b]}}$ under a linear mapping determined by the potential $V$.
We assume $u|_{R_{[1,a],[1,2]}}=0$ and $u(1,3)=1$ (recall that we are working in the tilted coordinate \eqref{eq:def-tilted-coordinate}). It suffices to prove that, with high probability,
\begin{equation}\label{eq:equaton(107)}
    \text{$u|_{R_{[1,a],[b-1,b]}}\not=0$ for \emph{any} choice of $u|_{R_{[1,2],[4,b]}}$.}
\end{equation}
Once this is proved, $u|_{R_{[1,a],[1,2]}\cup R_{[1,a],[b-1,b]}}=0$ will force $u(1,3)=0$ and further $u|_{R_{[1,a],[1,3]}}=0$. By repeating this argument, $u|_{R_{[1,a],[1,2]}\cup R_{[1,a],[b-1,b]}}=0$ will force $u(s,t)=0$ for each $(s,t)\in R_{[1,2],[3,b]}$ and then $u\equiv 0$ in $R_{[1,a],[1,b]}$ by Lemma \ref{lem:extension}. 
%Otherwise, let $(s,t)\in R_{[1,2],[3,b]}$ be with smallest $t$ such that $u(s,t)\not = 0$ and we argue with $R_{[1,a],[t-2,b]}$ instead of $R_{[1,a],[1,b]}$ then do a union bound for probability. Note that, if $u\equiv 0$ on $R_{[1,2],[3,b]}$, then by Lemma \ref{lem:extension}, $u\equiv 0$ on $R_{[1,a],[1,b]}$ and our claim already follows.

To see \eqref{eq:equaton(107)}, let us first calculate $u|_{R_{[1,a],\{3\}}}$. Using equation \eqref{eq:equ-on-tilted-coordinate} for $t=2$, we have $u(s,3)+u(s-2,3)=0$ for any odd number $s\in [3,a]$. Since $u(1,3)=1$, inductively we have
\begin{equation}\label{eq:transfer-intro-3rd-column}
    u(s,3)=(-1)^{\frac{s-1}{2}}
\end{equation}
for odd $s\in [1,a]$. Let us calculate further $u|_{R_{[1,a],\{4\}}}$. Using equations \eqref{eq:equ-on-tilted-coordinate} and \eqref{eq:transfer-intro-3rd-column} for $t=3$, we have $u(s,4)+u(s-2,4)=(-1)^{\frac{s-2}{2}}(4+ V(s-1,3)-\lambda )$ for any even number $s\in [3,a]$. Inductively, for even $s\in [1,a]$,
\begin{equation}\label{eq:transfer-intro-4th-column}
    u(s,4)=(-1)^{\frac{s-2}{2}}\left(u(2,4) + \sum_{2<s'<s\text{, $s'$ is odd}} \left(4+V(s',3)-\lambda\right) \right).
\end{equation}
By equations \eqref{eq:transfer-intro-3rd-column} and \eqref{eq:transfer-intro-4th-column}, we can write $u|_{R_{[1,a],[3,4]}}=u^{(1)}+u^{(2)}+u(2,4)u^{(3)}$ with $u^{(i)}\in \ell^{2}(R_{[1,a],[3,4]})$ for $i=1,2,3$. Here, we have $u^{(1)}|_{R_{[1,a],\{3\}}} = 0$ and $u^{(1)}|_{R_{[1,a],\{4\}}}=A (\vec{V})$ in which $A$ is a triangular matrix and the vector
\begin{equation}\label{eq:transfer-intro-def-V}
    \vec{V}=\left(V(3,3),V(5,3),\cdots,V(a-i_{a},3)\right)
\end{equation}
satisfies $i_{a}\in \{1,2\}$ and $a-i_{a}$ is an odd number. Moreover, $u^{(2)}(s,3)=(-1)^{\frac{s-1}{2}}$ for odd $s \in [1,a]$ and $u^{(2)}(s,4) = (-1)^{\frac{s-2}{2}}\frac{(s-2)(4-\lambda)}{2}$ for even $s \in [1,a]$; $u^{(3)}|_{R_{[1,a],\{3\}}} = 0$ and $u^{(3)}(s,4)=(-1)^{\frac{s-2}{2}}$ for even $s \in [1,a]$. Note that, $u^{(2)}$ and $u^{(3)}$
are independent of potential $V$ (in the sense of random variables).
By Lemma \ref{lem:extension}, $u|_{R_{[1,a],[b-1,b]}}$ is determined linearly by $u|_{\partial^{w}R_{[1,a],[3,b]}}$. Hence, there are linear operators $M_{0},M_{1}$ such that
\begin{equation}
    u|_{R_{[1,a],[b-1,b]}}= M_{0}(u^{(1)}+u^{(2)}+u(2,4)u^{(3)}) + M_{1}(u|_{R_{[1,2],[5,b]}}).
\end{equation}
Since $u^{(1)}$ is the zero extension of $A(\vec{V})$ and $u(2,4)u^{(3)}$ is determined linearly by $u(2,4)$, we have  
\begin{equation}\label{eq:sec5-intui}
    u|_{R_{[1,a],[b-1,b]}} = M( A (\vec{V})) + M_{0} (u^{(2)}) +M_{2}(u|_{R_{[1,2],[4,b]}})
\end{equation}
with linear operators $A$, $M$, $M_{0}$ and $M_{2}$ all independent of $V|_{R_{[1,a],\{3\}}}$. Thus $u|_{R_{[1,a],[b-1,b]}}=0$ implies
\begin{equation}\label{eq:transfer-intro-preV}
     M (A(\vec{V})) + M_{0} (u^{(2)}) +M_{2}(u|_{R_{[1,2],[4,b]}})=0.
\end{equation}
It will be proved later that $M$ can be regarded as a triangular matrix and the operator $M A$ is injective. Thus \eqref{eq:transfer-intro-preV} implies
\begin{equation}\label{eq:transfer-intro-V}
    \vec{V} = -(M A)^{-1}( M_{0} (u^{(2)})+ M_{2}(u|_{R_{[1,2],[4,b]}}))
\end{equation}
with $(M A)^{-1}$ defined on the range of $M A$.

However, the rank of operator $M_{2}$ is at most $|R_{[1,2],[4,b]}|$ which is bounded by $b\leq \frac{a}{10}$. Thus, conditioning on $V|_{R_{[1,a],[1,b]}\setminus R_{[1,a],\{3\}}}$, $$\left\{-(M A)^{-1}( M_{0} (u^{(2)})+ M_{2}(v)):v \in \ell^{2}(R_{[1,2],[4,b]})\right\}$$ is an affine subspace with dimension no larger than $\frac{a}{10}$. Recall \eqref{eq:transfer-intro-def-V}, $\vec{V}$ is $V|_{R_{[1,a],\{3\}}}$-measurable and can be regarded as a random element in a Boolean cube with dimension larger than $\frac{a}{3}$. Thus by Lemma \ref{lem:Boolean-cube}, with probability no less than $1-2^{\frac{a}{10}-\frac{a}{3}+1}>1-\exp(-c a)$, \eqref{eq:transfer-intro-V} fails for \emph{any} $u|_{R_{[1,2],[4,b]}}$. Our claim follows.

The proof of Lemma \ref{lem:transfer-u_3} below makes the above argument quantitative. Lemma \ref{lem:transfer-u_3} is also the key in proving Lemma \ref{lem:key-lemma}. We start by defining the operator $M$ in \eqref{eq:sec5-intui} and prove its triangular matrix structure.
\begin{defn}
Given $S_{1}\subset S_{2}\subset \widetilde{\Z^{2}}$, we use $P^{S_{2}}_{S_{1}}:\ell^{2}(S_{2})\rightarrow \ell^{2}(S_{1})$ to denote the restriction operator from $S_{2}$ to $S_{1}$. i.e. $P^{S_{2}}_{S_{1}}(u)=u|_{S_{1}}$ for $u\in \ell^{2}(S_{2})$. We use $I^{S_{2}}_{S_{1}}$ to denote the adjoint operator $(P^{S_{2}}_{S_{1}})^{\dag}$, i.e. $I^{S_{2}}_{S_{1}}(u)=u$ on $S_{1}$ and $I^{S_{2}}_{S_{1}}(u)=0$ on $S_{2}\setminus S_{1}$ for each $u\in \ell^{2}(S_{1})$.
\end{defn}
\begin{defn}\label{def:transfer-matrix}
Given energy $\lambda \in [0,8]$, real number $\Bar{V}$ and integers $a,k,k'$ such that $a> 1$ and $k<k'$, we define the linear operator $$M_{[1,a]}^{k,k'}: \ell^{2}(R_{[1,a],\{k\}})\rightarrow \ell^{2}(R_{[1,a],\{k'\}})$$ as follows:
%Suppose $u: R_{[1,a],\{k\}}\rightarrow \R$. Let $u':\partial^{w}R_{[1,a],[k-1,k']}\rightarrow \R$ satisfy
%$u'=u$ on $R_{[1,a],\{k\}}$ and $u'=0$ on $\partial^{w}R_{[1,a],[k-1,k']}\setminus R_{[1,a],\{k\}}$.
\begin{equation}
    M_{[1,a]}^{k,k'}=P^{R_{[1,a],[k-1,k']}}_{R_{[1,a],\{k'\}}} E_{[1,a],[k-1,k']} I^{\partial^{w}R_{[1,a],[k-1,k']}}_{R_{[1,a],\{k\}}}.
\end{equation}
\end{defn}
\begin{lemma}
Given energy $\lambda \in [0,8]$, real number $\Bar{V}$ and integers $a,k,k'$ such that $a> 1$ and $k<k'$, the linear operator $M_{[1,a]}^{k,k'}$ is $V|_{R_{[2,a-1],[k,k'-1]}}$-measurable.
\end{lemma}
\begin{proof}
By Lemma \ref{lem:extension}, the extension operator $E_{[1,a],[k-1,k']}$ is $V|_{R_{[2,a-1],[k,k'-1]}}$-measurable, thus $M_{[1,a]}^{k,k'}$ is also $V|_{R_{[2,a-1],[k,k'-1]}}$-measurable.
\end{proof}
Given $(s,t)\in \widetilde{\Z^{2}}$,
we use $\delta_{(s,t)}$ to denote the function that equals $1$ on $(s,t)$ and $0$ elsewhere. 

\begin{prop}\label{prop:transfer-conditions-col}
Suppose we have energy $\lambda\in [0,8]$, real number $\Bar{V}>2$ and integers $a,k,k',s,s'$ such that $a\geq 4$, $k<k'$, $(s,k),(s',k')\in \widetilde{\Z^{2}}$ and $4 \leq s,s' \leq a$. Then
\begin{align}
|\langle \delta_{(s',k')},M_{[1,a]}^{k,k'} \delta_{(s,k)} \rangle| =
\begin{cases}
     0 &\text{ if $s'<s$}\\
     1 &\text{ if $s'=s$}\\
\end{cases}
\end{align}
and 
\begin{equation}
    |\langle \delta_{(s',k')},M_{[1,a]}^{k,k'} \delta_{(s,k)} \rangle|  \leq ((k'-k+2)\Bar{V})^{C_{1}(s'-s)} \;\;\text{ if $s'>s$}.
\end{equation}
Here, $C_{1}$ is the constant in Lemma \ref{lem:extension-bound}.
\end{prop}
\begin{proof}
    Denote $R_{1}=R_{[1,a],[k-1,k']}$. Assume the function $u:R_{1}\rightarrow \R$ satisfies  $u|_{\partial^{w}R_{1}}=\delta_{(s,k)}$ and $H u=\lambda u$ in $R_{[2,a-1],[k,k'-1]}$. It suffices to show that 
    \begin{align}
    u(s',k') &= 0 &\text{ if $s'<s$}\label{eq:transf-cond-1} \\
    u(s',k') &= (-1)^{\frac{k'-k}{2}} &\text{ if $s'=s$}\label{eq:transf-cond-2}  \\
    |u(s',k')|& \leq ((k'-k+2)\Bar{V})^{C_{1}(s'-s)} &\text{ if $s'>s$}.\label{eq:transf-cond-3}
\end{align}
    Firstly, since $u=0$ on $\partial^{w}R_{[1,s-1],[k-1,k']}$, we have $u=0$ on $R_{[1,s-1],[k-1,k']}$ by Lemma \ref{lem:extension}. Thus $\eqref{eq:transf-cond-1}$ holds.
    Secondly, we inductively prove $u(s,k+2i)=(-1)^{i}$ for $i=0,1,\cdots,\left\lfloor\frac{k'-k}{2}\right\rfloor$. This is true for $i=0$ since $u|_{\partial^{w}R_{1}}=\delta_{(s,k)}$. Suppose $u(s,k+2i)=(-1)^{i}$ for some $i<\left\lfloor\frac{k'-k}{2}\right\rfloor$. Since $s \geq 4$, we can use the equation $H u=\lambda u$ at the point $(s-1,k+2i+1)$. By \eqref{eq:transf-cond-1}, we have $u(s,k+2i)+u(s,k+2i+2)=0$ and thus $u(s,k+2i+2)=(-1)^{i+1}$. By induction we have $\left|u\left(s,k+2\left\lfloor\frac{k'-k}{2}\right\rfloor\right)\right|=1$. Since $s=s'$ implies $k-k'$ is even, \eqref{eq:transf-cond-2} follows. 
    
    Finally we suppose $s'>s$. By \eqref{eq:transf-cond-1} and $\eqref{eq:transf-cond-2}$, $\|u\|_{\ell^{\infty}(\partial^{w}R_{[s-1,s'],[k-1,k']})}=1$. Then by \eqref{eq:exten-bound-2} in Lemma \ref{lem:extension-bound},  $$\|u\|_{\ell^{\infty}(R_{[s-1,s'],[k-1,k']})}\leq (\Bar{V}(k'-k+2))^{C_{1}(s'-s)}.$$ In particular, $|u(s',k')|\leq (\Bar{V}(k'-k+2))^{C_{1}(s'-s)}$ and \eqref{eq:transf-cond-3} follows.
\end{proof}

\begin{cor}\label{cor:transfer-bound-M}
Suppose we have energy $\lambda\in [0,8]$, real number $\Bar{V}>2$ and integers $a,k,k'$ such that $a\geq 6$ and $k<k'$. Assume $k$ and $k'$ have the same parity. Suppose $$S_{1} \subset R_{[4,a-1],\{k'\}}$$ and let $S_{2}=\{(s,k): (s,k')\in S_{1}\}\subset R_{[4,a-1],\{k\}}$. Then 
\begin{equation}
    \|(P^{R_{[1,a],\{k'\}}}_{S_{1}} M_{[1,a]}^{k,k'} I^{R_{[1,a],\{k\}}}_{S_{2}})^{-1}\|\leq  a (2 \Bar{V} (k'-k+2))^{2C_{1}a}.
\end{equation}
\end{cor}
\begin{proof}
    By Proposition \ref{prop:transfer-conditions-col}, $P^{R_{[1,a],\{k'\}}}_{R_{[4,a-1],\{k'\}}} M_{[1,a]}^{k,k'} I^{R_{[1,a],\{k\}}}_{R_{[4,a-1],\{k\}}}$ can be regarded as an upper triangular matrix $\left(a_{i j}\right)_{1\leq i,j \leq d}$ such that $|a_{i i}|=1$ and $$|a_{i j}|\leq ((k'-k+2)\Bar{V})^{2C_{1}|i-j|}$$ for $1\leq i,j\leq d$. Here, $d=|R_{[4,a-1],\{k\}}|\leq a$.
    
    Since $P^{R_{[1,a],\{k'\}}}_{S_{1}} M_{[1,a]}^{k,k'} I^{R_{[1,a],\{k\}}}_{S_{2}}$ can be regarded as a principal submatrix which is also an upper triangular matrix, our conclusion follows from Lemma \ref{lem:upper-triangular}.
\end{proof}

\begin{lemma}\label{lem:transfer-relation}
Suppose we have real numbers $\lambda,\Bar{V}$, integers $a>1$ and $2 < b_{*}< b$. Denote $R_{1}=R_{[1,a],[1,b]}$, $R_{2}=R_{[1,a],[1,b_{*}+1]}$ and $R_{3}=R_{[1,a],[1,b_{*}-1]}$. Then the following linear operator from $\ell^{2}(\partial^{w} R_{3})\rightarrow \ell^{2}(R_{[1,a],\{b\}})$
\begin{equation}\label{eq:transfer-relation-op}
    P^{R_{1}}_{R_{[1,a],\{b\}}} E_{R_{1}} I_{\partial^{w}R_{3}}^{\partial^{w}R_{1}}-M_{[1,a]}^{b_{*}+1,b} P_{R_{[1,a],\{b_{*}+1\}}}^{R_{2}} E_{R_{2}} I_{\partial^{w}R_{3}}^{\partial^{w}R_{2}}
\end{equation}
is independent of $V|_{R_{[1,a],\{b_{*}\}}}$ (in the sense of random variables).
\end{lemma}
Lemma \ref{lem:transfer-relation} allows us to write the operator $P^{R_{1}}_{R_{[1,a],\{b\}}} E_{R_{1}} I_{\partial^{w}R_{3}}^{\partial^{w}R_{1}}$ as the sum of two operators: a $V|_{R_{[1,a],\{b_{*}\}}}$-measurable operator and a $V|_{R_{[1,a],\{b_{*}\}}}$-independent operator. Here, the $V|_{R_{[1,a],\{b_{*}\}}}$-measurable operator can be written as the composition of a $V|_{R_{[1,a],\{b_{*}\}}}$-independent operator $M_{[1,a]}^{b_{*}+1,b}$ and the operator $P_{R_{[1,a],\{b_{*}+1\}}}^{R_{2}} E_{R_{2}} I_{\partial^{w}R_{3}}^{\partial^{w}R_{2}}$. Thus intuitively, Lemma \ref{lem:transfer-relation} says that the $V|_{R_{[1,a],\{b_{*}\}}}$-measurable ``part'' of $P^{R_{1}}_{R_{[1,a],\{b\}}} E_{R_{1}} I_{\partial^{w}R_{3}}^{\partial^{w}R_{1}}$ is ``contained'' in $P_{R_{[1,a],\{b_{*}+1\}}}^{R_{2}} E_{R_{2}} I_{\partial^{w}R_{3}}^{\partial^{w}R_{2}}$.
The proof is by direct calculation.
\begin{proof}[Proof of Lemma \ref{lem:transfer-relation}]
    Denote $R_{4}=R_{[1,a],[b_{*},b]}$ and let $u\in \ell^{2}(\partial^{w}R_{3})$.
    Let $v=E_{R_{1}} I^{\partial^{w}R_{1}}_{\partial^{w}R_{3}} (u)$, then by uniqueness in Lemma \ref{lem:extension},
    \begin{equation}\label{eq:transfer-rest-v}
        v|_{R_{[1,a],\{b\}}}=P^{R_{4}}_{R_{[1,a],\{b\}}} E_{R_{4}} (v|_{\partial^{w}R_{4}}). 
    \end{equation}
    Let $v_{1}=v|_{R_{[1,a],\{b_{*}\}}}$ and $v_{2}=v|_{R_{[1,a],\{b_{*}+1\}}}$. Note that $v|_{R_{[1,2],[b_{*},b]}}=0$. By \eqref{eq:transfer-rest-v} and linearity of $E_{R_{4}}$, 
    \begin{align}
    \begin{split}\label{eq:split-v-into-two -parts}
            &v|_{R_{[1,a],\{b\}}}\\
            &=P^{R_{4}}_{R_{[1,a],\{b\}}} E_{R_{4}} I_{R_{[1,a],\{b_{*}\}}}^{\partial^{w}R_{4}}(v_{1}) + P^{R_{4}}_{R_{[1,a],\{b\}}} E_{R_{4}} I_{R_{[1,a],\{b_{*}+1\}}}^{\partial^{w}R_{4}}(v_{2})\\
            &=P^{R_{4}}_{R_{[1,a],\{b\}}} E_{R_{4}} I_{R_{[1,a],\{b_{*}\}}}^{\partial^{w}R_{4}}(v_{1})+
            M_{[1,a]}^{b_{*}+1,b}(v_{2}).
    \end{split}
    \end{align}
    Here, we used Definition \ref{def:transfer-matrix}.
    
    By uniqueness in Lemma \ref{lem:extension}, $v_{2}=P_{R_{[1,a],\{b_{*}+1\}}}^{R_{2}} E_{R_{2}} I^{\partial^{w}R_{2}}_{\partial^{w}R_{3}} (u)$. Thus the image of $u$ under the operator $\eqref{eq:transfer-relation-op}$ is $v|_{R_{[1,a],\{b\}}}- M_{[1,a]}^{b_{*}+1,b}(v_{2})$.
    Thus by \eqref{eq:split-v-into-two -parts}, in order to prove the conclusion, it suffices to prove that the linear operator $$u\mapsto P^{R_{4}}_{R_{[1,a],\{b\}}} E_{R_{4}} I_{R_{[1,a],\{b_{*}\}}}^{\partial^{w}R_{4}}(v_{1})$$ is independent of $V|_{R_{[1,a],\{b_{*}\}}}$.
    
    To see this, note that $E_{R_{4}}$ is independent of $V|_{R_{[1,a],\{b_{*}\}}}$ by Lemma \ref{lem:extension}. On the other hand, let $R_{5}=R_{[1,a],[1,b_{*}]}$, then by uniqueness in Lemma \ref{lem:extension} again, we have $v_{1}=P^{R_{5}}_{R_{[1,a],\{b_{*}\}}} E_{R_{5}} I_{\partial^{w}R_{3}}^{\partial^{w}R_{5}}(u)$. Since $E_{R_{5}}$ is also independent of $V|_{R_{[1,a],\{b_{*}\}}}$ by Lemma \ref{lem:extension}, the conclusion follows.
\end{proof}

\begin{proof}[Proof of Lemma \ref{lem:transfer-u_3}]
    For each $(s',t')\in R_{[1,2],[3,b]}$,
    let $\mathcal{E}_{tr}^{(s',t')}$ denote the following event:
    \begin{equation}\label{eq:event-transfer}
        \left\{\begin{array}{l}
    \text{$H u= \lambda_{0} u$ in $R_{[2,a-1],[2,b-1]}$ }\\
    \text{$u(s',t')=1$} \\
    \text{$u(s,t)=0$ on $R_{[1,a],[1,2]}$}\\
    \text{$|u(s,t)|\leq (a\Bar{V})^{10C_{1} (t-t')}$ on $R_{[1,2],[1,t'-1]}$} 
    \end{array}\right.
    \end{equation}
    implies $|u|\geq (a\Bar{V})^{-\frac{1}{2}\alpha_{1} a}$ on a $\frac{1}{10^{6}}$ fraction of $R_{[1,a],[b-1,b]}$.
    
\begin{cla}\label{cla:all-s'-t'}
$\bigcap \left\{\mathcal{E}_{tr}^{(s',t')} : (s',t')\in R_{[1,2],[3,b]}\right\}\subset \mathcal{E}_{tr}(R_{[1,a],[1,b]})$ for $\alpha_{1}>20 C_{1}$.  
\end{cla}
\begin{proof}[Proof of the claim]
    Assume $\mathcal{E}_{tr}^{(s,t)}$ holds for each $(s,t)\in R_{[1,2],[3,b]}$, we prove $\mathcal{E}_{tr}(R_{[1,a],[1,b]})$ also holds.
    
    Given any $u: R_{[1,a],[1,b]}\rightarrow \R$ satisfying \eqref{eq:cond-u-5.15}, let $(s',t')\in R_{[1,2],[3,b]}$ maximize $(a\Bar{V})^{-10C_{1}t'}|u(s',t')|$. Then $\|u\|_{\ell^{\infty}(R_{[1,2],[1,b]})} = 1$ implies $$|u(s',t')|\geq (a\Bar{V})^{-10C_{1}b}.$$ Let $\Tilde{u}=\frac{u}{u(s',t')}$, then $\Tilde{u}$ satisfies \eqref{eq:event-transfer} and thus $|\Tilde{u}|\geq (a\Bar{V})^{-\frac{1}{2}\alpha_{1} a}$ on a $\frac{1}{10^{6}}$ fraction of $R_{[1,a],[b-1,b]}$. Hence $|u|\geq (a\Bar{V})^{-(\frac{1}{2}\alpha_{1}+10C_{1}) a}$ on a $\frac{1}{10^{6}}$ fraction of $R_{[1,a],[b-1,b]}$. The claim follows from $\alpha_{1}>20C_{1}$.
\end{proof}

\begin{cla}\label{cla:special-t'}
If $t'\in \{b-1,b\}$, then $\mathbb{P}\left[\mathcal{E}_{tr}^{(s',t')} \big|\; V|_{\Theta}=V'\right]=1$.
\end{cla}
\begin{proof}
    If $t'\in \{b-1,b\}$ and $u$ satisfies \eqref{eq:event-transfer}, we claim that 
    \begin{equation}\label{eq:bound-in-claim-5.17}
        \|u\|_{\ell^{\infty}(R_{[1,a],\{t\}})}\leq (a\Bar{V})^{5C_{1}(t-t')}
    \end{equation}
    for each $t=1,\cdots,t'-1$ and we prove \eqref{eq:bound-in-claim-5.17} by induction. For $t=1,2$, this is true since $u=0$ on $R_{[1,a],[1,2]}$. Suppose our claim holds up to $t<t'-1$, using equation \eqref{eq:equ-on-tilted-coordinate} on $R_{[1,a],\{t\}}$ and inductive hypothesis, we have
    \begin{equation}\label{eq:on-the-line-t+1}
        |u(s,t+1)+u(s+2,t+1)|\leq |16+\Bar{V}|(a\Bar{V})^{5C_{1}(t-t')}
    \end{equation}
    for $s\in [1,a-2]$ with the same parity as $t+1$.
    By \eqref{eq:event-transfer}, $$|u(s_{0},t+1)|\leq (a\Bar{V})^{10C_{1}(t+1-t')}$$ for $s_{0}\in \{1,2\}$ with the same parity as $t+1$. Recursively using \eqref{eq:on-the-line-t+1}, we have 
    $$|u(s,t+1)|\leq s(16+\Bar{V})(a\Bar{V})^{5C_{1}(t-t')}\leq (a\Bar{V})^{5C_{1}(t+1-t')}$$
    for any $s\in [1,a]$ with the same parity as $t+1$. Thus induction proves \eqref{eq:bound-in-claim-5.17} and $$\|u\|_{\ell^{\infty}(R_{[1,a],\{t'-2,t'-1\}})}\leq (a\Bar{V})^{-5C_{1}}.$$ Using equation \eqref{eq:equ-on-tilted-coordinate} on $R_{[1,a],\{t'-1\}}$, we have 
    \begin{equation}\label{eq:last-line-bdd}
        |u(s,t')+u(s+2,t')|\leq |16+\Bar{V}|(a\Bar{V})^{-5C_{1}}\leq (a\Bar{V})^{-2}
    \end{equation}
    for $s\in [1,a-2]$ with the same parity as $t'$. Since $u(s',t')=1$, using \eqref{eq:last-line-bdd} recursively, we have $|u(s,t')|\geq \frac{1}{2}$ for any $s\in [1,a]$ with the same parity as $t'$. Thus $\mathcal{E}_{tr}^{(s',t')}$ holds since $t'\in \{b-1,b\}$.
\end{proof}
\begin{cla}\label{cla:expand-u-first}
Suppose $(s',t')\in R_{[1,2],[3,b-2]}$. Let $s''\in \{1,2\}$ and $b_{0}\in \{b-1,b\}$ both have the same parity as $t'+1$. Then there exist
\begin{enumerate}
    \item operators $A_{1}:\ell^{2}(R_{[1,2],[1,t'-1]})\rightarrow \ell^{2}(R_{[1,a],\{b_{0}\}})$ and $A_{2}:\ell^{2}(R_{[1,2],[t'+1,b]})\rightarrow \ell^{2}(R_{[1,a],\{b_{0}\}})$ which are independent of $V|_{R_{[1,a],\{t'\}}}$,
    \item vector $v^{*}\in \ell^{2}(R_{[1,a],\{b_{0}\}})$ which is independent of $V|_{R_{[1,a],\{t'\}}}$,
    \item vector $\vec{V}_{t'}\in \mathbb{R}^{d_{0}}$ with $d_{0}=|R_{[1,a],\{t'+1\}}|-1$ defined by
    \begin{equation}\label{eq:transfer-potential-vector-on-t'}
        \vec{V}_{t'}=\sum_{i=1}^{d_{0}} V(s''+2i-1,t') e_{i} 
    \end{equation}
    where $\{e_{i}:1\leq i\leq d_{0}\}$ is the standard basis of $\mathbb{R}^{d_{0}}$,
    \item $A_{0}:\R^{d_{0}}\rightarrow \ell^{2}(R_{[1,a],\{t'+1\}})$ defined as follows: for any $(s,t'+1)\in R_{[1,a],\{t'+1\}}$ and $i\in \{1,\cdots,d_{0}\}$,
    \begin{equation}\label{eq:def-A_0}
    \langle \delta_{(s,t'+1)} ,A_{0} e_{i}\rangle=
        \left\{
        \begin{array}{l}
        \text{$(-1)^{\frac{s-s'-1}{2}}$  \; if $s>s''$ and $1\leq i\leq \frac{s-s''}{2}$}\\
        \text{$0$ \qquad\qquad\;\; otherwise,}
    \end{array}\right.
    \end{equation}
\end{enumerate}
such that the following holds.

For any $u$ satisfying \eqref{eq:event-transfer}, there exists $u^{*} \in \ell^{2}(R_{[1,a],\{t'+1\}})$ with
\begin{equation}\label{eq:transfer-bound-u*}
    \|u^{*}\|\leq (a\Bar{V})^{-5},
\end{equation}such that
    \begin{equation}\label{eq:tranfer-combine}
        u|_{R_{[1,a],\{b_{0}\}}}=M_{[1,a]}^{t'+1,b_{0}}(u^{*}+ A_{0}(\vec{V}_{t'}))+ A_{1}(u|_{R_{[1,2],[1,t'-1]}})+A_{2}(u|_{R_{[1,2],[t'+1,b]}})+v^{*}.
    \end{equation}
\end{cla}
\begin{figure}
    \centering
    \includegraphics{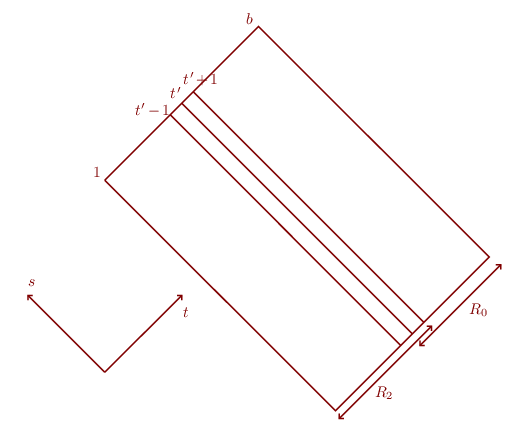}
    \caption{An illustration for tilted rectangles $R_{0} = R_{[1,a],[t',b]}$ and $R_{2} = R_{[1,a],[1,t'+1]}$ which are contained in $R_{1} = R_{[1,a],[1,b]}$.}
    \label{fig:tilted}
\end{figure}
\begin{proof}
    Assume $u$ satisfies \eqref{eq:event-transfer}. Denote $R_{1}=R_{[1,a],[1,b]}$. Let $u_{0}=\delta_{(s',t')}$ on $\partial^{w}R_{1}$, $u_{1}=u|_{R_{[1,2],[1,t'-1]}}$ and $u_{2}=u|_{R_{[1,2],[t'+1,b]}}$. Then $u$ is determined by $u_{1}$ and $u_{2}$ since we can decompose 
    \begin{align}
     u|_{\partial^{w}R_{1}}&=u'_{1}+u'_{2}+u_{0},
    \end{align}
    where $u'_{1}=I_{R_{[1,2],[1,t'-1]}}^{\partial^{w}R_{1}}u_{1}$ and $u'_{2}=I_{R_{[1,2],[t'+1,b]}}^{\partial^{w}R_{1}}u_{2}$.
    Thus $$u=E_{R_{1}}(u'_{1})+E_{R_{1}}(u'_{2})+E_{R_{1}}(u_{0})$$ and
    \begin{align}
        &u|_{R_{[1,a],\{b_{0}\}}}\\
        &=P^{R_{1}}_{R_{[1,a],\{b_{0}\}}} E_{R_{1}}(u'_{1})+P^{R_{1}}_{R_{[1,a],\{b_{0}\}}}E_{R_{1}}(u'_{2})+P^{R_{1}}_{R_{[1,a],\{b_{0}\}}}E_{R_{1}}(u_{0}).\label{eq:transfer-decompose}
    \end{align}
    We analyse each of three terms in \eqref{eq:transfer-decompose} and will arrive at equation \eqref{eq:tranfer-combine}. More specifically, we will derive the correspondence between terms in \eqref{eq:tranfer-combine} and \eqref{eq:transfer-decompose} as follows:
    \begin{enumerate}
        \item $P^{R_{1}}_{R_{[1,a],\{b_{0}\}}} E_{R_{1}}(u'_{1})=A_{1}(u_{1})+M^{t'+1,b_{0}}_{[1,a]}(u^{*})$, 
        \item $P^{R_{1}}_{R_{[1,a],\{b_{0}\}}}E_{R_{1}}(u'_{2})=A_{2}(u_{2})$,
        \item $P^{R_{1}}_{R_{[1,a],\{b_{0}\}}}E_{R_{1}}(u_{0}) = v^{*}+M^{t'+1,b_{0}}_{[1,a]} A_{0} (\vec{V}_{t'})$. 
    \end{enumerate}
    Here, $A_{0}$, $A_{1}$, $A_{2}$, $u^{*}$, $v^{*}$ and $\vec{V}_{t'}$ satisfy the properties in the conditions of this claim.

    \noindent{\textbf{The first term in \eqref{eq:transfer-decompose}:}}
    The strategy here is to apply Lemma \ref{lem:transfer-relation}.
    Note that 
    \begin{equation}\label{eq:transfer-express-u_1'}
            P^{R_{1}}_{R_{[1,a],\{b_{0}\}}} E_{R_{1}}(u'_{1})=P^{R_{1}}_{R_{[1,a],\{b_{0}\}}} E_{R_{1}}I_{R_{[1,2],[1,t'-1]}}^{\partial^{w}R_{1}}(u_{1}).
    \end{equation}
    Denote $R_{2}=R_{[1,a],[1,t'+1]}$ (see Figure \ref{fig:tilted}),
    using Lemma \ref{lem:transfer-relation} with $b_{*}=t'$, we can write 
    \begin{equation}\label{eq:transfer-decompose-first-term}
        P^{R_{1}}_{R_{[1,a],\{b_{0}\}}} E_{R_{1}}I_{R_{[1,2],[1,t'-1]}}^{\partial^{w}R_{1}}= A_{1}+M^{t'+1,b_{0}}_{[1,a]} P^{R_{2}}_{R_{[1,a],\{t'+1\}}} E_{R_{2}}I_{R_{[1,2],[1,t'-1]}}^{\partial^{w}R_{2}}.
    \end{equation}
    Here, $A_{1}:\ell^{2}(R_{[1,2],[1,t'-1]})\rightarrow \ell^{2}(R_{[1,a],\{b_{0}\}})$ is a linear operator which is independent of $V|_{R_{[1,a],\{t'\}}}$.
    We claim that 
    \begin{equation}\label{eq:transfer-v_1-on-t'+1}
        \|P^{R_{2}}_{R_{[1,a],\{t'+1\}}} E_{R_{2}}I_{R_{[1,2],[1,t'-1]}}^{\partial^{w}R_{2}} u_{1}\|_{2}\leq (\Bar{V}a)^{-5}. 
    \end{equation}
    To see this, let $v_{1}=E_{R_{2}}I_{R_{[1,2],[1,t'-1]}}^{\partial^{w}R_{2}} u_{1}$. We inductively prove that 
    \begin{equation}\label{eq:transfer-inductive-v-1}
        |v_{1}(s,t)|\leq (a\Bar{V})^{5 C_{1}(t-t')}
    \end{equation}
    for each $(s,t)\in R_{[1,a],[1,t'-1]}$. For $t=1,2$, this is true since $v_{1}=0$ on $R_{[1,a],[1,2]}$. Suppose \eqref{eq:transfer-inductive-v-1} is true for $t$ and $t+1$ and suppose $t+2<t'$, using inductive hypothesis and \eqref{eq:equ-on-tilted-coordinate} on $R_{[1,a],\{t+1\}}$, we have $$|v_{1}(s,t+2)+v_{1}(s+2,t+2)|\leq |16+\Bar{V}|(a\Bar{V})^{5C_{1}(t+1-t')}$$ for each $s\in [1,a-2]$ with the same parity as $t$. By \eqref{eq:event-transfer}, $$|v_{1}(s_{1},t+2)|\leq (a\Bar{V})^{10C_{1}(t+2-t')}$$ for $s_{1}\in \{1,2\}$ with the same parity as $t+2$. We recursively have 
    \begin{equation}
        |v_{1}(s,t+2)|\leq s(16+\Bar{V})(a\Bar{V})^{5C_{1}(t+1-t')}\leq (a\Bar{V})^{5C_{1}(t+2-t')}
    \end{equation}
    for each $s\in [1,a]$ with the same parity as $t+2$. Thus by induction we have $|v_{1}(s,t)|\leq (a\Bar{V})^{-5C_{1}}$ for $(s,t)\in R_{[1,a],[1,t'-1]}$. Finally, since $v_{1}=0$ on $R_{[1,2],[t',t'+1]}$, we have $$\|v_{1}\|_{\ell^{\infty}(\partial^{w}R_{[1,a],[t'-2,t'+1]})}\leq (a\Bar{V})^{-5C_{1}}.$$ Thus \eqref{eq:transfer-v_1-on-t'+1} follows from Lemma \ref{lem:extension-bound} and $C_{1}\geq 10$.
    
    In conclusion, by \eqref{eq:transfer-express-u_1'}, \eqref{eq:transfer-decompose-first-term} and \eqref{eq:transfer-v_1-on-t'+1},
    \begin{equation}\label{eq:transfer-A_1}
        P^{R_{1}}_{R_{[1,a],\{b_{0}\}}} E_{R_{1}}(u'_{1})=A_{1}(u_{1})+M^{t'+1,b_{0}}_{[1,a]}(u^{*}),
    \end{equation}
    where $u^{*}= P^{R_{2}}_{R_{[1,a],\{t'+1\}}} E_{R_{2}}I_{R_{[1,2],[1,t'-1]}}^{\partial^{w}R_{2}}(u_{1})$ with 
    \begin{equation}\label{eq:control-u^*}
        \|u^{*}\|_{2}\leq (a\Bar{V})^{-5}.
    \end{equation}
    
   \noindent{\textbf{The second term in \eqref{eq:transfer-decompose}:}} We have
    \begin{align}
        P^{R_{1}}_{R_{[1,a],\{b_{0}\}}}E_{R_{1}}(u'_{2})=A_{2}(u_{2}) \label{eq:transfer-A_2}
    \end{align}
    where $A_{2}=P^{R_{1}}_{R_{[1,a],\{b_{0}\}}}E_{R_{1}}I^{\partial^{w} R_{1}}_{R_{[1,2],[t'+1,b]}}$. We claim that $A_{2}$ is independent of $V|_{R_{[1,a],\{t'\}}}$. To see this, let $v_{2}=E_{R_{1}}I^{\partial^{w} R_{1}}_{R_{[1,2],[t'+1,b]}} (u_{2})$. Since $I^{\partial^{w} R_{1}}_{R_{[1,2],[t'+1,b]}} (u_{2})=0$ on $\partial^{w}R_{[1,a],[1,t']}$, we have $v_{2} \equiv 0$ on $R_{[1,a],[1,t']}$ by Lemma \ref{lem:extension}. %Let $s''\in \{1,2\}$ be with the same parity as $t'+1$.
    Using equation $\eqref{eq:equ-on-tilted-coordinate}$ for $v_{2}$ on $R_{[1,a],\{t'\}}$, we get 
    \begin{equation}\label{eq:transfer-v_2-on-t'+1}
        v_{2}(s,t'+1)=(-1)^{\frac{s-s''}{2}} u_{2}(s'',t'+1)
    \end{equation}
    for $s \in [1,a]$ with the same parity as $s''$. By \eqref{eq:transfer-v_2-on-t'+1} and $v_{2}|_{R_{[1,a],\{t'\}}}=0$ and Lemma \ref{lem:extension}, the linear transform $$I^{*}: u_{2} \mapsto v_{2}|_{\partial^{w}R_{[1,a],[t',b]}}$$ is independent of $V|_{R_{[1,a],\{t'\}}}$.
    Note that 
    \begin{align*}
        \begin{split}
                v_{2} |_{R_{[1,a],\{b_{0}\}}}=& P^{R_{[1,a],[t',b]}}_{R_{[1,a],\{b_{0}\}}}E_{[1,a],[t',b]} (v_{2}|_{\partial^{w}R_{[1,a],[t',b]} })\\
                =& P^{R_{[1,a],[t',b]}}_{R_{[1,a],\{b_{0}\}}}E_{[1,a],[t',b]} I^{*} (u_{2}).
        \end{split}
    \end{align*}
    By Lemma \ref{lem:extension}, $E_{[1,a],[t',b]}$ is independent of $V|_{R_{[1,a],\{t'\}}}$. Since $$A_{2}= P^{R_{[1,a],[t',b]}}_{R_{[1,a],\{b_{0}\}}}E_{[1,a],[t',b]} I^{*},$$ thus $A_{2}$ is independent of $V|_{R_{[1,a],\{t'\}}}$. 
    
    \noindent{\textbf{The third term in \eqref{eq:transfer-decompose}:}}
    Let $v_{0}=E_{R_{1}}(u_{0})$. The strategy here is to express $v_{0}|_{R_{[1,a],[t',t'+1]}}$ as a function of $\vec{V}_{t'}$. We have
    \begin{equation}\label{eq:former-region}
        v_{0}|_{R_{[1,a],[1,t'-1]}}=E_{[1,a],[1,t'-1]}(u_{0}|_{\partial^{w}R_{[1,a],[1,t'-1]}})=0.
    \end{equation}
    Using equation \eqref{eq:extension} on the segment $R_{[2,a-1],\{t'-1\}}$, by \eqref{eq:former-region}, we have $$v_{0}(s,t')+v_{0}(s+2,t')=0$$ for each $(s,t')\in R_{[1,a-2],\{t'\}}$.
    Thus recursively from $v_{0}(s',t')=1$, we have
    \begin{equation}\label{eq:v_0-on-t'}
        v_{0}(s,t')=(-1)^{\frac{s-s'}{2}} 
    \end{equation}
    for $(s,t')\in R_{[1,a],\{t'\}}$. Using equation \eqref{eq:extension} on the segment $R_{[2,a-1],\{t'\}}$, we have $$v_{0}(s+1,t'+1)+v_{0}(s-1,t'+1)=(V(s,t')-\lambda_{0}+4) v_{0}(s,t')$$ for each $(s,t')\in R_{[2,a-1],\{t'\}}$. Recall $s'' \in \{1,2\}$ has the same parity as $t'+1$. Then recursively from $v_{0}(s'',t'+1)=0$ and \eqref{eq:v_0-on-t'} we have
    \begin{align}
        &v_{0}(s_{1},t'+1)\\
        =&(-1)^{\frac{s_{1}-s'-1}{2}} \sum_{\substack{s''<s<s_{1}\\ s\not \equiv s'' \bmod 2}} (V(s,t')-\lambda_{0}+4) \\
        =&(-1)^{\frac{s_{1}-s'-1}{2}}\frac{(4-\lambda_{0})(s_{1}-s'')}{2}+ (-1)^{\frac{s_{1}-s'-1}{2}}\sum_{i=1}^{\frac{s_{1}-s''}{2}}V(s''+2i-1,t')\label{eq:on-the-t'+1}
    \end{align}
    for any $s_{1}\in (s'',a]$ with the same parity as $s''$. 
    By \eqref{eq:transfer-potential-vector-on-t'} and \eqref{eq:def-A_0}, we can rewrite
    \eqref{eq:on-the-t'+1} as
    \begin{equation}
        v_{0}|_{R_{[1,a],\{t'+1\}}}=v_{*}+A_{0}(\vec{V}_{t'}),
    \end{equation}
    where $v_{*}\in \ell^{2}(R_{[1,a],\{t'+1\}})$ satisfies $$v_{*}(s,t'+1)=(-1)^{\frac{s-s'-1}{2}}\frac{(4-\lambda_{0})(s-s'')}{2},$$ for $s\in [1,a]$ with the same parity as $s''$. Hence we have 
    \begin{equation}\label{eq:transfer-v_0}
        v_{0}|_{R_{[1,a],[t',t'+1]}}=v_{**}+I^{R_{[1,a],[t',t'+1]}}_{R_{[1,a],\{t'+1\}}} A_{0} (\vec{V}_{t'}),
    \end{equation}
    where $v_{**}|_{R_{[1,a],\{t'+1\}}}=v_{*}$ and $v_{**}(s,t')=(-1)^{\frac{s-s'}{2}}$ for $s\in [1,a]$ with the same parity as $s'$.
    
    Denote $R_{0}=R_{[1,a],[t',b_{0}]}$ (see Figure \ref{fig:tilted}), then $$v_{0}|_{R_{[1,a],\{b_{0}\}}}=P^{R_{0}}_{R_{[1,a],\{b_{0}\}}} E_{R_{0}} I_{R_{[1,a],[t',t'+1]}}^{\partial^{w} R_{0}} v_{0}|_{R_{[1,a],[t',t'+1]}}.$$ Together with \eqref{eq:transfer-v_0}, we have
    \begin{align}
    \begin{split}
        &v_{0}|_{R_{[1,a],\{b_{0}\}}}\\
        &= P^{R_{0}}_{R_{[1,a],\{b_{0}\}}} E_{R_{0}} I_{R_{[1,a],[t',t'+1]}}^{\partial^{w} R_{0}}(v_{**}) + P^{R_{0}}_{R_{[1,a],\{b_{0}\}}} E_{R_{0}} I_{R_{[1,a],\{t'+1\}}}^{\partial^{w} R_{0}} A_{0} (\vec{V}_{t'})\\
        &=P^{R_{0}}_{R_{[1,a],\{b_{0}\}}} E_{R_{0}} I_{R_{[1,a],[t',t'+1]}}^{\partial^{w} R_{0}}(v_{**})+M^{t'+1,b_{0}}_{[1,a]} A_{0} (\vec{V}_{t'})\\
        &=v^{*}+M^{t'+1,b_{0}}_{[1,a]} A_{0} (\vec{V}_{t'}).
    \end{split}
    \end{align}
    Here, we used the Definition \ref{def:transfer-matrix} of $M^{t'+1,b_{0}}_{[1,a]}$, and in the last equation we denoted $$v^{*}=P^{R_{0}}_{R_{[1,a],\{b_{0}\}}} E_{R_{0}} I_{R_{[1,a],[t',t'+1]}}^{\partial^{w} R_{0}}(v_{**})$$
    which is independent of $V|_{R_{[1,a],\{t'\}}}$ by Lemma \ref{lem:extension}. %By definition of $v_{**}$, $v^{*}$ does not depend on $u_{1}$ or $u_{2}$.
    In conclusion,
    \begin{equation}\label{eq:transfer-A_0}
        P^{R_{1}}_{R_{[1,a],\{b_{0}\}}}E_{R_{1}}(u_{0}) = v_{0}|_{R_{[1,a],\{b_{0}\}}}= v^{*}+M^{t'+1,b_{0}}_{[1,a]} A_{0} (\vec{V}_{t'}).
    \end{equation}
    %where $v^{*}$ is fixed conditioning on $V|_{R_{1}\setminus R_{[1,a],\{t'\}}}$.
    
    Finally plug \eqref{eq:transfer-A_1},\eqref{eq:transfer-A_2} and $\eqref{eq:transfer-A_0}$ into equation \eqref{eq:transfer-decompose}, we have
    \begin{equation}
        u|_{R_{[1,a],\{b_{0}\}}}=M_{[1,a]}^{t'+1,b_{0}}(u^{*}+ A_{0}(\vec{V}_{t'}))+ A_{1}(u_{1})+A_{2}(u_{2})+v^{*}
    \end{equation}
    which is equation \eqref{eq:tranfer-combine} and our claim follows.
    %Here, $M_{[1,a]}^{t'+1,b_{0}}$, $A_{0}$, $A_{1}$ and $A_{2}$ are independent of $V|_{R_{[1,a],\{t'\}}}$; $v^{*}$ does not depend on $u_{1}$, $u_{2}$ or $V|_{R_{[1,a],\{t'\}}}$; and by \eqref{eq:control-u^*}, we have
    %\begin{equation}
    %    \|u^{*}\|\leq (a\Bar{V})^{-5}.
    %\end{equation}
\end{proof}

Now let $c_{4}<\frac{1}{10^{7}}$. Fix $(s',t')\in R_{[1,2],[3,b-2]}$. Since $\Theta$ is $(c_{4},-)$-sparse in $R_{[1,a],[1,b]}$, we have
\begin{equation}\label{eq:sparsity-theta}
    |\Theta \cap R_{[1,a],\{t'\}}|\leq c_{4} a \leq \frac{a}{10^{7}}.
\end{equation}
Pick $b_{0}\in \{b-1,b\}$ and $s''\in \{1,2\}$ with the same parity as $t'+1$. Denote
\begin{equation}\label{eq:define-theta}
\Theta_{*}=\left\{(s,b_{0}):(s-1,t')\in \Theta\right\}\cap R_{[4,a-1],\{b_{0}\}}.
\end{equation}
For any $S\subset R_{[4,a-1],\{b_{0}\}}$, let
$\mathcal{E}_{S}^{(s',t')}$ denote the event: 
\begin{equation}\label{eq:def-E_S}
    \text{\eqref{eq:event-transfer} implies $\|u\|_{\ell^{2}(S)}\geq (a\Bar{V})^{-\frac{1}{3}\alpha_{1} a}$.}
\end{equation}

\begin{cla}\label{cla:all-S} For any $a\geq 10^{7}$, we have
$$\bigcap \left\{\mathcal{E}_{S}^{(s',t')}:S \subset R_{[4,a-1],\{b_{0}\}}\setminus \Theta_{*}, |R_{[1,a],\{b_{0}\}}\setminus S|= \left\lfloor a/10^{5} \right\rfloor\right\}\subset \mathcal{E}_{tr}^{(s',t')}.$$
\end{cla}
\begin{proof}[Proof of the claim]
    Assume the event $\mathcal{E}_{tr}^{(s',t')}$ does not hold.
    Then we can find $u\in \ell^{2}(R_{[1,a],[1,b]})$ satisfying \eqref{eq:event-transfer} but $$|\{(s,t)\in R_{[1,a],\{b-1,b\}}:|u(s,t)|\geq (a\Bar{V})^{-\alpha_{1} a}\}|\leq 10^{-6} a.$$
    Hence by \eqref{eq:sparsity-theta},
    \begin{align*}
            &|\Theta_{*} \cup\{(s,t)\in R_{[1,a],\{b_{0}\}}:|u(s,t)|\geq (a\Bar{V})^{-\alpha_{1} a}\}| \\
            &\leq 10^{-7}a +10^{-6}a \\
            &\leq 10^{-5}a-5. 
    \end{align*}

    Thus there is $S\subset R_{[4,a-1],\{b_{0}\}}\setminus \Theta_{*}$ such that $|R_{[1,a],\{b_{0}\}}\setminus S|= \left\lfloor a/10^{5} \right\rfloor$ and $\|u\|_{\ell^{\infty}(S)}\leq (a\Bar{V})^{-\alpha_{1} a}$. This implies
    $$\|u\|_{\ell^{2}(S)}\leq a (a\Bar{V})^{-\alpha_{1} a}< (a\Bar{V})^{-\frac{1}{3}\alpha_{1} a}.$$ Hence $\mathcal{E}_{S}^{(s',t')}$ does not hold.
\end{proof}

\begin{cla}\label{cla:bound-one-S} For large enough $a$,
$\mathbb{P}\left[\mathcal{E}_{S}^{(s',t')}\big|\; V|_{\Theta}=V'\right]\geq 1-\exp(-a/50)$ holds for any subset $S\subset R_{[4,a-1],\{b_{0}\}}\setminus \Theta_{*}$ such that $|R_{[1,a],\{b_{0}\}}\setminus S|= \left\lfloor a/10^{5} \right\rfloor$.
\end{cla} 
\begin{proof}[Proof of the claim]
    Denote $R_{1}=R_{[1,a],[1,b]}$.
    It is sufficient to prove that
    \begin{equation}\label{eq:transfer-conditioning}
        \mathbb{P}\left[\mathcal{E}_{S}^{(s',t')}\big|\; V|_{\Theta\cup (R_{1}\setminus R_{[1,a],\{t'\}})}\right]\geq 1-\exp(-a/50)
    \end{equation}
    for any $S\subset R_{[4,a-1],\{b_{0}\}}\setminus \Theta_{*}$ such that $|R_{[1,a],\{b_{0}\}}\setminus S|\leq a/50$.
    We pick an arbitrary $S_{0}\subset R_{[4,a-1],\{b_{0}\}}\setminus \Theta_{*}$ with 
    \begin{equation}\label{eq:transfer-bound-S}
            |R_{[1,a],\{b_{0}\}}\setminus S_{0}|\leq a/50.
    \end{equation}
    Let $$S_{0}'=\{(s,t'+1): (s,b_{0})\in S_{0}\}\subset R_{[4,a-1],\{t'+1\}}$$ and
    $$M_{S_{0}}=P^{R_{[1,a],\{b_{0}\}}}_{S_{0}} M^{t'+1,b_{0}}_{[1,a]} I^{R_{[1,a],\{t'+1\}}}_{S_{0}'}.$$ 
    By Corollary \ref{cor:transfer-bound-M},
    \begin{equation}\label{eq:transfer-bounding-M-inverse}
        \|M_{S_{0}}^{-1}\|\leq a(2(b_{0}-t'+1)\Bar{V})^{2C_{1} a}\leq (a\Bar{V})^{3C_{1} a}.
    \end{equation}
    Let $d_{0}=|R_{[1,a],\{t'+1\}}|-1$ and $\{e_{i}\}_{i=1}^{d_{0}}$ be the standard basis in $\R^{d_{0}}$.
    For any $\mathcal{S}\subset \{1,\cdots,d_{0}\}$, let $P_{\mathcal{S}}$ be the orthogonal projection onto the span of $\{e_{i}:i\in \mathcal{S}\}$ and $P_{\mathcal{S}}^{\dag}$ be its adjoint. Denote
    \begin{equation}\label{eq:transfer-def-S_0}
        \mathcal{S}_{0}=\{(s-s'')/2: (s,t'+1)\in S'_{0}\}\subset \{1,\cdots,d_{0}\},
    \end{equation}
    and let
    $$A_{S_{0}}= P^{R_{[1,a],\{t'+1\}}}_{S_{0}'} A_{0} P_{\mathcal{S}_{0}}^{\dag}$$
    where $A_{0}$ is defined in \eqref{eq:def-A_0}.
    By \eqref{eq:def-A_0}, $A_{S_{0}}$ can be regarded as a triangular matrix and by simple calculations, we have
    \begin{equation}\label{eq:transfer-bounding-A-inverse}
            \|A_{S_{0}}^{-1}\|\leq a.
    \end{equation}
    Denote $A'=I^{R_{[1,a],\{t'+1\}}}_{R_{[1,a],\{t'+1\}}\setminus S_{0}'} P^{R_{[1,a],\{t'+1\}}}_{R_{[1,a],\{t'+1\}}\setminus S_{0}'}$ and $\mathcal{S}_{0}^{c}=\{1,\cdots,d_{0}\}\setminus \mathcal{S}_{0}$. Then we can decompose the identity operator on $\ell^{2}(R_{[1,a],\{t'+1\}})$ by $I^{(1)}=A'+I^{R_{[1,a],\{t'+1\}}}_{ S_{0}'} P^{R_{[1,a],\{t'+1\}}}_{ S_{0}'}$, and the identity operator on $\mathbb{R}^{d_{0}}$ by $I^{(2)}= P_{\mathcal{S}_{0}}^{\dag} P_{\mathcal{S}_{0}}+P_{\mathcal{S}_{0}^{c}}^{\dag} P_{\mathcal{S}_{0}^{c}}$.
    
    Suppose $u$ satisfies \eqref{eq:event-transfer}.
    By Claim \ref{cla:expand-u-first}, there exists $u^{*} \in \ell^{2}(R_{[1,a],\{t'+1\}})$ with
    $\|u^{*}\|\leq (a\Bar{V})^{-5}$ and we have
    \begin{equation}\label{eq:decom_u}
        u|_{R_{[1,a],\{b_{0}\}}}=M_{[1,a]}^{t'+1,b_{0}}(u^{*}+ A_{0}(\vec{V}_{t'}))+ A_{1}(u|_{R_{[1,2],[1,t'-1]}})+A_{2}(u|_{R_{[1,2],[t'+1,b]}})+v^{*}
    \end{equation}
    such that $A_{0},A_{1},A_{2},v^{*}$ are all independent of $V|_{R_{[1,a],\{t'\}}}$ and vector $\vec{V}_{t'} \in \R^{d_{0}}$ is $V|_{R_{[1,a],\{t'\}}}$-measurable.
    By the argument above, we can expand the first term in \eqref{eq:decom_u} (or \eqref{eq:tranfer-combine}) as follows:
    \begin{align}\label{eq:expand-first-step}
    \begin{split}
        &M_{[1,a]}^{t'+1,b_{0}}(u^{*}+ A_{0}(\vec{V}_{t'}))\\
        = & M_{[1,a]}^{t'+1,b_{0}} I^{(1)} (u^{*}+ A_{0}(\vec{V}_{t'}))\\
        = & M_{[1,a]}^{t'+1,b_{0}} \left(A'+I^{R_{[1,a],\{t'+1\}}}_{ S_{0}'} P^{R_{[1,a],\{t'+1\}}}_{ S_{0}'}\right) (u^{*}+ A_{0}(\vec{V}_{t'}))\\
        = & M_{[1,a]}^{t'+1,b_{0}} A' (u^{*}+ A_{0}(\vec{V}_{t'})) +  M_{[1,a]}^{t'+1,b_{0}} I^{R_{[1,a],\{t'+1\}}}_{ S_{0}'} P^{R_{[1,a],\{t'+1\}}}_{ S_{0}'} (u^{*})\\
        + &M_{[1,a]}^{t'+1,b_{0}} I^{R_{[1,a],\{t'+1\}}}_{ S_{0}'} P^{R_{[1,a],\{t'+1\}}}_{ S_{0}'}A_{0}(\vec{V}_{t'}),
    \end{split}
    \end{align}
    and the last term in the last equation of \eqref{eq:expand-first-step} can be further expanded:
    \begin{align}\label{eq:expand-second-step}
    \begin{split}
        &M_{[1,a]}^{t'+1,b_{0}} I^{R_{[1,a],\{t'+1\}}}_{ S_{0}'} P^{R_{[1,a],\{t'+1\}}}_{ S_{0}'}A_{0}(\vec{V}_{t'})\\
        =& M_{[1,a]}^{t'+1,b_{0}} I^{R_{[1,a],\{t'+1\}}}_{ S_{0}'} P^{R_{[1,a],\{t'+1\}}}_{ S_{0}'}A_{0} I^{(2)} (\vec{V}_{t'})\\
        =& M_{[1,a]}^{t'+1,b_{0}} I^{R_{[1,a],\{t'+1\}}}_{ S_{0}'} P^{R_{[1,a],\{t'+1\}}}_{ S_{0}'}A_{0} \left( P_{\mathcal{S}_{0}}^{\dag} P_{\mathcal{S}_{0}}+P_{\mathcal{S}_{0}^{c}}^{\dag} P_{\mathcal{S}_{0}^{c}}\right) (\vec{V}_{t'})\\
        =& M_{[1,a]}^{t'+1,b_{0}} I^{R_{[1,a],\{t'+1\}}}_{ S_{0}'} A_{S_{0}} P_{\mathcal{S}_{0}}(\vec{V}_{t'})\\
        +&M_{[1,a]}^{t'+1,b_{0}} I^{R_{[1,a],\{t'+1\}}}_{ S_{0}'} P^{R_{[1,a],\{t'+1\}}}_{ S_{0}'}A_{0} P_{\mathcal{S}_{0}^{c}}^{\dag} P_{\mathcal{S}_{0}^{c}}(\vec{V}_{t'}).
    \end{split}
    \end{align}
    Plug \eqref{eq:expand-first-step} and \eqref{eq:expand-second-step} into \eqref{eq:tranfer-combine},
    after projecting onto $S_{0}$, we have
    \begin{align}
    \begin{split}\label{eq:transfer-combine-2}
        &u|_{S_{0}}\\
        =&M_{S_{0}} A_{S_{0}} (A_{S_{0}}^{-1} P^{R_{[1,a],\{t'+1\}}}_{S_{0}'}(u^{*})
        +P_{\mathcal{S}_{0}} (\vec{V}_{t'}) )\\
        +&P_{S_{0}}^{R_{[1,a],\{b_{0}\}}} M^{t'+1,b_{0}}_{[1,a]} A' (u^{*}+A_{0} (\vec{V}_{t'})) + M_{S_{0}} P^{R_{[1,a],\{t'+1\}}}_{S_{0}'} A_{0} P_{\mathcal{S}_{0}^{c}}^{\dag} P_{\mathcal{S}_{0}^{c}} (\vec{V}_{t'})\\
        + &P_{S_{0}}^{R_{[1,a],\{b_{0}\}}} A_{1}(u|_{R_{[1,2],[1,t'-1]}})+P_{S_{0}}^{R_{[1,a],\{b_{0}\}}} A_{2}(u|_{R_{[1,2],[t'+1,b]}})\\
        + &v^{*}|_{S_{0}}.
    \end{split}
    \end{align}
    Let $\Gamma\subset \ell^{2}(S_{0})$ be the direct sum of the ranges of the following four operators appeared in the third and fourth lines of \eqref{eq:transfer-combine-2}:
    \begin{equation*}
    P_{S_{0}}^{R_{[1,a],\{b_{0}\}}} M^{t'+1,b_{0}}_{[1,a]} A', \;
    M_{S_{0}}P^{R_{[1,a],\{t'+1\}}}_{S_{0}'} A_{0} P_{\mathcal{S}_{0}^{c}}^{\dag} P_{\mathcal{S}_{0}^{c}},\;
    P_{S_{0}}^{R_{[1,a],\{b_{0}\}}}A_{1},\; 
    P_{S_{0}}^{R_{[1,a],\{b_{0}\}}}A_{2}.
    \end{equation*}
    Let us denote $$\Gamma-v^{*}|_{S_{0}}=\{-v^{*}|_{S_{0}}+x:x\in \Gamma\}\subset \ell^{2}(S_{0})$$ 
     and define the event $\mathcal{E}_{dist}$ as
    \begin{equation}\label{eq:event-dist}
        \dist(P_{\mathcal{S}_{0}} (\vec{V}_{t'}) , (M_{S_{0}} A_{S_{0}})^{-1}(\Gamma-v^{*}|_{S_{0}}))\geq \frac{\Bar{V}}{4}a^{-1}
    \end{equation}
    where $\dist$ is the Euclidean distance.
    We claim that $\mathcal{E}_{dist}\subset \mathcal{E}_{S_{0}}^{(s',t')}$ by choosing $\alpha_{1}>15C_{1}$  (recall definition \eqref{eq:def-E_S} of $\mathcal{E}_{S}^{(s',t')}$). To see this, assume $\mathcal{E}_{dist}$ holds. \eqref{eq:transfer-combine-2} implies
    \begin{align}
    \begin{split}\label{eq:transfer-bound-distance}
    &\|u\|_{\ell^{2}(S_{0})}\\
    &\geq \dist \big(M_{S_{0}} A_{S_{0}} (A_{S_{0}}^{-1} P^{R_{[1,a],\{t'+1\}}}_{S_{0}'} (u^{*})+P_{\mathcal{S}_{0}} (\vec{V}_{t'}) ), \Gamma-v^{*}|_{S_{0}}\big)\\
    & \geq \|(M_{S_{0}} A_{S_{0}})^{-1} \|^{-1} \dist\big(A_{S_{0}}^{-1} P^{R_{[1,a],\{t'+1\}}}_{S_{0}'}(u^{*})+P_{\mathcal{S}_{0}} (\vec{V}_{t'}) , (M_{S_{0}} A_{S_{0}})^{-1}(\Gamma-v^{*}|_{S_{0}})\big)\\
    & \geq (a\Bar{V})^{-4C_{1} a} \dist\big(A_{S_{0}}^{-1} P^{R_{[1,a],\{t'+1\}}}_{S_{0}'}(u^{*})+P_{\mathcal{S}_{0}} (\vec{V}_{t'}) , (M_{S_{0}} A_{S_{0}})^{-1}(\Gamma-v^{*}|_{S_{0}})\big).
    \end{split}
    \end{align}
    Here, we used \eqref{eq:transfer-bounding-M-inverse} and $\eqref{eq:transfer-bounding-A-inverse}$. By \eqref{eq:transfer-bound-u*} and \eqref{eq:transfer-bounding-A-inverse}, we have $$\|A_{S_{0}}^{-1} P^{R_{[1,a],\{t'+1\}}}_{S_{0}'}(u^{*})\|\leq \|A_{S_{0}}^{-1}\|\|u^{*}\|\leq  a (a\Bar{V})^{-5}\leq a^{-4}.$$ Thus \eqref{eq:transfer-bound-distance} further implies
    \begin{equation}\label{eq:transfer-bound-u-3}
        \|u\|_{\ell^{2}(S_{0})} \geq (a\Bar{V})^{-4C_{1} a} \big(\dist\big(P_{\mathcal{S}_{0}} (\vec{V}_{t'}) , (M_{S_{0}} A_{S_{0}})^{-1}(\Gamma-v^{*}|_{S_{0}})\big)-a^{-4}\big).
    \end{equation}
    By letting $\alpha_{1}>15 C_{1}$, $\mathcal{E}_{dist}$ (or \eqref{eq:event-dist}) implies
    $$\|u\|_{\ell^{2}(S_{0})} \geq (a\Bar{V})^{-4C_{1} a} \left(\frac{\Bar{V}}{4}a^{-1}-a^{-4}\right)\geq (a\Bar{V})^{-5C_{1}a} \geq (a\Bar{V})^{-\frac{1}{3}\alpha_{1} a}.
    $$
    This proves our claim that $\mathcal{E}_{dist}\subset \mathcal{E}_{S_{0}}^{(s',t')}$. Thus in order to prove \eqref{eq:transfer-conditioning} with $S=S_{0}$, it suffices to prove
    \begin{equation}\label{eq:transfer-conditioning-2}
        \mathbb{P}\left[\mathcal{E}_{dist}\big|\; V|_{\Theta\cup (R_{1}\setminus R_{[1,a],\{t'\}})}\right]\geq 1-\exp(-a/50).
    \end{equation}
    To see this, we first prove an upper bound for the dimension of $\Gamma$.
    The ranks of operators $P_{S_{0}}^{R_{[1,a],\{b_{0}\}}}A_{1}$ and $P_{S_{0}}^{R_{[1,a],\{b_{0}\}}}A_{2}$ are less than $b$ since the dimensions of their domains are less than $b$. On the other hand, the ranks of operators $M_{S_{0}} P^{R_{[1,a],\{t'+1\}}}_{S_{0}'} A_{0} P_{\mathcal{S}_{0}^{c}}^{\dag} P_{\mathcal{S}_{0}^{c}}$ and $P_{S_{0}}^{R_{[1,a],\{b_{0}\}}} M^{t'+1,b_{0}}_{[1,a]} A'$ are at most $a/50$ since $|\mathcal{S}_{0}^{c}|,|R_{[1,a],\{t'+1\}}\setminus S_{0}'|\leq a/50$. Since $b\leq a/10$, the dimension of $\Gamma$ is at most $b+b+a/50+a/50\leq \frac{2}{5}a$. 
    
    Together with Claim \ref{cla:expand-u-first}, these imply that $(M_{S_{0}} A_{S_{0}})^{-1}(\Gamma-v^{*}|_{S_{0}})\subset \R^{\mathcal{S}_{0}}$ is an affine subspace with dimension at most $\frac{2}{5}a$ and is independent of $V|_{R_{[1,a],\{t'\}}}$. On the other hand, since $S_{0}\subset R_{[4,a-1],\{b_{0}\}}\setminus \Theta_{*}$, by definition \eqref{eq:define-theta} of $\Theta_{*}$ and equations \eqref{eq:transfer-potential-vector-on-t'} and \eqref{eq:transfer-def-S_0}, $P_{\mathcal{S}_{0}} (\vec{V}_{t'})$ is independent of $V|_{\Theta}$.
    Moreover, by \eqref{eq:transfer-bound-S}, we have $$a\geq |\mathcal{S}_{0}|\geq d_{0}-a/50\geq \frac{2}{5}a+\frac{1}{20}a.$$ 
    Thus by Lemma \ref{lem:Boolean-cube}, conditioning on $V|_{\Theta\cup (R_{1}\setminus R_{[1,a],\{t'\}})}$, with probability no less than $1-2^{-\frac{1}{20}a +1} \geq 1-\exp(-a/50)$, we have
    $$ \dist(P_{\mathcal{S}_{0}} (\vec{V}_{t'}) , (M_{S_{0}} A_{S_{0}})^{-1}(\Gamma-v^{*}|_{S_{0}}))\geq \frac{\Bar{V}}{4}a^{-1}
    $$
    which is \eqref{eq:event-dist}.
    Hence \eqref{eq:transfer-conditioning-2} holds and Claim \ref{cla:bound-one-S} follows.
\end{proof}
    Now, by Claim \ref{cla:all-S} and Claim \ref{cla:bound-one-S}, and letting $c_{4}$ be small enough,
    \begin{align*}
        &\mathbb{P}\left[\left(\mathcal{E}_{tr}^{(s',t')}\right)^{c}\big|\; V|_{\Theta}\right]\\
        &\leq \sum_{\substack{S\subset R_{[4,a-1],\{b_{0}\}}\setminus \Theta_{*}\\ |R_{[1,a],\{b_{0}\}}\setminus S|= \left\lfloor a/10^{5} \right\rfloor}} \mathbb{P}\left[\left(\mathcal{E}_{S}^{(s',t')}\right)^{c}\big|\; V|_{\Theta}\right]\\
        &\leq \sum_{\substack{S\subset R_{[4,a-1],\{b_{0}\}}\setminus \Theta_{*}\\ |R_{[1,a],\{b_{0}\}}\setminus S|= \left\lfloor a/10^{5} \right\rfloor}}
        \exp(-a/50)\\
        &\leq \binom{a}{\left\lfloor a/10^{5} \right\rfloor} \exp(-a/50)\\
        &\leq \exp(-2c_{4} a)
    \end{align*}
    for any large enough $a$.
    
    Finally, by Claim \ref{cla:all-s'-t'} and Claim \ref{cla:special-t'},
     \begin{align*}
       &\mathbb{P}\left[\mathcal{E}_{tr}(R_{[1,a],[1,b]})\big|\; V|_{\Theta}=V'\right]\\
       &\geq 1-\sum_{(s',t')\in R_{[1,2],[3,b]}} \mathbb{P}\left[\left(\mathcal{E}_{tr}^{(s',t')}\right)^{c}\big|\; V|_{\Theta}=V'\right]\\
       &\geq 1-b\exp(-2c_{4} a)\\
       &\geq 1-\exp(-c_{4}a).
    \end{align*}
    Our conclusion follows.
\end{proof}
\begin{lemma}\label{lem:key-lemma}
There are constants $\alpha_{2}>1>c_{5}>0$ such that, if
\begin{enumerate}
    \item integers $a>b>\alpha_{2}$ with $10b \leq a \leq 60b$,
    \item $\lambda_{0}\in [0,8]$, $\Bar{V}\geq 2$,
    \item $\Theta\subset \Z^{2}$ is $(c_{5},-)$-sparse in $R_{[1,a],[1,b]}$,
    \item $V':\Theta \rightarrow \{0,\Bar{V}\}$,
    \item $\mathcal{E}_{ni}(R_{[1,a],[1,b]})$ denotes the event that,
    \begin{equation}\label{eq:event-ni}
    \left\{\begin{array}{l}
    |\lambda-\lambda_{0}|\leq (a\Bar{V})^{-\alpha_{2} a}\\
    \text{$H u= \lambda u$ in $R_{[2,a-1],[2,b-1]}$ }\\
    \text{$|u|\leq 1$ on $R_{[1,a],[1,2]}$} \\
    \text{$|u|\leq 1$ on a $1-10^{-7}$ fraction of $R_{[1,a],[b-1,b]}$} 
    \end{array}\right. 
    \end{equation}
    implies $|u|\leq (a\Bar{V})^{\alpha_{2} a}$ in $R_{[1,a],[1,b]}$,
\end{enumerate}
then $\mathbb{P}\left[\mathcal{E}_{ni}(R_{[1,a],[1,b]})\big|\; V|_{\Theta}=V'\right]\geq 1-\exp(-c_{5} a)$. 

\end{lemma}
\begin{proof}
    Denote $R_{1}=R_{[1,a],[1,b]}$. Set $c_{5}=c_{4}$ where $c_{4}$ is the constant in Lemma \ref{lem:transfer-u_3} and $\alpha_{2}$ to be determined.
    We prove that $\mathcal{E}_{tr}(R_{1})\subset \mathcal{E}_{ni}(R_{1})$ where $\mathcal{E}_{tr}(R_{1})$ is defined in Lemma \ref{lem:transfer-u_3}. Suppose event $\mathcal{E}_{tr}(R_{1})$ holds and $u$ satisfies \eqref{eq:event-ni}. By Lemma \ref{lem:extension}, there is $u_{1}:R_{1}\rightarrow \R$ such that
    \begin{equation}
    \left\{\begin{array}{l}
    \text{$H u_{1}= \lambda u_{1}$ in $R_{[2,a-1],[2,b-1]}$ }\\
    \text{$u_{1}=u$ on $R_{[1,a],[1,2]}$} \\
    \text{$u_{1}=0$ on $R_{[1,2],[3,b]}$}.
    \end{array}\right. 
    \end{equation}
    By Lemma \ref{lem:extension-bound}, we have $\|u_{1}\|_{\ell^{\infty}(R_{1})}\leq (a\Bar{V})^{C_{1}b}$ since $\|u_{1}\|_{\ell^{\infty}(\partial^{w}R_{1})}\leq 1$. Let $u_{2}=u-u_{1}$, then 
    \begin{equation}\label{eq:portion-u_2}
        |u_{2}|\leq 1+(a\Bar{V})^{C_{1}b} 
    \end{equation}
    on a $1-10^{-7}$ fraction of $R_{[1,a],[b-1,b]}$.
    Define $u_{3}:R_{1}\rightarrow \R$ as follows:
     \begin{equation}
    \left\{\begin{array}{l}
    \text{$H u_{3}= \lambda_{0} u_{3}$ in $R_{[2,a-1],[2,b-1]}$ }\\
    \text{$u_{3}=0$ on $R_{[1,a],[1,2]}$} \\
    \text{$u_{3}=u_{2}$ on $R_{[1,2],[3,b]}$}.
    \end{array}\right. 
    \end{equation}
    By Lemma \ref{lem:variate-lambda}, 
    \begin{align}
    \begin{split}\label{eq:dist-u_2-u_3}
        &\|u_{3}-u_{2}\|_{\ell^{\infty}(R_{1})}\\
        &\leq (a\Bar{V})^{C_{2}b} \|u_{3}\|_{\ell^{\infty}(\partial^{w}R_{1})}|\lambda-\lambda_{0}|\\
        &\leq (a\Bar{V})^{C_{2}b-\alpha_{2}a} \|u_{3}\|_{\ell^{\infty}(\partial^{w}R_{1})}\\
        &\leq (a\Bar{V})^{-2\alpha_{1}a} \|u_{3}\|_{\ell^{\infty}(R_{[1,2],[3,b]})},
    \end{split}
    \end{align}
    as long as $\alpha_{2}>2\alpha_{1}+C_{2}$.
    By the definition of $\mathcal{E}_{tr}(R_{1})$, $$|u_{3}|\geq (a\Bar{V})^{-\alpha_{1}a} \|u_{3}\|_{\ell^{\infty}(R_{[1,2],[3,b]})}$$ on a $10^{-6}$ fraction of $R_{[1,a],[b-1,b]}$. Thus by \eqref{eq:dist-u_2-u_3}, $$|u_{2}|\geq \big( (a\Bar{V})^{-\alpha_{1}a}-(a\Bar{V})^{-2\alpha_{1}a}\big) \|u_{3}\|_{\ell^{\infty}(R_{[1,2],[3,b]})}\geq (a\Bar{V})^{-2\alpha_{1}a} \|u_{3}\|_{\ell^{\infty}(R_{[1,2],[3,b]})}$$ on a $10^{-6}$ fraction of $R_{[1,a],[b-1,b]}$. By \eqref{eq:portion-u_2}, $$(a\Bar{V})^{-2\alpha_{1}a} \|u_{3}\|_{\ell^{\infty}(R_{[1,2],[3,b]})} \leq 1+(a\Bar{V})^{C_{1}b}$$ and thus $$\|u_{2}\|_{\ell^{\infty}(R_{[1,2],[3,b]})}=\|u_{3}\|_{\ell^{\infty}(R_{[1,2],[3,b]})} \leq (a\Bar{V})^{2\alpha_{1}a}+ (a\Bar{V})^{C_{1}b+2\alpha_{1}a}.$$
    Since $u_{2}=0$ on $R_{[1,a],[1,2]}$, by Lemma \ref{lem:extension-bound}, we have $\|u_{2}\|_{\ell^{\infty}(R_{1})}\leq 2 (a\Bar{V})^{2C_{1}b+2\alpha_{1}a}$.
    
    Finally, 
    \begin{align*}
        \|u\|_{\ell^{\infty}(R_{1})}
        &\leq \|u_{1}\|_{\ell^{\infty}(R_{1})} +\|u_{2}\|_{\ell^{\infty}(R_{1})}\\
        &\leq (a\Bar{V})^{C_{1}b} + 2 (a\Bar{V})^{2C_{1}b+2\alpha_{1}a}\\
        &\leq (a\Bar{V})^{\alpha_{2}a}
    \end{align*}
    as long as $\alpha_{2}>2\alpha_{1}+3C_{1}$. Thus $\mathcal{E}_{tr}(R_{1})\subset \mathcal{E}_{ni}(R_{1})$ and our conclusion follows from Lemma \ref{lem:transfer-u_3}.
\end{proof}
\subsection{Growth lemma}
\begin{defn}
Given a tilted square $R_{[a_{1},a_{2}],[b_{1},b_{2}]}$ and integer $k\in \Z^{+}$, we define
$k R_{[a_{1},a_{2}],[b_{1},b_{2}]}$ to be $R_{[a_{3},a_{4}],[b_{3},b_{4}]}$
where $a_{3}=\left\lceil \frac{(k+1)a_{1}-(k-1)a_{2}}{2} \right\rceil$, $a_{4}=\left\lfloor \frac{(k+1)a_{2}-(k-1)a_{1}}{2} \right\rfloor$, $b_{3}=\left\lceil \frac{(k+1)b_{1}-(k-1)b_{2}}{2} \right\rceil$ and $b_{4}=\left\lfloor \frac{(k+1)b_{2}-(k-1)b_{1}}{2} \right\rfloor$.
\end{defn}
For a tilted square $\Tilde{Q}$, the following lemma allows us to estimate $\|u\|_{\ell^{\infty}(2 \Tilde{Q})}$ from an upper bound of $\|u\|_{\ell^{\infty}(\Tilde{Q})}$, provided the portion of points with $|u|>1$ is small enough in $4 \Tilde{Q}$. The proof is similar to that of \cite[Lemma 3.18]{ding2020localization} and \cite[Lemma 3.6]{buhovsky2017discrete}.
\begin{lemma}\label{lem:growthlemma}
For every small $\varepsilon>0$, there is a large $\alpha>1$ such that, if
\begin{enumerate}
    \item $\Tilde{Q}$ tilted square with $\ell(\Tilde{Q})>\alpha$,
    \item $\Theta\subset \Z^{2}$ is $\varepsilon$-sparse in $2\Tilde{Q}$,
    \item $\lambda_{0}\in [0,8]$ and $\Bar{V}\geq 2$,
    \item $V':\Theta \rightarrow \{0,\Bar{V}\}$,
    \item $\mathcal{E}^{\varepsilon,\alpha}_{ex}(\Tilde{Q},\Theta)$ denotes the event that,\\
    \begin{equation}\label{eq:cond-covering-argu}
        \left\{\begin{array}{l}
    |\lambda-\lambda_{0}|\leq (\ell(\Tilde{Q})\Bar{V})^{-\alpha \ell(\Tilde{Q})}\\
    \text{$H u= \lambda u$ in $2\Tilde{Q}$ }\\
    \text{$|u|\leq 1$ in $\frac{1}{2} \Tilde{Q}$} \\
    \text{$|u|\leq 1$ in a $1-\varepsilon$ fraction of $2\Tilde{Q}\setminus \Theta$} 
    \end{array}\right. 
    \end{equation}
    implies $|u|\leq (\ell(\Tilde{Q})\Bar{V})^{\alpha \ell(\Tilde{Q})}$ in $\Tilde{Q}$,
\end{enumerate}
then $\mathbb{P}\left[\mathcal{E}^{\varepsilon,\alpha}_{ex}(\Tilde{Q},\Theta)\big|\; V|_{\Theta}=V'\right]\geq 1-\exp(-\varepsilon \ell(\Tilde{Q}))$.
\end{lemma}

\begin{figure}
        \centering
        \includegraphics{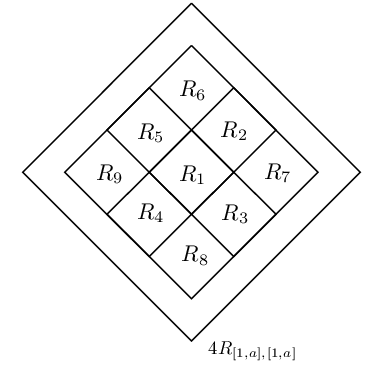}
        \caption{An illustration of covering argument. Here, we have $R_{1} = R_{[1,a],[1,a]}$, $R_{2}=R_{[1,a],[a+1,2a]}$, $R_{3}=R_{[1-a,0],[1,a]}$, $R_{4}=R_{[1,a],[1-a,0]}$, $R_{5}=R_{[a+1,2a],[1,a]}$, $R_{6}=R_{[a+1,2a],[a+1,2a]}$, $R_{7}=R_{[1-a,0],[a+1,2a]}$, $R_{8}=R_{[1-a,0],[1-a,0]}$ and $R_{9}=R_{[a+1,2a],[1-a,0]}$.}
        \label{fig:rotation}
    \end{figure}
\begin{proof}
    We identify $\Tilde{Q}$ with $2R_{[1,a],[1,a]}$. For each $1 \leq i \leq 9$, let $R_{i}$ be the square illustrated in Figure \ref{fig:rotation} and its caption.
    For some large $a$ and $1\leq i,j\leq 9$, let $\mathcal{E}'_{ex}(i,j)$ denote the event that
    \begin{equation}\label{eq:extend-square}
        \left\{\begin{array}{l}
        |\lambda-\lambda_{0}|\leq (a\Bar{V})^{-\alpha a}\\
             \text{$H u=\lambda u$ in $R_{i}\cup R_{j}$}\\
             \text{$|u|\leq 1$ in $R_{i}$} \\ 
             \text{$|\{(s,t) \in R_{j}:|u(s,t)|> 1\}| \leq 100\varepsilon a^{2}$} 
        \end{array}\right.
    \end{equation}
    implies $|u|\leq (a\Bar{V})^{\frac{1}{2} \alpha a}$ in $R_{j}$.
    \begin{cla}\label{cla:rotation}
        Let $S = \{(1,2),(1,3),(1,4),(1,5),(2,6),(3,7),(4,8),(5,9)\}$. Then
        $\bigcap_{(i,j) \in S} \mathcal{E}'_{ex}(i,j) \subset \mathcal{E}^{\varepsilon,\alpha}_{ex}(2R_{1},\Theta)$.
    \end{cla}
    \begin{proof}[Proof of the claim]
        The strategy here is to use a covering argument from elementary geometry. The covering argument was used in \cite[Lemma 3.18]{ding2020localization} without giving the details.
        Assume event $\bigcap_{(i,j) \in S} \mathcal{E}'_{ex}(i,j)$ holds and $u$ satisfies \eqref{eq:extend-square}. Our goal is to prove $|u|\leq (\ell(2 R_{1})\Bar{V})^{\alpha \ell(2 R_{1})}$ in $2 R_{1}$.
        
        Since $\Theta$ is $\varepsilon$-spares in $4R_{[1,a],[1,a]}$ and $|u|\leq 1$ in a $1-\varepsilon$ fraction of $4R_{[1,a],[1,a]}\setminus \Theta$, we have $|\{(s,t) \in 4R_{[1,a],[1,a]}:|u(s,t)|> 1\}| \leq 100\varepsilon a^{2}$.
        Then the event $\mathcal{E}'_{ex}(1,2)\cap\mathcal{E}'_{ex}(1,3)\cap\mathcal{E}'_{ex}(1,4)\cap \mathcal{E}'_{ex}(1,5)$ implies $|u|\leq (a\Bar{V})^{\frac{1}{2}\alpha a}$ in $\bigcup_{1\leq j \leq 5} R_{j}$. Finally, the event $\mathcal{E}'_{ex}(2,6)\cap\mathcal{E}'_{ex}(3,7)\cap\mathcal{E}'_{ex}(4,8)\cap \mathcal{E}'_{ex}(5,9)$ implies $|u|\leq (a\Bar{V})^{\alpha a}$ in $\bigcup_{1\leq j \leq 9} R_{j}$. Since $2 R_{1} \subset \bigcup_{1\leq j \leq 9} R_{j}$, the claim follows.
    \end{proof}
    Denote $\mathcal{E}'_{ex}(1,2)$ by $\mathcal{E}'_{ex}$ and let $S$ be the set in the Claim \ref{cla:rotation}. By Claim \ref{cla:rotation}, it is sufficient to prove that $\Prob\left [\mathcal{E}'_{ex}(i,j)\big|\; V|_{\Theta}=V'\right ] \geq 1-\exp(-\varepsilon a)$ for each $(i,j) \in S$. By symmetry,
    we only need to prove for the case where $(i,j) = (1,2)$, i.e. $\Prob\left [\mathcal{E}'_{ex}(1,2)\big|\; V|_{\Theta}=V'\right ] \geq 1-\exp(-\varepsilon a)$. 
    
    By Lemma \ref{lem:key-lemma} and a union bound, the event 
    $$ \mathcal{E}_{ni}= \bigcap_{\substack{[c,d]\subset [1,\frac{5}{2}a]\\ \frac{a}{60}\leq d-c\leq \frac{a}{10}}} \mathcal{E}_{ni}(R_{[1,a],[c,d]})
    $$ 
    satisfies $\Prob[\mathcal{E}_{ni}\big| V|_{\Theta}=V']\geq 1-\exp(-c_{5} a+ C \log(a))$ where $c_{5}$ is the constant in Lemma \ref{lem:key-lemma}. It suffices to prove that, for every small $\varepsilon<\frac{1}{4}c_{5}$, there is a large $\alpha$ such that $\mathcal{E}_{ni}\subset \mathcal{E}'_{ex}(1,2)$. Assume $\mathcal{E}_{ni}$ and \eqref{eq:extend-square} hold, our goal is to prove $|u|\leq (a\Bar{V})^{\frac{1}{2}\alpha a}$ in $R_{2} = R_{[1,a],[a+1,2a]}$.

    \begin{cla}\label{cla:extension_cla}
    Suppose $\varepsilon<10^{-12}$. Then there is a sequence $b_{0}\leq \cdots \leq b_{25}\in [a,\frac{5}{2}a-2]$ such that
    \begin{enumerate}
        \item $b_{0}=a$
        \item $b_{25}\geq 2a$
        \item $\frac{1}{60}a \leq b_{k+1}-b_{k}\leq \frac{3}{40}a$ for $0\leq k < 25$
        \item $|u|\leq 1$ on a $1-10^{-7}$ fraction of $R_{[1,a],[b_{k+1}-1,b_{k+1}]}$ for $0\leq k< 25$
    \end{enumerate}
    \end{cla}
    \begin{proof}[Proof of the claim]
        Let $b_{0}=a$. For each $k\in \{1,\cdots,25\}$, let interval $$J_{k}=\left(a+\frac{2k}{40}a,a+\frac{2k+1}{40}a\right].$$ Since $|u|> 1$ on at most $100\varepsilon a^{2}$ points in $R_{[1,a],[1,2a]}$, we have
        $$ \#\{(s,t)\in R_{[1,a],J_{k}}: |u(s,t)|>1\}< 10^{4}\varepsilon \left|R_{[1,a],J_{k}}\right|
        $$
        for each $k=1,\cdots,25$.
        The pigeonhole principle implies that, there is $b_{k}\in J_{k}\cap \Z$ such that 
        $$ \#\{(s,t)\in R_{[1,a],[b_{k}-1,b_{k}]}: |u(s,t)|>1\}< 10^{5}\varepsilon \left|R_{[1,a],[b_{k}-1,b_{k}]}\right|. 
        $$
        Since $\varepsilon<10^{-12}$, we have $$ \#\{(s,t)\in R_{[1,a],[b_{k}-1,b_{k}]}: |u(s,t)|>1\}< 10^{-7} \left|R_{[1,a],[b_{k}-1,b_{k}]}\right|
        $$
        for each $k=1,\cdots,25$.
        On the other hand, $b_{k+1}-b_{k}\in \left[\frac{1}{40}a, \frac{3}{40}a\right] \subset \left[\frac{1}{60}a,\frac{3}{40}a\right]$ for $0\leq k <25$. Finally, $b_{25}> a+\frac{5}{4}a >2a$ and our claim follows.
    \end{proof}
    With the claim in hand, we apply $\mathcal{E}_{ni}(R_{[1,a],[b_{k}-1,b_{k+1}]})$ to conclude
    $$\|u\|_{\ell^{\infty}(R_{[1,a],[b_{k}-1,b_{k+1}]})}\leq (a\Bar{V})^{\alpha_{2} a}(1+ \|u\|_{\ell^{\infty}(R_{[1,a],[b_{k}-1,b_{k}]})})
    $$
    for $k=0,\cdots,24$. Since $\|u\|_{\ell^{\infty}(R_{[1,a],[1,a]})}\leq 1$, by induction, we obtain $$\|u\|_{\ell^{\infty}(R_{[1,a],[1,2a]})}\leq 2^{25} (a\Bar{V})^{25 \alpha_{2} a} <(a\Bar{V})^{\frac{1}{2}\alpha a}$$ by setting $\alpha > 100\alpha_{2}$.
\end{proof}
\subsection{Covering argument}\label{sec:covering}
The proof of Lemma \ref{lem:unique-continuation} below is a random version of \cite[Proposition 3.9]{buhovsky2017discrete}. 
\begin{defn}
Given a tilted square $R_{[a_{1},a_{2}],[b_{1},b_{2}]}$ with $a_{2}-a_{1}=b_{2}-b_{1}>0$, we call the point $$\left(\left\lfloor\frac{a_{1}+a_{2}}{2} \right\rfloor, \left\lfloor\frac{b_{1}+b_{2}}{2} \right\rfloor +i\right)\in \widetilde{\Z^{2}}$$ the \emph{center} of $R_{[a_{1},a_{2}],[b_{1},b_{2}]}$. Here, $i\in \{0,1\}$ such that $\left\lfloor\frac{a_{1}+a_{2}}{2} \right\rfloor- \left\lfloor\frac{b_{1}+b_{2}}{2} \right\rfloor-i$ is an even number.
\end{defn}
\begin{proof}[Proof of Lemma \ref{lem:unique-continuation}]
    Let $\alpha'>1>\varepsilon'>0$ be a pair of valid constants in Lemma \ref{lem:growthlemma}. Let 
    \begin{equation}\label{eq:bound-epsilon_1}
        \varepsilon_{1}<10^{-30}\varepsilon'
    \end{equation}
    and suppose $\varepsilon<\varepsilon_{1}$. We impose further constraints on $\varepsilon_{1},\alpha$ during the proof.
    
    Assume without loss of generality that $Q=Q_{n}(\mathbf{0})$.
    Given integers $|s_{1}|,|t_{1}|\leq 10^{-10} \ell(Q)^{\frac{1}{3}}$ and $|s_{2}|,|t_{2}|\leq \varepsilon \ell(Q)^{\frac{2}{3}}$, let $Q_{s_{1},t_{1},s_{2},t_{2}}$ be the tilted square with center 
    $$ \left(100 \left\lceil \ell(Q)^{\frac{2}{3}} \right\rceil s_{1} +2 \left\lceil \varepsilon^{-1} \right\rceil s_{2},100 \left\lceil \ell(Q)^{\frac{2}{3}} \right\rceil t_{1}+2 \left\lceil \varepsilon^{-1} \right\rceil t_{2}\right)
    $$ and length being any integer satisfying
    \begin{equation}\label{eq:length-of-s_1-t_1-s_2-t_2}
            (4\varepsilon)^{-1}\leq \ell(Q_{s_{1},t_{1},s_{2},t_{2}}) \leq (2\varepsilon)^{-1}.
    \end{equation}
    Then for different pairs $(s_{1},t_{1},s_{2},t_{2})$ and $(s'_{1},t'_{1},s'_{2},t'_{2})$, 
    \begin{equation}\label{eq:disjoint-s_1-t_1-s_2-t_2}
             Q_{s_{1},t_{1},s_{2},t_{2}}\cap Q_{s'_{1},t'_{1},s'_{2},t'_{2}}=\emptyset.
    \end{equation}
    Meanwhile, for any $s_{2},t_{2}\in \left[-\varepsilon \ell(Q)^{\frac{2}{3}},\varepsilon \ell(Q)^{\frac{2}{3}}\right]$,
    \begin{equation}\label{eq:dist-s_2-t_2}
        \dist(Q_{s_{1},t_{1},s_{2},t_{2}},Q_{s'_{1},t'_{1},s_{2},t_{2}})>50\ell(Q)^{\frac{2}{3}} 
    \end{equation}
    when $(s_{1},t_{1})\not=(s'_{1},t'_{1})$. Let 
    \begin{align*}
        \mathcal{E}^{s_{1},t_{1},s_{2},t_{2}}_{ex}
        =\bigcap\{\mathcal{E}^{\alpha',\varepsilon'}_{ex}(Q',\Theta): Q'\supseteq Q_{s_{1},t_{1},s_{2},t_{2}},\ell(Q')\leq \ell(Q)^{\frac{2}{3}}\}.
    \end{align*}
    By Lemma \ref{lem:growthlemma} and \eqref{eq:length-of-s_1-t_1-s_2-t_2}, for each $s_{1},t_{1},s_{2},t_{2}$,
    \begin{equation}\label{eq:large-deviation-cond}
            \Prob\left[\mathcal{E}^{s_{1},t_{1},s_{2},t_{2}}_{ex}\big| \; V|_{\Theta}=V'\right]\geq 1-\sum_{l\geq (4\varepsilon)^{-1}}10l^{2} \exp(-\varepsilon' l)>\frac{9}{10}
    \end{equation}
    by choosing $\varepsilon<\varepsilon_{1}$ small enough. Here, we used the fact that for any integer $l$, the number of tilted squares with length $l$ that contain $Q_{s_{1},t_{1},s_{2},t_{2}}$ is less than $10 l^{2}$.
    Note that, for each tilted $Q'$, $\mathcal{E}^{\alpha',\varepsilon'}_{ex}(Q',\Theta)$ is $V|_{2 Q'}$-measurable. Thus 
    for any $s'_{2},t'_{2}\in \left[-\varepsilon \ell(Q)^{\frac{2}{3}},\varepsilon \ell(Q)^{\frac{2}{3}}\right]$, by \eqref{eq:dist-s_2-t_2}, we have $$\left\{\mathcal{E}^{s_{1},t_{1},s'_{2},t'_{2}}_{ex}:|s_{1}|,|t_{1}|\leq 10^{-10} \ell(Q)^{\frac{1}{3}}\right\}$$ is a family of independent events. We denote by $\mathcal{E}^{s'_{2},t'_{2}}_{ex}$ the following event 
    \begin{equation}\label{eq:half-of-events-hold}
    \textit{{at least half of events in $\left\{\mathcal{E}^{s_{1},t_{1},s'_{2},t'_{2}}_{ex}:|s_{1}|,|t_{1}|\leq 10^{-10} \ell(Q)^{\frac{1}{3}}\right\}$ happen.}}
    \end{equation}
    Then by \eqref{eq:large-deviation-cond} and a large deviation estimate, 
    \begin{equation}\label{eq:large-deviation-es}
            \Prob\left[\mathcal{E}^{s'_{2},t'_{2}}_{ex}\big| \; V|_{\Theta}=V'\right]\geq 1-\exp(-c \ell(Q)^{\frac{2}{3}})
    \end{equation}
    for a numerical constant $c$.
    Let 
    \begin{align*}
             &\mathcal{E}_{ex}\\
             =&\bigcap\{\mathcal{E}^{s_{2},t_{2}}_{ex}:|s_{2}|,|t_{2}|\leq \varepsilon\ell(Q)^{\frac{2}{3}}\}\cap \bigcap\{\mathcal{E}^{\alpha',\varepsilon'}_{ex}(Q',\Theta):\ell(Q')\geq \ell(Q)^{\frac{2}{3}},Q'\subset Q\}.
    \end{align*}
    Then by Lemma \ref{lem:growthlemma} and \eqref{eq:large-deviation-es}, $$\Prob[\mathcal{E}_{ex}\big| \; V|_{\Theta}=V']\geq 1- \exp(-c' \ell(Q)^{\frac{2}{3}} +C \log\ell(Q))\geq 1-\exp(-c'' \ell(Q)^{\frac{2}{3}})$$ for constants $c',c''$ depending on $\varepsilon'$. Hence, it is sufficient to prove that $$\mathcal{E}_{ex}\subset \mathcal{E}_{uc}^{\varepsilon,\alpha}(Q,\Theta).$$Thus we assume $\mathcal{E}_{ex}$ holds and $u$ satisfies \eqref{eq:unique-conti}. Our goal is to prove
    \begin{equation}\label{eq:final-goal}
        \|u\|_{\ell^{\infty}(\frac{1}{100}Q)}\leq (\ell(Q)\Bar{V})^{\alpha\ell(Q)}.
    \end{equation}
    Let $\mathcal{Q}$ denote the subset of all $Q_{s_{1},t_{1},s_{2},t_{2}}$'s such that $\mathcal{E}^{s_{1},t_{1},s_{2},t_{2}}_{ex}$ happens.
    Then by definition of $\mathcal{E}_{ex}$ and \eqref{eq:half-of-events-hold}, we have 
    \begin{equation}\label{eq:card-of-mathcalQ}
        |\mathcal{Q}|\geq 10^{-21} \varepsilon^{2} \ell(Q)^{2}.
    \end{equation}
    \begin{cla}\label{cla:containing-mathcalQ}
    For any $Q_{s_{1},t_{1},s_{2},t_{2}}\in \mathcal{Q}$ and $Q''\subset Q$ with $Q''\supseteq Q_{s_{1},t_{1},s_{2},t_{2}}$, $\mathcal{E}^{\alpha',\varepsilon'}_{ex}(Q'',\Theta)$ holds.
    \end{cla}
    \begin{proof}
        If $\ell(Q'')\leq \ell(Q)^{\frac{2}{3}}$, then $\mathcal{E}^{s_{1},t_{1},s_{2},t_{2}}_{ex}\subset \mathcal{E}^{\alpha',\varepsilon'}_{ex}(Q'',\Theta)$. Otherwise, $$ \bigcap\{\mathcal{E}^{\alpha',\varepsilon'}_{ex}(Q',\Theta):\ell(Q')\geq \ell(Q)^{\frac{2}{3}},Q'\subset Q\}\subset \mathcal{E}^{\alpha',\varepsilon'}_{ex}(Q'',\Theta).$$
        The claim follows from the definition of $\mathcal{E}_{ex}$.
    \end{proof}
    Let 
    $$\mathcal{Q}_{sp}=\{Q'\in \mathcal{Q}: \exists Q''\subset Q\text{ such that $Q''\supseteq Q'$ and $\Theta$ is not $\varepsilon$-sparse in $Q''$}\}.
    $$
    Write $\mathcal{Q}_{sp}=\{Q^{(i)}_{sp}:1\leq i\leq K_{1}\}$. For each $1\leq i\leq K_{1}$, choose $Q^{(i)}_{spm}\subset Q$ to be a tilted square in which $\Theta$ is not $\varepsilon$-sparse and $Q^{(i)}_{sp}\subset Q^{(i)}_{spm}$. By Vitalli covering theorem, there exists $J'\subset \{1,\cdots,K_{1}\}$ such that 
    $$Q^{(i_{1})}_{spm} \cap Q^{(i_{2})}_{spm} =\emptyset
    $$
    for each $i_{1}\not=i_{2}\in J'$ and 
    \begin{align}
    \begin{split}\label{eq:lower-bound-spm}
         &|\bigcup \{Q^{(i)}_{spm}:i\in J'\}|\\
         \geq &\frac{1}{100}|\bigcup \{Q^{(i)}_{spm}:1\leq i\leq K_{1}\}|\\
         \geq &\frac{1}{100}|\bigcup \{Q^{(i)}_{sp}:1\leq i\leq K_{1}\}|.
    \end{split}
    \end{align}
    Since $\Theta$ is $\varepsilon$-regular in $Q$, we have
    $$
    |\bigcup \{Q^{(i)}_{spm}:i\in J'\}|\leq \varepsilon \ell(Q)^{2}.
    $$
    Thus by \eqref{eq:lower-bound-spm}, $|\bigcup\{Q^{(i)}_{sp}:1\leq i\leq K_{1}\}|\leq 100\varepsilon \ell(Q)^{2}$. Note that by \eqref{eq:disjoint-s_1-t_1-s_2-t_2}, $\{Q^{(i)}_{sp}:1\leq i\leq K_{1}\}$ are pairwise disjoint. Thus by \eqref{eq:length-of-s_1-t_1-s_2-t_2},
    \begin{equation}\label{eq:bound-K_1}
        K_{1}\leq 10^{4}\varepsilon^{3}\ell(Q)^{2}.
    \end{equation}
    By choosing $\varepsilon<10^{-26}$, \eqref{eq:card-of-mathcalQ} and \eqref{eq:bound-K_1} imply
    \begin{equation}\label{eq:lower-bound-number-of-sp}
        |\mathcal{Q}\setminus\mathcal{Q}_{sp}|>10^{-22}\varepsilon^{2}\ell(Q)^{2}.
    \end{equation}
    Now for any $Q'\in \mathcal{Q}\setminus\mathcal{Q}_{sp}$ and any $Q''\subset Q$ with $Q'' \supseteq Q'$, $\Theta$ is $\varepsilon$-sparse in $Q''$. In particular, $\Theta$ is $\varepsilon$-sparse in $Q'$ and by \eqref{eq:length-of-s_1-t_1-s_2-t_2} and Definition \ref{def:sparsity}, $\Theta\cap Q'=\emptyset$. Thus by \eqref{eq:unique-conti},
    \begin{equation}\label{eq:number-larger-1}
        \left|\{ |u|>1\}\cap \bigcup\left\{Q':Q'\in \mathcal{Q}\setminus\mathcal{Q}_{sp}\right\}\right|<\varepsilon^{3}\ell(Q)^{2}.
    \end{equation}
    Equations \eqref{eq:number-larger-1}, \eqref{eq:disjoint-s_1-t_1-s_2-t_2} and \eqref{eq:lower-bound-number-of-sp}, together with $\varepsilon<10^{-26}$, guarantee that there is $\mathcal{Q}_{good}\subset (\mathcal{Q}\setminus\mathcal{Q}_{sp})$ with 
    \begin{equation}\label{eq:lower-bound-mathcalQ-good}
            |\mathcal{Q}_{good}|>10^{-23}\varepsilon^{2}\ell(Q)^{2} 
    \end{equation}
    such that 
    $$ \|u\|_{\ell^{\infty}(Q')}\leq 1
    $$
    for each $Q'\in \mathcal{Q}_{good}$. 
    
    We call a tilted square $Q'\subset Q$ ``\emph{tamed}'' if the following holds:
    \begin{enumerate}
        \item the center of $Q'$ is in $\frac{1}{50}Q$,
        \item $Q'\supseteq Q''$ for some $Q''\in \mathcal{Q}_{good}$,
        \item $Q'\subset Q'''\subset Q$ implies $\Theta$ is $\varepsilon$-sparse in $Q'''$,
        \item $\|u\|_{\ell^{\infty}(Q')}\leq (\ell(Q')\Bar{V})^{\alpha \ell(Q')}$.
    \end{enumerate}
    Let $\mathcal{Q}_{ta}$ be the set of tamed squares. Then $\mathcal{Q}_{good}\subset \mathcal{Q}_{ta}$. We call $Q'\in \mathcal{Q}_{ta}$ maximal if any $Q''\in \mathcal{Q}_{ta}$ with $Q''\supseteq Q'$ implies $Q''=Q'$.
    \begin{cla}\label{cla:not-extensible}
    Suppose maximal $Q' \in \mathcal{Q}_{ta}$ with $\ell(Q')\leq \frac{1}{24}\ell(Q)$. Then $|u|>1$ on at least a $\varepsilon'$ fraction of $4Q'\setminus \Theta$. 
    \end{cla}
    \begin{proof}
        Since $Q'$'s center is in $\frac{1}{50}Q$, we have $\ell(Q')\leq \frac{1}{24}\ell(Q)$ implies $4Q' \subset Q$. Assume $|u|\leq 1$ on a $1-\varepsilon'$ fraction of $4Q'\setminus \Theta$. Since $Q'\supseteq Q''$ for some $Q''\in \mathcal{Q}$, by Claim \ref{cla:containing-mathcalQ}, $\mathcal{E}^{\alpha',\varepsilon'}_{ex}(2Q',\Theta)$ holds. Moreover,
        $Q'$ containing some $Q''\in \mathcal{Q}_{good}$ implies $\Theta$ is $\varepsilon'$-sparse in $4 Q'$ and thus $\mathcal{E}^{\alpha',\varepsilon'}_{ex}(2Q',\Theta)$ implies $$\|u\|_{\ell^{\infty}(2Q')}\leq (2\ell(Q')\Bar{V})^{2 \alpha' \ell(Q')}(1+ \|u\|_{\ell^{\infty}(Q')}) \leq (\ell(2Q')\Bar{V})^{\alpha \ell(2Q')},$$ as long as $\alpha>10\alpha'$. Thus $2Q'$ is also tamed and this contradicts with the maximality of $Q'$.
    \end{proof}
    Write $\mathcal{Q}_{good}=\{Q^{(i)}:1\leq i\leq K_{2}\}$ and by \eqref{eq:lower-bound-mathcalQ-good},
    \begin{equation}\label{eq:bounding-K_2}
        K_{2}>10^{-23}\varepsilon^{2}\ell(Q)^{2}. 
    \end{equation}
    For each $1\leq i\leq K_{2}$, pick a maximal $Q^{(i)}_{max}\in \mathcal{Q}_{ta}$ with $Q^{(i)}\subset Q^{(i)}_{max}$. Assume 
    \begin{equation}\label{eq:length-larger-than-1/24}
        \ell(Q^{(i_{0})}_{max})> \frac{1}{24}\ell(Q) 
    \end{equation}
    for some $1\leq i_{0}\leq K_{2}$. By definition of $\mathcal{Q}_{ta}$, the center of $Q^{(i_{0})}_{max}$ is in $\frac{1}{50}Q$ and \eqref{eq:length-larger-than-1/24} implies $\frac{1}{100}Q\subset Q^{(i_{0})}_{max}$. Hence 
    $$ \|u\|_{\ell^{\infty}(\frac{1}{100}Q)}\leq \|u\|_{\ell^{\infty}(Q^{(i_{0})}_{max})}\leq (\ell(Q^{(i_{0})}_{max})\Bar{V})^{\alpha \ell(Q^{(i_{0})}_{max})}\leq (\ell(Q)\Bar{V})^{\alpha \ell(Q)}
    $$
    and our conclusion \eqref{eq:final-goal} follows. 
    
    Now we assume $\ell(Q^{(i)}_{max})\leq \frac{1}{24}\ell(Q)$ for each $1\leq i\leq K_{2}$ and we will arrive at a contradiction. By Vitalli covering theorem, there is $J''\subset \{1,\cdots,K_{2}\}$ such that
    $$ 4Q^{(i_{1})}_{max}\cap 4Q^{(i_{2})}_{max}=\emptyset
    $$
    for $i_{1}\not=i_{2}\in J''$ and
    \begin{align}
    \begin{split}\label{eq:lower-bound-J''}
             &\sum_{i\in J''} |4Q^{(i)}_{max}|\\
        \geq &\frac{1}{100}|\bigcup \{4Q^{(i)}_{max}:1\leq i\leq K_{2}\}|\\
        \geq &\frac{1}{100}|\bigcup \{Q^{(i)}:1\leq i\leq K_{2}\}|.
    \end{split}
    \end{align}
    By Claim \ref{cla:not-extensible}, for each $1\leq i\leq K_{2}$, $|u|>1$ on a $\varepsilon'$ fraction of $4Q^{(i)}_{max}\setminus \Theta$.
    Thus
    \begin{align}
        |\{|u|>1\}\setminus \Theta|\geq &\varepsilon' \sum_{i\in J''} |4Q^{(i)}_{max}\setminus \Theta|\\ 
        \geq &\frac{1}{2}\varepsilon' \sum_{i\in J''} |4Q^{(i)}_{max}| \label{eq:lower-bond-bad points0}\\
        \geq &\frac{1}{200}\varepsilon'|\bigcup \{Q^{(i)}:1\leq i\leq K_{2}\}| \label{eq:lower-bond-bad points1}\\ 
        \geq &\frac{1}{200}\varepsilon' K_{2} (4\varepsilon)^{-2} \label{eq:lower-bond-bad points2}\\ 
        \geq &10^{-30} \varepsilon' \ell(Q)^{2}.\label{eq:lower-bond-bad points3}
    \end{align}
    Here, \eqref{eq:lower-bond-bad points0} is because $\Theta$ is $\varepsilon$-sparse in $4Q^{(i)}_{max}$; \eqref{eq:lower-bond-bad points1} is due to \eqref{eq:lower-bound-J''}; \eqref{eq:lower-bond-bad points2} is due to \eqref{eq:length-of-s_1-t_1-s_2-t_2} and \eqref{eq:disjoint-s_1-t_1-s_2-t_2}; \eqref{eq:lower-bond-bad points3} is due to \eqref{eq:bounding-K_2}. However,  by \eqref{eq:bound-epsilon_1},
    \eqref{eq:lower-bond-bad points3} contradicts with $|\{|u|>1\}\setminus \Theta|\leq \varepsilon^{3}\ell(Q)^{2}$ in \eqref{eq:unique-conti}.
\end{proof}
\subsection*{Acknowledgement}
The author thanks Professor Jian Ding for several stimulating discussions and helpful comments. The author also thanks Professor Charles Smart for helpful comments (communicated via Professor Jian Ding). 

\bibliographystyle{style}
\bibliography{bib}

\end{document}